\documentclass[12pt,reqno]{amsart}
\usepackage[margin=1in]{geometry}
\usepackage{amsmath,amssymb,amsthm,graphicx,amsxtra, setspace}
\usepackage[utf8]{inputenc}
\usepackage{mathrsfs}
\usepackage{hyperref}
\usepackage{upgreek}
\usepackage{mathtools}
\usepackage[dvipsnames]{xcolor}
\usepackage[mathcal]{euscript}
\usepackage{amsmath,units}
\usepackage{pgfplots}
\usepackage{tikz}
\usepackage{verbatim}
\allowdisplaybreaks
\usepackage{dsfont}
\usepackage{fontenc}
\usepackage{textcomp}
\usepackage{marvosym}
\usepackage{eurosym}
\usepackage{upgreek}
\usepackage[pagewise]{lineno}
\usepackage{amsmath}
\numberwithin{equation}{section}
\usepackage{amsmath,amssymb,amsthm}
\usepackage{xcolor}
\usepackage{bm}
\DeclareMathAlphabet{\mathpzc}{OT1}{pzc}{m}{it}

\usepackage[cyr]{aeguill}

\allowdisplaybreaks

\colorlet{darkblue}{blue!50!black}

\hypersetup{
	colorlinks,%
	citecolor=blue,%
	filecolor=red,%
	linkcolor=red,%
	urlcolor=blue,%
	pdfnewwindow=true,%
	pdfstartview={FitH}
}

\newtheorem{theorem}{Theorem}[section]
\newtheorem{lemma}[theorem]{Lemma}

\theoremstyle{definition}
\newtheorem{remark}{Remark}[section]

\let\originalleft\left
\let\originalright\right
\renewcommand{\left}{\mathopen{}\mathclose\bgroup\originalleft}
\renewcommand{\right}{\aftergroup\egroup\originalright}

\newcommand{\abs}[1]{\left\vert#1\right\vert}

\newcommand{\Addresses}{{
		\footnote{
			
			\noindent \textsuperscript{1,2}Department of Applied Mathematics and Scientific Computing, Indian Institute of Technology Roorkee-IIT Roorkee,
			Haridwar Highway, Roorkee, Uttarakhand 247667, INDIA.
			
				\textsuperscript{3}Department of Mathematics, Indian Institute of Technology Roorkee-IIT Roorkee,
			Haridwar Highway, Roorkee, Uttarakhand 247667, INDIA.
				\par\nopagebreak
			\noindent  \textit{e-mail:}  \texttt{Om Tripathi: om\_t@amsc.iitr.ac.in.}
			
					\textit{e-mail:} \texttt{Sourav Kumar Sasmal: sourav.sasmal@amsc.iitr.ac.in.}

		\textit{e-mail:}  	\texttt{Manil T. Mohan: maniltmohan@ma.iitr.ac.in, maniltmohan@gmail.com.}

		\noindent \textsuperscript{*}Corresponding author.

}}}

\begin{document}

		\title[Global existence and uniqueness in a density suppressed motility system]{Global existence and uniqueness of solutions in a chemotaxis system with density-suppressed motility, a $\theta-$logistic source, and indirect signal production under a strong Allee effect 
			\Addresses}
		\author[O. Tripathi, S. K.  Sasmal  and M. T. Mohan]{Om Tripathi\textsuperscript{1}, Sourav Kumar Sasmal\textsuperscript{2*},   and Manil T. Mohan\textsuperscript{3}}

			\maketitle
	

	\begin{abstract}
		We consider the following chemotaxis-growth system, featuring density-suppressed motility and a $\theta$-logistic growth term that incorporates a strong Allee effect:
		\begin{equation*}
			\left\{
			\begin{aligned}
				u_t &= \Delta(\gamma(v)u) + \mu u(1-u^\theta)(u-A), & x \in \Omega, \quad t > 0, \\
				v_t &= \Delta v - v+ w^\beta, & x \in \Omega,\quad t > 0, \\
				w_t &= -\delta w +u, & x \in \Omega, \quad t > 0, 
			\end{aligned}
			\right.
		\end{equation*}
	subject to homogeneous Neumann boundary conditions in a bounded domain $\Omega\subset \mathbb{R}^d (d\geq 2)$ with smooth boundary.  Here $\mu\in \mathbb{R} , \delta,\beta>0, \theta\geq1$, and the positive motility function $\gamma(v) \in C^3([0,\infty))$ fulfills $\gamma'(v) \leq 0$ for all $v\geq 0$. The main objective of this paper is to establish the global existence and boundedness of classical solutions to the proposed problem. Our analysis combines the Schauder fixed point theorem for proving local-in-time existence with the extensibility criterion, fundamental energy estimates, $L^p-$bounds, the Gagliardo-Nirenberg interpolation inequality, Young's inequality, semigroup estimates, and a Moser-type iteration scheme to derive uniform-in-time $L^{\infty}-$bounds, thereby ensuring global bounded classical solutions.  More precisely, we establish that the chemotaxis-growth system admits a unique globally bounded classical solution whenever $\beta < \frac{2(\theta+2)}{d}$. Furthermore, if the logarithmic derivative of the motility function is assumed to be uniformly bounded on $[0,\infty)$, that is, $\frac{\gamma'(v)}{\gamma(v)} \in L^{\infty}([0,\infty))$, then the above restriction on $\beta$ can be weakened to $\beta < \frac{d(\theta+1)+2(\theta+2)}{2d}$, provided that $\theta > \frac{4-d}{d-2}$.
		\vspace{0.4cm}
		
		\vspace{0.2cm}
		
		\noindent 	\textbf{MSC Classification 2020:} 35B45, 35K57, 35Q92, 35A01, 92C17. 
			
	\end{abstract}

	
	\section{Introduction}\setcounter{equation}{0}
    Chemotaxis refers to the directed migration of biological entities toward or away from regions of varying chemical concentration. This phenomenon plays a significant role in various biological processes, including cell migration, aggregation, embryonic development, immune responses, and bacterial pattern formation. To mathematically describe such behavior, Keller and Segel proposed a pioneering class of chemotaxis models in their classical works \cite{KSJTB1,KSJTB2}. Since then, the Keller-Segel model has become one of the most extensively studied systems in mathematical biology and has attracted considerable attention over the past several decades. In a bounded domain $\Omega\subset\mathbb{R}^{d}$ $(d\geq2)$ with smooth boundary $\partial\Omega$, let $u=u(x,t)$ and $v=v(x,t),$ $(x,t)\in\Omega\times[0,T]$ denote the density of the cell population and the concentration of the chemoattractant, respectively. The classical Keller-Segel chemotaxis system with signal production, describing the movement of bacteria toward higher concentrations of self-produced chemical signals, is formulated as follows (\cite{MR4188348}):
     \begin{equation}\label{keller segel}
     	\left\{
     	\begin{aligned}
     		u_t &= \nabla\cdot\bigl(\varphi(u,v)\nabla u
     		-\chi(u,v) u\nabla v\bigr)
     		+f(u,v),\quad &&x\in\Omega,\; t>0,\\
     		v_t &= D\Delta v+g(u,v),\quad &&x\in\Omega,\; t>0,
     	\end{aligned}
     	\right. 
     \end{equation}
 where $D>0$ is a constant. The function $\varphi(u,v)$ denotes the diffusivity of the cells, while $\chi(u,v)$ represents their chemotactic sensitivity. The cross-diffusion term $-\chi(u,v)u\nabla v$ represents the chemotactic migration of bacteria in response to chemical concentration gradients, directing their movement toward regions of higher chemoattractant concentration. The function $f(u,v)$ models the net effect of bacterial proliferation and mortality, while $g(u,v)$ characterizes the generation and consumption of the chemical substance.
     
    A comprehensive treatment of various chemotaxis processes can be found in \cite{chemo1,chemo2,chemo3,chemo4}, and the references therein. Among the various Keller-Segel chemotaxis models, one of the most extensively studied  classical parabolic-parabolic Keller-Segel system is
     \begin{equation}\label{parabolic KS}
     	\begin{cases}
     		u_t=\Delta u-\chi\nabla\cdot(u\nabla v), \quad & x\in\Omega,\; t>0,\\
     		v_t=\Delta v-v+u, \quad & x\in\Omega,\; t>0,
     	\end{cases}
     \end{equation}
    where \(\chi>0\) is a positive fixed constant representing the chemotactic sensitivity. A remarkable feature of this system \eqref{parabolic KS} is that the qualitative behavior of its solutions depends strongly on the spatial dimension. In the one-dimensional setting, Osaki and Yagi \cite{KSstudy1} established the global boundedness of solutions. In two-dimensional domains, Nagai \textit{et al.} \cite{KSstudy2} proved the existence of globally bounded solutions provided that the initial cell mass satisfies $\|u_0\|_{L^1(\Omega)}<\frac{4\pi}{\chi}$. Subsequently, this model has attracted considerable attention, leading to extensive studies on the existence, boundedness, and blow-up behavior of solutions. For comprehensive surveys and further developments, one can refer to Fujie and Senba \cite{DSM4}, Horstmann \cite{hortsman1,Hortsman2}. Moreover, Horstmann and Wang \cite{Horstman3} showed that solutions may blow up in finite time whenever the initial cell mass satisfies $\|u_0\|_{L^1(\Omega)}>\frac{4\pi}{\chi}$. 
     
\par\vspace{0.2em}


Several studies in mathematical biology have investigated chemotaxis models incorporating local sensing mechanisms \cite{PRL2,Science,PRL1}. In this mechanism, cells decrease motility at high concentrations of signaling molecules, which leads to a self-trapping effect that can induce aggregation and spatial patterns. Unlike the system \eqref{parabolic KS}, in the model \eqref{density suppressed motility} given below, the cellular diffusion rate and the chemotactic sensitivity are assumed to be nonlinear functions of the signal concentration in order to incorporate the effect of density-suppressed motility observed in bacterial populations. A typical model exhibiting this behavior is given by 
\begin{equation}\label{density suppressed motility}
	\begin{cases}
\begin{aligned}
u_t&=\nabla\cdot	(\gamma(v)u)+\nabla\cdot(\gamma'(v)u\nabla u),\quad &&x\in\Omega,\; t>0,\\
	v_t&=\Delta v-v+u,\quad &&x\in\Omega,\; t>0,
\end{aligned}
\end{cases}
\end{equation}
where $u(x,t)$ and $v(x,t)$ denote the density of the bacterium \emph{Escherichia coli (E. coli)} and the concentration of the signaling molecule acyl-homoserine lactone (AHL), respectively. Here, $\gamma:[0,\infty)\to(0,\infty)$ denotes the cell motility function, which is assumed to be smooth and nonincreasing with respect to the chemical concentration $v$. This assumption is biologically motivated by the observation that cell motility decreases as the chemical concentration increases.

The density-suppressed motility system has been extensively studied under various structural assumptions on the motility function $\gamma$. In particular, under the assumptions that $\gamma$ is positive and uniformly bounded from above and below, with a bounded derivative, Tao and Winkler \cite{DSM1} established the existence of globally bounded classical solutions in bounded convex domains of two dimensions, as well as global weak solutions in higher-dimensional settings. Subsequently, Xiao and Jiang \cite{MF1} extended these results by relaxing the boundedness requirement on $\gamma'$ and established the global boundedness of classical solutions in bounded domains of arbitrary spatial dimension.

Further progress was made by Fujie and Senba \cite{MF3}, who established the global existence of bounded classical solutions under suitable smoothness and monotonicity assumptions on $\gamma$, together with the condition that $\gamma(s)\to0$ as $s\to\infty$. Specific choices of the motility function have also been the subject of extensive investigation. In particular, for the algebraic motility function $\gamma(s)=cs^{-k}$ with $c,k>0$, the existence of global bounded classical solutions was established in \cite{MF1}, provided that $c$ is sufficiently small. In the prototype power-law case $\gamma(s)=s^{-k}$, Fujie and Senba \cite{MF2} proved the global existence of bounded classical solutions for $k\in\left(0,\frac{2}{d-2}\right), d>2$. Moreover, when the motility function is chosen as $\gamma(s)=e^{-s}$, a critical-mass phenomenon was established in two-dimensional domains \cite{MF3,MF4}.
 Motivated by biological considerations, numerous studies have investigated the system \eqref{density suppressed motility} incorporating a logistic source and nonlinear production. The logistic source is commonly employed to model the proliferation and mortality of bacterial populations. Accordingly, the corresponding model is formulated as follows:
		\begin{equation}\label{non linear v}
			\left\{
			\begin{aligned}
				u_t &= \nabla\cdot	(\gamma(v)u)+\nabla\cdot(\gamma'(v)u\nabla u)+f(u), \quad &&x\in\Omega,\; t>0,\\
				v_t &= \Delta v-v+u^\beta, \quad &&x\in\Omega,\; t>0.
			\end{aligned}
			\right.
		\end{equation}
In the special case $\beta=1$, the system with logistic growth $f(u)=\mu u(1-u)$, $\mu>0$, was studied by Jin \emph{et al.} \cite{SIAM} who proved the existence of globally bounded classical solutions in two-dimensional domains under the assumptions $$\gamma\in C^3([0,\infty)),\  \gamma>0,\  \gamma'\le0,\  \lim_{s\to\infty}\gamma(s)=0\ \text{ and } \lim_{s\to\infty}\frac{\gamma'(s)}{\gamma(s)}\ \text{ exists.}$$ Global classical solutions for dimensions $d\ge 3$ were obtained by Wang and Wang \cite{logistic1} under the boundedness of $\gamma'$ and the sufficient largeness of the parameter $\mu$. Subsequently, Song \emph{et al.} \cite{IMAA24} considered the system \eqref{non linear v} with $\beta=1$ and a logistic growth term involving a strong Allee effect which is modeled as
\begin{equation}\label{Strong Allee}
\left\{
\begin{aligned}
	u_t &= \nabla\cdot	(\gamma(v)u)+\nabla\cdot(\gamma'(v)u\nabla u)+u(1-u)(u-A), \quad &&x\in\Omega,\; t>0,\\
	v_t &= \Delta v-v+u, \quad &&x\in\Omega,\; t>0,
\end{aligned}
\right.
\end{equation}
 where $A\in(0,1)$ denotes the Allee threshold parameter. They proved the global existence and boundedness of classical solutions in spatial dimensions \(d\leq 3\). The strong Allee effect refers to the phenomenon in which a population can persist and grow only if its density exceeds a critical value, known as the Allee threshold. Populations with densities below this threshold experience negative growth and are ultimately driven to extinction, whereas populations with densities above the threshold can survive and may converge to a positive equilibrium. The strong Allee effect occurs when low population density impairs essential biological processes, such as mate finding, cooperative behaviour, and predator avoidance \cite{Allee3,Allee2}.
At bacterial level, to describe the interplay between chemotaxis and strong Allee effects, Mimura and Tsujikawa \cite{Allee5} proposed the following chemotaxis system:
\[
\left\{
\begin{aligned}
	&u_t = \nabla\cdot(d_1\nabla u-\chi u\nabla B)+u(1-u)(u-A),
	\quad &&x\in\Omega,\; t>0,\\
	&(d_1\nabla u-\chi u\nabla B)\cdot\nu  =0,
	\quad &&x\in\partial\Omega,\; t>0,\\
	&u(x,0)=u_0(x)\ge0, \quad &&x\in\Omega,
\end{aligned}
\right.
\]
where \(B\) represents a given environmental stimulus responsible for directed motion, and $\nu$ denotes the outward unit normal vector on $\partial\Omega$. The study established the existence of multiple spike-type localized solutions under strong advection and analyzed their stability by examining the spectrum of the corresponding linearized eigenvalue problem.
\par\vspace{0.2em}


\indent In order to incorporate stronger density-dependent growth effects, several authors have considered chemotaxis systems with generalized logistic source terms in \eqref{non linear v}, where $f(u)=ru-\mu u^{\alpha}$ with $r\in\mathbb{R}$ and $\mu>0$. Such nonlinear growth terms provide a stronger damping effect that suppresses excessive cell aggregation, thereby playing an important role in proving the global existence and uniform boundedness of solutions. Under suitable assumptions on the motility function, $\gamma$, Lv and Wang \cite{GLS1} proved the existence of globally bounded classical solutions for $\alpha>\max\left\{\frac{d+2}{2},\,2\right\}$ for the system \eqref{non linear v}. Subsequently, Tao and Fang \cite{ZAMP2022} investigated the case $\beta>0$ and established the global boundedness of classical solutions whenever $\beta<\frac{2\alpha}{d+2}$. In addition, several other variants of chemotaxis systems with signal-dependent motility, including models incorporating signal absorption, have also been extensively studied \cite{signalabsorption1, signalabsorption2}.

Conventionally, chemotaxis models assume that the chemoattractant is produced directly by the cellular population, as in the systems \eqref{keller segel}-\eqref{non linear v}. However, this assumption is not suitable for certain biological processes in which the chemical signal is generated indirectly through intermediate substances. To capture such mechanisms, a variety of chemotaxis models with indirect signal production have been proposed and extensively investigated. One important application of these models arises in the mathematical description of the aggregation and dispersal of Mountain Pine Beetles (MPB) \cite{pinebeetlepaper,MPB}. A representative model is given by
\begin{equation}\label{MPB}
	\left\{
	\begin{aligned}
		u_t &= \Delta u-\nabla\cdot(u\nabla v)+\mu u-\mu u^\alpha,
		\quad &&x\in\Omega,\; t>0,\\
		v_t &= \Delta v-v+w^{\beta},
		\quad &&x\in\Omega,\; t>0,\\
		w_t &= -\delta w+u,
		\quad &&x\in\Omega,\; t>0,
	\end{aligned}
	\right.
\end{equation}
where $\mu,\delta,\beta>0$ and $\alpha>1$. The variables $u$, $w$, and $v$ denote the densities of flying MPBs, nesting MPBs, and the beetle pheromone, respectively. For the case $\alpha=2$ and $\beta=1$, Hu and Tao \cite{MPB1} established the global existence and boundedness of classical solutions in three-dimensional domains. Subsequently, Li and Tao \cite{MPB2} extended this result to arbitrary spatial dimensions by proving the global existence of bounded classical solutions whenever $\alpha>\frac{d}{2}$. Motivated by these developments, several authors have further investigated the system \eqref{MPB} by incorporating signal-dependent (density-suppressed) motility of the form
\begin{equation}\label{MPB modified}
	\left\{
	\begin{aligned}
		u_t &= \nabla\cdot(\gamma(v)u)-\nabla\cdot(\gamma'(v)u\nabla u)+ru-\mu u^\alpha,
		\quad &&x\in\Omega,\; t>0,\\
		v_t &= \Delta v-v+w^{\beta},
		\quad &&x\in\Omega,\; t>0,\\
		w_t &= -\delta w+u,
		\quad &&x\in\Omega,\; t>0.
	\end{aligned}
	\right.
\end{equation}
Recently, Lv and Wang \cite{JAMAAequvi} studied system \eqref{MPB modified} in the case $\beta=1$ and $\alpha>1$. By imposing suitable additional conditions on the motility function $\gamma$, they proved the global existence of bounded classical solutions whenever $\alpha>\frac{3d}{d+2}$.

	\subsection{Problem formulation and main results}Motivated by the ecological significance of the strong Allee effect in \eqref{Strong Allee}, we consider replacing the generalized logistic source term in the system \eqref{MPB modified} with a $\theta$-logistic growth term incorporating a strong Allee effect \cite{thetalogistic}. This formulation captures both nonlinear density-dependent regulation and the existence of a critical population threshold, leading to more realistic population dynamics and a stronger damping effect against excessive aggregation. Accordingly, we study the following chemotaxis system:
	\begin{equation} \label{Main}
		\left\{
		\begin{aligned}
			&u_t = \Delta(\gamma(v)u) + \mu u(1-u^\theta)(u-A),\quad &&x \in \Omega, t > 0, \\
			&v_t = \Delta v - v+ w^\beta,\quad &&x \in \Omega, t > 0, \\
			&w_t = -\delta w +u,\quad &&x \in \Omega, t > 0, \\
			&\frac{\partial u}{\partial \nu} = \frac{\partial v}{\partial \nu} = 0,\quad &&x \in \partial \Omega, \\
			&u(x, 0) = u_0(x), \quad v(x, 0) = v_0(x), \quad w(x, 0) = w_0(x),\quad &&x \in \Omega,
		\end{aligned}
		\right.
	\end{equation}
	in a bounded domain $\Omega\subset\mathbb{R}^d$ $(d\ge2)$ with smooth boundary $\partial\Omega$, assuming that $\partial\Omega$ is of class $C^{2+\alpha}$ for some $0<\alpha<1$. Throughout this paper, we assume that the parameters \ $\mu\in\mathbb{R}$, $\beta,\delta>0$, $A\in(0,1)$, and $\theta\geq1$. In addition, for the existence of local solutions, the initial data $(u_0,v_0,w_0)$ are assumed to satisfy
	\begin{equation}\label{Initial Cond}
		\left\{
		\begin{aligned}
			&u_0 \in C^0(\overline{\Omega}),\
			v_0,w_0 \in C^1(\overline{\Omega}),\\
			&u_0,\, v_0,\, w_0 \ge 0\ \text{and}\ u_0 \not\equiv 0.
		\end{aligned}
		\right.
	\end{equation}
\begin{remark}
	One may relax the conditions of the initial data in the following way also
	\begin{equation*}
		\left\{
		\begin{aligned}
			&u_0 \in C^0(\overline{\Omega}),\
			v_0,w_0 \in W^{1,\infty}({\Omega}),\\
			&u_0,\, v_0,\, w_0 \ge 0\ \text{and}\ u_0 \not\equiv 0.
		\end{aligned}
		\right.
	\end{equation*}
		By the Sobolev embedding $W^{1,\infty}(\Omega)\hookrightarrow C^{0}(\overline{\Omega})$, the above implies $v_0,\, w_0\in C^{0}(\overline{\Omega})$.
\end{remark}
For the uniqueness of local solutions, we assume that the initial data satisfy the following: 
\begin{equation}\label{initial-data-reg}
	\begin{aligned}
		&(u_0,v_0,w_0)\in\left(C^{2+\alpha}(\overline{\Omega})\right)^3\ \text{ for some }\ 0<\alpha<1, \  u_0,v_0\ge0 \text{  with }\ u_0\not\equiv0 \ \text{ and }  \\
		&	\partial_\nu u_0
		=
		\partial_\nu v_0
		=
		0
		\ \text{on }\partial\Omega.
	\end{aligned} 
\end{equation}
Moreover, we assume that the motility function satisfies
		\begin{equation}\label{Hypothesis H1}
			 \gamma(v) \in C^3([0,\infty)), \ \gamma(v) > 0, \  \gamma'(v) < 0 \ \text{ on } \ [0,\infty), \ \text{and }\  \lim_{v\to\infty} \gamma(v) = 0,  
		\end{equation}
		with an extra assumption that
		\begin{equation}\label{Hypothesis H2}
			\lim_{v\to\infty} \frac{\gamma'(v)}{\gamma(v)} \text{ exists}.
		\end{equation}
The primary aim of this work is to determine suitable conditions on the parameters $\theta$ and $\beta$ that ensure the existence and uniqueness of globally bounded classical solutions to system \eqref{Main}. The principal results of this study are summarized below.
\begin{theorem}\label{global bound_H1}
Let $\Omega \subset \mathbb{R}^d$ $(d\geq2)$ be a bounded domain with smooth boundary. Suppose that the initial data satisfy \eqref{Initial Cond}, with $\mu,\beta,\delta>0$ and $\theta\geq1$. If the motility function satisfies \eqref{Hypothesis H1} and
$
\beta<\frac{2(\theta+2)}{d},
$
then problem \eqref{Main} admits a  globally bounded classical solution $(u,v,w)$. Moreover, there exists a constant $C>0$, independent of $t$, such that
\[
\|u(\cdot,t)\|_{L^\infty(\Omega)}
+\|v(\cdot,t)\|_{W^{1,\infty}(\Omega)}
+\|w(\cdot,t)\|_{L^\infty(\Omega)}
\leq C,
\ \text{for all } t>0.
\]
\end{theorem}
\begin{remark}
The global boundedness of classical solutions was established for the direct signal production system with the generalized logistic source term $f(u)=ru-\mu u^\alpha$, under the condition $\beta<\frac{2\alpha}{d+2}$ \cite{ZAMP2022}, which was later improved to $\beta<\frac{2\alpha}{d}$ for the corresponding  indirect signal production system \cite{JAMAA25}. In contrast, Theorem \ref{global bound_H1} proves global boundedness for a wider range of $\beta$, highlighting the stronger damping effect of the $\theta$-logistic source term.
\end{remark}	
Under stronger assumptions on the motility function, we obtain the following result.
\begin{theorem}\label{global bound H1-H2}
	Let $\Omega \subset \mathbb{R}^d$ $(d\geq2)$ be a bounded domain with smooth boundary. Suppose that the initial data satisfy \eqref{Initial Cond}, with $\mu,\beta,\delta>0$ and $\theta\geq1$. If the motility function satisfies both \eqref{Hypothesis H1} and \eqref{Hypothesis H2}, then the condition imposed in Theorem \ref{global bound_H1} for the existence of globally bounded classical solutions can be relaxed to
	$
	\beta<\frac{d(\theta+1)+2(\theta+2)}{2d},
	$
	provided that
	$
	\theta>\frac{4-d}{d-2}.
	$
	\end{theorem}
\begin{remark}
	In system \eqref{Main}, the parameter $\theta\ge1$ determines the strength of the density-dependent damping, while $\beta$ represents the production rate of $v$. Theorems \ref{global bound_H1} and \ref{global bound H1-H2} demonstrate that the balance between these two effects is crucial when the damping induced by $\theta$ is sufficiently strong relative to the production governed by $\beta$,  and the solution component $u$ remains uniformly bounded for all time.
	
\end{remark}
\subsection{Difficulties, approaches and novelties}
In Section \ref{sec:preliminaries}, the condition in Lemma \ref{beta k condition} is repeatedly used in the analysis of the \(v\)-equation; see, for instance, \cite{SIAM, JAMAAequvi, ZAMP2022, JAMAA24, JAMAA25}. Nevertheless, the endpoint case \(m=\infty\) is not explicitly covered by the existing estimates. To address this issue, we extend the estimate from \cite[Lemma 1.3]{Neumann}, stated in Lemma \ref{Neumann estimates}(i), to the endpoint \(p=\infty\). Since the original result does not include this endpoint, we provide a separate proof in Lemma \ref{Neumann est infty case}. This extension allows us to rigorously apply the corresponding estimate in the case \(m=\infty\). Consequently, we obtain a constant \(C>0\) such that
$
\|\nabla v\|_{L^{\infty}(\Omega)}\leq C
$
whenever
$
\beta<\frac{k}{d}
\ \text{with}\
\beta\leq k
$ in Lemma \ref{beta k condition}.

To the best of our knowledge, the regularity assumptions imposed on the initial data for establishing the existence and uniqueness of local solutions to \eqref{Main} have not been explicitly characterized in the existing literature \cite{SIAM, JAMAAequvi, IMAA24, JAMAA25} on chemotaxis systems with density-suppressed motility and indirect signal production. The construction of local classical solutions is, however, nontrivial, as the initial data prescribed in \eqref{Initial Cond} do not possess sufficient regularity to permit a direct application of the abstract parabolic theory developed in \cite{Ladyzhynska} within a Schauder fixed-point framework. To address this issue, we approximate the given initial data \eqref{Initial Cond} by a sequence of smoother functions that satisfy the appropriate compatibility conditions such that $\{(u_0^n,v_0^n,w_0^n)\}_{n\in\mathbb{N}}
\in
\bigl(C^{2+\alpha}(\overline{\Omega})\bigr)^3,
\ 0<\alpha<1,\ \text{with}\ \partial_\nu u_0^n
=
\partial_\nu v_0^n
=0
\ \text{on }\partial\Omega$. Using the closed, bounded, and convex set in the underlying Banach space introduced for the fixed-point argument in Section \ref{sec:local}, we establish the necessary parabolic Hölder estimates for the local solutions associated with the regularized initial data. These estimates allow us to construct a sequence of local approximating solutions to \eqref{Main}. We then exploit the compactness of the relevant parabolic Hölder embeddings to extract a subsequence that converges strongly in the required function spaces. This convergence enables us to pass to the limit in the approximating problem and thereby obtain the desired local solution of \eqref{Main}.
 
The construction of local classical solutions to \eqref{Main} requires some additional care, mainly due to the possibility that the motility function \(\gamma\) may vanish, thereby causing degeneracy in the diffusion term of the \(u\)-equation. Consequently, the standard theory for uniformly parabolic equations is not directly applicable. To handle this degeneracy, we introduce an appropriate cut-off and replace \(\gamma\) with a truncated motility function $\widehat{\gamma}$ in \eqref{gamma hat} that remains uniformly positive.

Another difficulty arises in establishing uniqueness of local solutions for
the initial data specified in \eqref{Initial Cond}. The uniqueness argument
requires sufficient regularity of $(u,v,w)$ on
$\Omega\times[0,T_{\max})$, in particular the regularity stated in
\eqref{eqn-regularity} and \eqref{regularity assumptions}. Under the weaker
assumptions imposed on the initial data in \eqref{Initial Cond}, the
regularity, 
$
u\in L^\infty\big([0,T_{\max});W^{1,\infty}(\Omega)\big),
$
is not immediately available. Therefore, we establish uniqueness separately
for initial data satisfying the stronger regularity assumptions in
\eqref{initial-data-reg}, for which the corresponding solution possesses the
regularity required by Theorem \ref{Uniqueness of local classical solutions}.

The sublinear production term \(w^\beta\) in the \(v\)-equation introduces a further complication, making the analysis more involved than in the linear case \(\beta=1\) considered in \cite{SIAM, IMAA24}. When \(0<\beta<1\), the function \(w\mapsto w^\beta\) is not globally Lipschitz continuous, which creates difficulties in verifying the continuity of the fixed-point operator. This issue is addressed through suitable estimates based on the inequality \eqref{ineq power p 1}. For \(\beta\geq1\), on the other hand, the nonlinear term can be controlled more directly by means of the inequality \eqref{ineq power p 2}. In addition, for \(0<\beta<1\), the uniqueness argument requires the strict positivity of \(w_0\) in order to obtain a positive lower bound for \(w_i\), whereas the weaker assumption \(w_0\geq0\) is sufficient when \(\beta\geq1\). These distinctions motivate treating the existence and uniqueness of local classical solutions separately in the two parameter regimes.

A further significant challenge in establishing the global existence of classical solutions to system \eqref{Main} arises from the density-suppressed motility mechanism under the assumption $\gamma'(v)\leq0$. In contrast to many classical chemotaxis and reaction-diffusion models, the motility function $\gamma(v)$ is not assumed to be bounded away from zero, which may cause the diffusion in the $u$-equation to become degenerate. As a result, several standard techniques commonly used to prove the global existence and boundedness of solutions cannot be applied directly and require substantial modification. So, for this purpose, we have established the higher regularity of $v$ in Lemma \ref{Lq norm of v and grad v}, for attaining the lower bound on $\gamma$ under the condition $\beta < \frac{2(\theta+2)}{d}$.
For lower-dimensional cases under assumptions \eqref{Hypothesis H1} and \eqref{Hypothesis H2}, the analysis is comparatively simpler. 

In Section \ref{sec:global1}, for \(d\geq2\), we develop the analysis under the sole structural assumption \eqref{Hypothesis H1} on the motility function. We first establish a uniform-in-time upper bound for \(v\). In particular, Lemma \ref{beta cond for grad v} yields a \(W^{1,q}\)-estimate for \(v\) with \(q>d\), which provides the regularity needed to derive an \(L^p\)-bound for \(u\), with \(p>2\) sufficiently large, as shown in Lemma \ref{Lp estimate of u 1}. Subsequently, by combining the Gagliardo--Nirenberg interpolation inequality with the Moser iteration procedure in Lemma \ref{Linfty bound of u}, we upgrade this \(L^p\)-estimate to a uniform \(L^\infty\)-bound for \(u\). This uniform-in-time bound on \(u\) is a key ingredient in establishing the global existence of classical solutions stated in Theorem \ref{global bound_H1}.

Moreover, when the motility function satisfies the additional assumption \eqref{Hypothesis H2}, the approach developed in Section \ref{sec:global2} avoids the need for higher-order regularity estimates for \(v\). Instead, we establish a coupled differential inequality involving \(\int_{\Omega}u^p\), \(\int_{\Omega}|\nabla v|^{2q}\), and \(\int_{\Omega}w^p\) in Lemma \ref{lemma 4.8}, valid for all \(t\in(0,T_{\max})\). The main challenge is the presence of nonlinear terms of the form \(\int_{\Omega}|\nabla v|^{\chi_i}\), where the exponents \(\chi_i\) depend on \(p\), \(q\), \(\beta\), and \(\theta\). To control these terms, we choose \(p\) and \(q\) appropriately according to the conditions in Lemma \ref{cond on p and q}. The Gagliardo-Nirenberg interpolation inequality then allows us to estimate each \(\int_{\Omega}|\nabla v|^{\chi_i}\) in terms of the quantities appearing in the differential inequality, which ultimately yields the required \(L^p\)-bound for \(u\). Once this estimate is established, the conditions on \(q\) given in Lemma \ref{cond on p and q} ensure that \(q>\frac{d}{2}\) for some $p>p*$ defined in \eqref{p* define}. Consequently, the Sobolev embedding yields the uniform estimate \(\|v\|_{L^{\infty}(\Omega)}\leq C\). The remainder of the argument can then be carried out in a manner analogous to the proof of Theorem \ref{global bound_H1}, ultimately leading to the global boundedness of the classical solutions.

It is worth highlighting that our analysis is performed in the presence of a \(\theta\)-logistic growth term incorporating a strong Allee effect. The resulting nonlinear damping mechanism is instrumental in deriving the higher-order estimates needed for the global existence results. To the best of our knowledge, global solvability for chemotaxis systems that simultaneously involve density-suppressed motility, indirectly produced nonlinear signals, and \(\theta\)-logistic growth with a strong Allee effect has not been studied in the existing literature. Our results therefore extend the current theory by identifying admissible ranges of the parameter \(\beta\) for which the damping generated by the \(\theta\)-logistic term is sufficiently strong to compensate for the destabilizing effects arising from density-suppressed motility and nonlinear signal production.
 \subsection{Organization of the paper}
 The remainder of this paper is organized as follows. Section \ref{sec:preliminaries} presents the basic preliminaries, including the functional spaces, auxiliary lemmas, and technical tools that will be used throughout the paper. In Section \ref{sec:local}, we establish the local existence and uniqueness of classical solutions. Section \ref{sec:global1} is devoted to the global boundedness of classical solutions under assumption \eqref{Hypothesis H1} on the motility function. To avoid the additional restrictions imposed in earlier works \cite{SIAM,JAMAAequvi,IMAA24}, we first derive an appropriate $L^{\theta+2}$ estimate for $w$ in Lemma \ref{L1 estimate of w theta}. This estimate is then combined with Neumann heat semigroup estimates to establish $L^q$-bounds for $\nabla v$ and the uniform-in-time boundedness of $v$ in Lemma \ref{Lq norm of v and grad v}. The resulting regularity of $v$ enables us to derive $L^p$-estimates for $u$ in Lemma \ref{Lp estimate of u 1}. Finally, these estimates are incorporated into the Alikakos--Moser iteration scheme to obtain the uniform-in-time boundedness of $u$, thereby completing the proof of Theorem \ref{global bound_H1}.
 
 In Section \ref{sec:global2}, we investigate the global boundedness of classical solutions under the combined assumptions \eqref{Hypothesis H1} and \eqref{Hypothesis H2}. We begin by establishing a series of a priori estimates for the solution, as presented in Lemmas \ref{estimate 4.1}-\ref{lemma 4.8}. These estimates form the basis for deriving an $L^p$-estimate of $u$ in Lemma \ref{final lemma for theorem 2}, where a suitable admissible range of the parameter $\beta$ is identified through Lemma \ref{cond on p and q}. With these estimates in hand, the proof of Theorem \ref{global bound H1-H2} follows by adapting the arguments developed in Section \ref{sec:global1} for uniform-in-time bound of $u$.

\section{Basic preliminaries}\label{sec:preliminaries}
In this section, we introduce the necessary preliminaries, including the functional spaces and auxiliary results that will be used in the proofs of the main theorems.
	
	\subsection{Important inequalities} The following inequalities will be used frequently in the paper. 
	\begin{lemma}[{\cite[Lemma 3.4]{odelemma}}]\label{ode comparison}
		Let $T>0$ and let $x:[0,T)\to[0,\infty)$ be absolutely continuous and satisfy
		\[
		x'(t)+a\,x(t)\le \rho(t),\ \text{ for a.e.}\ t\in(0,T),
		\]
		for some $a>0$, where $\rho\ge0$ and there exist $\tau>0$ and $b>0$ such that
		\[\int_t^{t+\tau} \rho(s)\,ds \le b,\ \text{ for all}\ t\in[0,T-\tau].
		\]
		Then for every $t\in[0,T)$, we have
		\[
		x(t)\le \max\Bigl\{\,x(0)+b,\ \frac{b}{a\tau}+2b\Bigr\}.
		\]
	\end{lemma}
	\begin{lemma}[{\cite[Lemma 2.3]{JAMAA24}}]\label{power inequality}
		Let $A>0$ and $B>0$. Then
		\begin{equation}\label{ineq power p 1}
			\left|A^p - B^p\right|
			\le \left(\frac{B}{2}\right)^{p-1} \left|A-B\right|,
			\ \text{for all} \ 0<p<1,
		\end{equation}
		and
		\begin{equation}\label{ineq power p 2}
			\left|A^p - B^p\right|
			\le p(A+B)^{p-1} \left|A-B\right|,
			\ \text{for all}\ p \ge 1.
		\end{equation}
	\end{lemma}
	
	\begin{lemma}[{\cite[p.~11]{GNI}}]\label{GNI}
		Let $\Omega \subset \mathbb{R}^{d}$ be a bounded domain with smooth boundary. 
		Assume that the integers $q,r,m$ satisfy
		\[
		1\le q,r\le \infty,
		\ m>0,
		\]
		and let $p\in \mathbb{R}_{+}$ be such that
		\[
		\frac{1}{p}
		=
		\vartheta\left(\frac{1}{q}-\frac{m}{d}\right)
		+
		\frac{1-\vartheta}{r},
		\]
		where $0\le \vartheta \le 1$. Then for each
		\[
		\varphi \in W^{m,q}(\Omega)\cap L^{r}(\Omega),
		\]
		there exist positive constants $C_{1}$ and $C_{2}$ depending solely on $\Omega$, $q$, $r$, and $m$ such that
		\[
		\|\varphi\|_{L^{p}}
		\le
		C_{1}
		\|D^{m}\varphi\|_{L^{q}}^{\vartheta}
		\|\varphi\|_{L^{r}}^{1-\vartheta}
		+
		C_{2}\|\varphi\|_{L^{r}}.
		\]
	\end{lemma}

	\begin{lemma}
		[{\cite[Lemma 4.2]{MizoguchiSouplet2001}}]\label{control boun integral}
		Let $\Omega$ be a bounded domain in $\mathbb{R}^d$ with smooth boundary. If $f \in C^2(\overline{\Omega})$ satisfies
		\[
		\frac{\partial f}{\partial \nu} = 0,
		\]
		then
		\begin{equation}\label{normal derivative bound}
			\frac{\partial |\nabla f|^2}{\partial \nu} \leq c_{\Omega} |\nabla f|^2,
		\end{equation}
		where $c_{\Omega} > 0$ is a constant depending only on the curvature of $\partial \Omega$.
	\end{lemma}

	The next lemma is essentially acquired by Hajaiej et.al \cite{Hajajej}.
	\begin{lemma}[The fractional Gagliardo--Nirenberg inequality]\label{fractional gagliardo}
		Let $j \in \mathbb{N}$, $\varrho \geq 1$ and $q \geq 1$. Assume that $p > 0$ and $a \in (0,1)$ satisfy
		\begin{equation*}
			\frac{1}{2} - \frac{p}{j} = (1-a)\frac{q}{\varrho} + a\left(\frac{1}{2} - \frac{1}{j}\right)
			\quad \text{and} \quad p \leq a.
		\end{equation*}
		Then, there exist $c_0,\ c_0' > 0$ such that for all $f \in W^{1,2}(\Omega) \cap L^{\frac{k}{q}}(\Omega)$,
		\begin{equation}\label{gagliardo ineq}
			\|f\|_{W^{p,2}(\Omega)} 
			\leq c_0 \|\nabla f\|_{L^2(\Omega)}^{a} \|f\|_{L^{\frac{\varrho}{q}}(\Omega)}^{1-a}
			+ c_0' \|f\|_{L^{\frac{\varrho}{q}}(\Omega)}.
		\end{equation}
	\end{lemma}

	Unless otherwise specified, throughout the paper, we use $C$ to denote a generic positive constant whose value may change from line to line.
	
	\subsection{Functional spaces}\label{subsec:Functional space}
	In this subsection, we introduce some necessary function spaces that will be used throughout the paper.
    For $0<l<1$, let $C^{l,\frac{l}{2}}(\overline{\Omega}\times[0,T])$ denote the parabolic Hölder space (\cite[Section 5.1]{evans}, \cite[Chapter I]{Ladyzhynska}), whose elements are $l$-Hölder continuous with respect to the spatial variables and $\frac{l}{2}$-Hölder continuous with respect to the time variable. Its norm is defined by
	\begin{align}\label{eqn-parabolic-norm}
	\|u\|_{C^{l,\frac{l}{2}}(\overline{\Omega}\times[0,T])}
	=\|u\|_{L^\infty(\Omega\times(0,T))}
	+[u]_{x;l}
	+[u]_{t;\frac{l}{2}},
	\end{align}
	where $[\,\cdot\,]_{x;l}$ and $[\,\cdot\,]_{t;\frac{l}{2}}$ denote the spatial and temporal Hölder seminorms, respectively. For $0<l<1$, we define
	\[
	[u]_{x;l}
	:=
	\sup_{\substack{x,y\in\overline{\Omega},\,x\neq y\\ t\in[0,T]}}
	\frac{|u(x,t)-u(y,t)|}{|x-y|^{l}},
	\]
	and
	\[
	[u]_{t;l/2}
	:=
	\sup_{\substack{x\in\overline{\Omega}\\ s,t\in[0,T],\,s\neq t}}
	\frac{|u(x,t)-u(x,s)|}{|t-s|^{l/2}}.
	\]
	
    The following definitions are adopted from \cite[Chapter IV-V]{Ladyzhynska}. The parabolic Hölder space $C^{2+l,1+\frac{l}{2}}(\overline{\Omega}\times[0,T])$ consists of all functions whose spatial derivatives up to second order and first-order time derivative belong to $C^{l,\frac{l}{2}}(\overline{\Omega}\times[0,T])$. Its norm is defined by
	\[
	\|u\|_{C^{2+l,1+\frac{l}{2}}(\overline{\Omega}\times[0,T])}
	=
	\sum_{2r+|\alpha|\le2}
	\|D_x^\alpha\partial_t^r u\|_{C^{l,\frac{l}{2}}(\overline{\Omega}\times[0,T])}.
	\]
		For $1<p<\infty$, the parabolic Sobolev space $W_p^{2,1}(\Omega\times(0,T))$ is defined by
	\[
	W_p^{2,1}(\Omega\times(0,T))
	=
	\left\{
	u\in L^p(\Omega\times(0,T)):
	D_x^\alpha u,\,
	\partial_t u
	\in L^p(\Omega\times(0,T)),
	\ |\alpha|\le2
	\right\},
	\]
	equipped with the norm
	\[
	\|u\|_{W_p^{2,1}(\Omega\times(0,T))}
	=
	\sum_{|\alpha|\le2}
	\|D_x^\alpha u\|_{L^p(\Omega\times(0,T))}
	+
	\|\partial_tu\|_{L^p(\Omega\times(0,T))}.
	\]

\subsection{Auxiliary results}
In this subsection, we present collection of auxiliary results that will be useful in the proofs of the main results, including estimates associated with the Neumann heat semigroup, compact embedding results, and trace theorems for fractional Sobolev spaces.
\begin{lemma}[{\cite[Lemma 1.3]{Neumann}}] \label{Neumann estimates}
	Let $(e^{t\Delta})_{t\geq0}$ represent the Neumann heat semigroup defined on a bounded domain $\Omega\subset \mathbb{R}^d$ with smooth boundary, and let $\lambda_1>0$ be the first nonzero eigenvalue of the operator $-\Delta$ in $\Omega$ subject to homogeneous Neumann boundary conditions. Then there exist positive constants $C_1$ and $C_2$, depending only on the domain $\Omega$, such that the following estimates are satisfied:
	
	\begin{enumerate}
		\item[(i)] If $2 \leq p <\infty$, then for every $z \in W^{1,p}(\Omega)$ and for all $t>0$,
		\begin{align}\label{p<infty Neumann}
		\|\nabla e^{t\Delta} z\|_{L^p(\Omega)}
		\leq C_1 e^{-\lambda_1 t}\|\nabla z\|_{L^p(\Omega)}.
		\end{align}
		
		\item[(ii)] If $1 \leq q \leq p \leq \infty$, then for every $z \in L^q(\Omega)$ and for all $t>0$,
		\begin{align}\label{p,q infty Neumann}
		\|\nabla e^{t\Delta} z\|_{L^p(\Omega)}
		\leq C_2 \left(1+t^{-\frac{1}{2}-\frac{d}{2}\left(\frac{1}{q}-\frac{1}{p}\right)}\right)
		e^{-\lambda_1 t}\|z\|_{L^q(\Omega)}.
		\end{align}
	\end{enumerate}
\end{lemma}

In \cite[Lemma~1.3]{Neumann}, the first estimate in Lemma~\ref{Neumann estimates} is established only for $p\in[2,\infty)$. The endpoint case $p=\infty$ is not considered there. Since this case plays a crucial role in our analysis, we provide a detailed proof.
\begin{lemma}\label{Neumann est infty case}
	Under the assumptions of Lemma~\ref{Neumann estimates}, for every
	\(z \in W^{1,\infty}\) and all \(t>0\), there exists a constant
	\(C>0\), depending only on the domain \(\Omega\), such that the following estimate holds:
	\begin{align}\label{estimate p infty}
		\|\nabla e^{t\Delta} z\|_{L^\infty(\Omega)}
		\leq C_{\Omega}\|\nabla z\|_{L^\infty(\Omega)}.
	\end{align}
\end{lemma}
\begin{proof}
We assume that $\Omega$ is uniformly smooth with $C^{2+\alpha},$  $0<\alpha<1$ boundary. Let
$	S(t)=e^{t\Delta_N}.$ By definition, if
$
u(t,x)=S(t)f(x),
$
then $u$ solves
\begin{align}\label{eqn-neumann}
\begin{cases}
	u_t=\Delta u,
	& x\in\Omega,\quad t>0,\\[1mm]
	\partial_\nu u=0,
	& x\in\partial\Omega,\quad t>0,\\[1mm]
	u(0,x)=f(x),
	& x\in\Omega.
\end{cases}
\end{align}
Now choose the initial datum
$
f(x)\equiv 1.
$
The constant function
$
u(t,x)\equiv 1
$
solves the problem since 
$
u_t=0,\
\Delta u=0, \
\partial_\nu u=0,
$
and
$
u(0,x)=1.
$
Hence, by uniqueness,
\begin{equation}\label{eqn-1.10}
	S(t)1=1.
\end{equation} 
Let  $G(\cdot,\cdot,\cdot,\cdot)$ be the Green's matrix defined for $0\leq \tau<t\leq T$ and $x,\xi\in\Omega$ such that given any initial state $f$, the corresponding solution of problem  \eqref{eqn-neumann} is given by 
\begin{align}\label{eqn-sol}
	u(t,x)=S(t)f(x)=\int_{\Omega}G(t,x;0,\xi)f(\xi)d\xi. 
\end{align}
 Then, we can apply \cite[Theorem 2.2, (eq. no. (2.21))]{XMG}   with $\mu=0$, $|\nu|=1$ to find 
\begin{align}\label{eqn-1.8}
	|\nabla G(t,x;\tau,\xi)|\leq C(t-\tau)^{-(d+1)/2}\exp\left(-c\frac{|x-\xi|^2}{t-\tau}\right),
\end{align}
for some positive constants $C$ and $c$. Taking $\tau=0$ and $\xi=y$, we denote $G(t,x;0,y)$ as $K_N(t,x,y)$, so that from \eqref{eqn-1.8}, we infer 
\begin{align}\label{eqn-1.9}
	|\nabla K_N(t,x,y)|\leq Ct^{-(d+1)/2} \exp\left(-c\frac{|x-y|^2}{t}\right).
\end{align}
Then from \eqref{eqn-sol}, we get 
\begin{equation}\label{eqn-1.11}
	S(t)f(x)
	=
	\int_\Omega K_N(t,x,y)f(y)\,dy.
\end{equation}
Taking $f\equiv 1$ in \eqref{eqn-1.11} and using \eqref{eqn-1.10}, we obtain
\begin{equation*}
1=S(t)1
=
\int_\Omega K_N(t,x,y)\,dy, 	\ 
t>0,\  x\in\Omega,
\end{equation*}
and hence 
\begin{align*}
	\int_\Omega \nabla_x K_N(t,x,y)\,dy=0.
\end{align*}
Therefore 
\begin{align}\label{grad S iden}
	\nabla S(t)f(x)=\int_{\Omega}\nabla_x K_N(t,x,y)[f(y)-f(x)]dy. 
\end{align}

Further, for \(f \in W^{1,\infty}(\Omega)\), the characterization of \(W^{1,\infty}\) given in \cite[Theorem~4, p.~277]{evans} implies that
\begin{align*}
	|f(y)-f(x)|\leq C_{\Omega}\|\nabla f\|_{L^{\infty}(\Omega)}|x-y|.
\end{align*}
Subsequently, we derive from \eqref{grad S iden} that 
\begin{align}\label{eqn-1.16}
	|\nabla S(t)f(x)|&\leq \int_{\Omega}|\nabla_x K_N(t,x,y)||f(y)-f(x)|dy\nonumber\\&\leq  C_{\Omega}\|\nabla f\|_{L^{\infty}(\Omega)}\int_{\Omega}|\nabla_x K_N(t,x,y)||x-y|dy. 
\end{align}
By using \eqref{eqn-1.9} in \eqref{eqn-1.16}, we obtain 
\begin{align}\label{eqn-1.17}
	|\nabla S(t)f(x)|&\leq C_{\Omega}\|\nabla f\|_{L^{\infty}(\Omega)} t^{-(d+1)/2} \int_{\Omega} \exp\left(-c\frac{|x-y|^2}{t}\right) |x-y|dy\nonumber\\&\leq C_{\Omega}\|\nabla f\|_{L^{\infty}(\Omega)} t^{-(d+1)/2} \int_{\mathbb{R}^d} \exp\left(-c\frac{|x-y|^2}{t}\right) |x-y|dy. 
\end{align}
Let us set $z=\frac{y-x}{\sqrt{t}}$. Then $dy=t^{d/2}dz$ and $|x-y|=\sqrt{t}|z|$, and  \eqref{eqn-1.17} implies 
\begin{align*}
	|\nabla S(t)f(x)|\leq C_{\Omega}\|\nabla f\|_{L^{\infty}(\Omega)} t^{-(d+1)/2} t^{1/2}t^{d/2} \int_{\mathbb{R}^d} |z|e^{-c|z|^2}dz=C_{\Omega}\|\nabla f\|_{L^{\infty}(\Omega)}  \int_{\mathbb{R}^d} |z|e^{-c|z|^2}dz. 
\end{align*}
Since $\int_{\mathbb{R}^d} |z|e^{-c|z|^2}dz=\frac{\pi^{d/2}\Gamma(\frac{d+1}{2})}{\Gamma(\frac{d}{2})}c^{-\frac{d+1}{2}}=:C_d<\infty,$  we arrive at 
\begin{align*}
	\|\nabla e^{t\Delta_N}f\|_{L^{\infty}(\Omega)}\leq C_{\Omega}\|\nabla f\|_{L^{\infty}(\Omega)}. 
\end{align*}
That is, \eqref{p<infty Neumann} holds true for $p=\infty$ with $f\in W^{1,\infty}(\Omega)$  without the exponential decay. 
\end{proof}
\begin{lemma}[{\cite[Lemma 2.5]{JAMAA25}}] \label{beta k condition}
	Let $\beta\leq k$. Then, for every
	\begin{equation}\label{beta cond for grad v}
		m \in
		\begin{cases}
			\left[1, \dfrac{dk}{\beta d - k}\right), & \beta \ge \dfrac{k}{d},\\[6pt]
			[1,\infty], & \beta < \dfrac{k}{d},
		\end{cases}
	\end{equation}
	there exists a constant $C>0$ such that
	\begin{equation}\label{grad v 1st cond}
		\|\nabla v(\cdot,t)\|_{L^m(\Omega)}
		\le C \left( 1 + \sup_{0<t<T} \|w(\cdot,t)\|_{L^k(\Omega)}^\beta \right),
	\end{equation}
	for all $t \in (0,T)$.
\end{lemma}
\begin{proof}
	For the case $\beta<\frac{k}{d}$ and $m\in[1,\infty)$, the result follows by the same argument as in \cite[Lemma 2.5]{JAMAA25}. We therefore restrict our attention to the case $m=\infty$ and provide the details below.
	From the second equation in \eqref{Main}, using the variation of constants formula, we have
	\begin{equation}\label{voc of v}
		v(\cdot,t)
		= e^{t(\Delta-1)}v_0
		+ \int_0^t e^{(t-s)(\Delta-1)} w^\beta(\cdot,s)\,ds,
		\ \text{for all } t \in (0,T).
	\end{equation}
Applying Lemma~\ref{Neumann est infty case} together with Lemma~\ref{Neumann estimates}(ii), with \(q=\frac{k}{\beta}\) and \(p=\infty\), we deduce
\begin{equation*}\label{L^m norm of grad v}
	\begin{aligned}
		\|\nabla v(\cdot,t)\|_{L^\infty(\Omega)}
		&\le \|\nabla e^{t(\Delta-1)}v_0\|_{L^\infty(\Omega)}
		+ \int_0^t
		\|\nabla e^{(t-s)(\Delta-1)} w^\beta(\cdot,s)\|_{L^\infty(\Omega)} ds\\
		&\le C \|\nabla v_0\|_{L^\infty(\Omega)}\\
		&\quad+C_2\int_0^t
		e^{-(t-s)(\lambda_1+1)}
		\left(1+(t-s)^{-\frac12 - \frac{d}{2}\left(\frac{\beta}{k}-\frac{1}{\infty}\right)}\right)
		\|w^\beta(\cdot,s)\|_{L^{k/\beta}(\Omega)} ds
		\\
		&\le C\|\nabla v_0\|_{L^\infty(\Omega)}\\
		&\quad +\sup_{0<t<T} \|w(\cdot,t)\|_{L^k(\Omega)}^\beta C_2
		\int_0^t
			e^{-(t-s)(\lambda_1+1)}
		\left(1+(t-s)^{-\frac12 - \frac{d}{2}\frac{\beta}{k}}\right) ds.	
	\end{aligned}
\end{equation*}
To ensure the finiteness of the integral, we require
\(
-\frac12-\frac{d}{2}\frac{\beta}{k}>-1,
\)
or equivalently,
\(
\beta<\frac{k}{d}
\). Therefore, in view of \eqref{Initial Cond} and the condition
\(\beta<\frac{k}{d}\), we infer that
\begin{equation*}
	\|\nabla v(\cdot,t)\|_{L^\infty(\Omega)}
	\le C \left( 1 + \sup_{0<t<T} \|w(\cdot,t)\|_{L^k(\Omega)}^\beta \right),
\end{equation*}
for all $t \in (0,T)$.

\end{proof}
\begin{lemma}[{\cite[Proposition 4.22(ii)]{Haroske}}]\label{compact embedd}
	Let $\Omega$ be a bounded domain with smooth boundary and let $r\in (0,\infty)$. Then
	\[
	W^{r,2}(\partial \Omega) \hookrightarrow L^2(\partial \Omega),
	\]
	is a compact embedding.
\end{lemma}

\begin{lemma}[{\cite[Theorem 4.24(i)]{Haroske}}]\label{linear bounded map}
	Let $\Omega$ be a bounded domain with smooth boundary. If $r > 0$, then there exists a linear and bounded map from $W^{r+\frac{1}{2},2}(\Omega)$ onto $W^{r,2}(\partial \Omega)$.
\end{lemma}

	\section{Local existence and uniqueness} 
	\label{sec:local}
	
	In this section, we establish the local existence and uniqueness of classical solutions to system \eqref{Main}. The main difficulty arises from the possible degeneracy of the motility function, which hinders the direct application of the classical theory for uniformly parabolic equations. To overcome this difficulty, we introduce a suitable cut-off function and construct an approximate uniformly parabolic system. We then employ classical parabolic regularity theory to derive the parabolic H\"older estimates required to verify the self-mapping property of the associated fixed-point operator. The existence of local classical solutions is obtained by applying the Schauder fixed-point theorem, while uniqueness follows from suitable a priori estimates.
	
	Our proof is based on the framework developed in \cite[Lemma~2.1]{SIAM} and the approach adopted in \cite[Section~2]{localex}. However, several modifications are required to deal with the additional difficulties posed by the present system.
\begin{theorem}\label{local ex lemma}
Let $\Omega\subset\mathbb{R}^d$ $(d\ge2)$ be a bounded domain with smooth boundary, and let
$\mu,\beta,\delta>0$, $r\in\mathbb{R}$, and $\theta\ge1$.
If the initial data fullfill \eqref{Initial Cond} and that the motility
function satisfies hypothesis \eqref{Hypothesis H1}, then there exist  $T_{\max}\in(0,\infty]$ and a triple of non-negative functions
$(u,v,w)$ satisfying 
\begin{equation}\label{eqn-regularity}
		\left\{
		\begin{aligned}
			u &\in C^0\big(\overline{\Omega}\times[0,T_{\max})\big)
			\cap C^{2,1}\big(\overline{\Omega}\times(0,T_{\max})\big),\\[2mm]
			v &\in C^0\big(\overline{\Omega}\times[0,T_{\max})\big)
			\cap C^{2,1}\big(\overline{\Omega}\times(0,T_{\max})\big)
			\cap L^\infty\big([0,T_{\max});W^{1,\infty}(\Omega)\big),\\[2mm]
			w &\in C^{0,1}\big(\overline{\Omega}\times[0,T_{\max})\big),
		\end{aligned}
		\right.
\end{equation}
	 which solves \eqref{Main} in the classical sense. Moreover, if
	$T_{\max}<\infty$, then it follows that
\begin{equation}\label{blow up crit}
		\begin{aligned}
				\limsup_{t\to T_{\max}} \|u(\cdot,t)\|_{L^\infty(\Omega)}= \infty.
		\end{aligned}
\end{equation}
\end{theorem}
\begin{proof}
Since the initial data \eqref{Initial Cond} do not possess sufficient regularity required to apply the abstract results from \cite{Ladyzhynska}, we first approximate them by a sequence of smoother functions. Specifically, for the non-negative initial data $(u_0,v_0,w_0) \in C^0(\overline{\Omega})\times 
 C^1(\overline{\Omega})\times  C^1(\overline{\Omega})$, there exists a sequence
\begin{align}\label{eqn-smooth}
\{(u_0^n,v_0^n,w_0^n)\}_{n\in\mathbb{N}}
\in
\bigl(C^{2+\alpha}(\overline{\Omega})\bigr)^3,
\ 0<\alpha<1,
\end{align}
such that as $n\to\infty$ 
\[
\begin{aligned}
	u_0^n \to u_0 \  \text{ in } \ C^0(\overline{\Omega}), \ 
	v_0^n \to v_0 \  \text{ in } \ C^1(\overline{\Omega}), \ 
	w_0^n \to w_0 \  \text{ in } \ C^1(\overline{\Omega}),
\end{aligned}
\]
and
	\[
	u_0^n,\; v_0^n,\; w_0^n \geq 0.
	\]
Moreover, we choose the approximating sequence to satisfy the compatibility conditions, namely,
	\[
	\partial_\nu u_0^n
	=
	\partial_\nu v_0^n
	=
	0
	\ \text{on }\partial\Omega.
	\]
	Such an approximation is possible by taking $u_0^n=e^{\frac{1}{n}\Delta_N}u_0$, $v_0^n=e^{\frac{1}{n}\Delta_N}v_0$ and $w_0^n=e^{\frac{1}{n}\Delta_N}w_0$, where $\Delta_N$ is the Neumann Laplacian. 
For notational convenience, we omit the superscript $n$ and denote the approximating sequence $(u_0^n,v_0^n,w_0^n)$ by $(u_0,v_0,w_0)$ throughout the remainder of the existence proof. We divide the local existence proof into the following steps:
\vskip 0.1cm 
\noindent 
\textbf{Step 1:} \emph{Basic setup for fixed point argument.} 
Fix $T\in(0,1)$, to be chosen appropriately later, and consider the Banach space
$
X:=C^0\bigl(\overline{\Omega}\times[0,T]\bigr).
$
We introduce a closed, bounded, and convex subset of \(X\) by
\begin{equation}\label{banach_space}
	S:=\left\{
	\bar{u}\in X:\;
	0\leq \bar{u}(x,t)\leq M,
	\ \text{for all}\ (x,t)\in\overline{\Omega}\times[0,T]
	\right\},
\end{equation}
where
\begin{align}\label{eqn-M}
M=\|u_0\|_{L^\infty(\Omega)}+1.
\end{align}
Since  the initial data $u_0$ satisfies \eqref{Initial Cond}, it is clear that $u_0\in S$. For a given \(\bar{u}\in S\), we first extend $\bar{u}$ to some continuous function defined on \(\overline{\Omega}\times[0,1]\).
We define the operator
$
\phi:S\longrightarrow S$,
$
\bar{u}\longmapsto \phi(\bar{u}):=u,
$
where \(u\) denotes the solution of the corresponding problem 
\begin{equation}\label{u}
	\left\{ 
	\begin{aligned}
		&u_t = \nabla \cdot \big(\widehat{\gamma}(v)\nabla u\big)
		-  \nabla \cdot \big(\widehat{\gamma}(v)u\nabla v) + \mu u(1-\bar{u}^\theta)(\bar{u}-A),	\quad	
		&& x \in \Omega,\ 0<t<1, \\
		&\partial_{\nu} u = 0, \quad
		&& x \in \partial \Omega,\ 0<t<1, \\
		&u(x,0) = u_0(x), \quad
		&& x \in \Omega,
	\end{aligned}
	\right.
\end{equation}
and $\widehat{\gamma}(v)\in C^3(\mathbb{R})$ satisfies 
\begin{equation}\label{gamma hat}
	\widehat{\gamma}(v) :=
	\begin{cases}
		\gamma(0), & \text{if } v < 0, \\
		\gamma(v), & \text{if } 0 \le v \le M, \\
		\gamma(M), & \text{if } v > M,
	\end{cases}
\end{equation}
where $v$ and $w$ are the solutions of 
\begin{equation}\label{v}
	\left\{
	\begin{aligned}
		&v_t = \Delta v-v+w^{\beta}, \
		&& x \in \Omega,\ 0<t<1, \\
		&\partial_{\nu} v = 0, \
		&& x \in \partial \Omega,\ 0<t<1, \\
		&v(x,0) = v_0(x), \
		&& x \in \Omega,
	\end{aligned}
	\right.
\end{equation}
and
\begin{equation}\label{w}
	\left\{
	\begin{aligned}
		&w_t = -\delta w+ \bar{u}, \
		&& x \in \Omega,\ 0<t<1, \\
		&w(x,0) = w_0(x), \
		&& x \in \Omega.
	\end{aligned}
	\right.
\end{equation}
We next apply the Schauder fixed point theorem to establish the existence of a fixed point $u$ of the mapping $\phi$, provided that $T>0$ is sufficiently small. Throughout the proof of Theorem \ref{local ex lemma}, we use $c_i>0$, $i=1,2,3,\ldots$, to represent generic positive constants depending upon $M=\|u_0\|_{L^\infty(\Omega)}, \|v_0\|_{W^{1,\infty}(\Omega)}$ and $\|w_0\|_{L^{\infty}(\Omega)}$. 

\vskip 0.1cm 
\noindent 
\textbf{Step 2:} \emph{Application of some parabolic regularity results.}
Since system \eqref{w} involves no spatial derivatives, consequently, its solution admits the explicit representation
\begin{equation}\label{duhamels for w}
	\begin{aligned}
		w(x,t) = w_0(x)e^{-\delta t}+ \int_0^t e^{-\delta(t-s)}\bar{u}(x,s)\,ds,
		\quad x \in \overline{\Omega},\ 0\leq t\leq 1.
	\end{aligned}
\end{equation}
By the assumption $w_0\in C^1(\overline{\Omega})$ and $\bar{u}\in C^0\bigl(\overline{\Omega}\times[0,1]\bigr)$, $w$ is jointly continous, non-negative, and satisfies
\begin{equation*}
	\begin{aligned}
		0\leq w(x,t)\leq \|w_0\|_{L^\infty(\Omega)}+\frac{M}{\delta}=c_1,\ \text{for all}\ (x,t)\in\overline{\Omega}\times [0,1].
	\end{aligned}
\end{equation*}
Differentiating \eqref{duhamels for w} with respect to $t$ under the 
integral sign yields the existence and continuity of $w_t$. Together 
with the continuity of $w$ on $\overline{\Omega}\times[0,1]$, this is consistent with the claimed regularity \(w \in C^{0,1}(\overline{\Omega}\times[0,1])\). Further, for $\bar{u}_1,\bar{u}_2$ with corresponding ${w}_1, {w}_2$,
\[w_1(x,t)-w_2(x,t)=\int_{0}^{t}e^{-\delta (t-s)}\left[\bar{u}_1(x,s)-\bar{u}_2(x,s)\right]ds,\]
it follows that
\[|w_1(x,t)-w_2(x,t)|\leq\sup_{0<s<t}\|\bar{u}_1(\cdot,s)-\bar{u}_2(\cdot,s)\|_{L^\infty(\Omega)}\int_{0}^{t}e^{-\delta (t-s)}ds\leq \frac{1}{\delta}\|\bar{u}_1-\bar{u}_2\|_{L^\infty(\Omega\times (0,1))},\]
taking the supremum over $\overline{\Omega}\times [0,1]$, we obtain
\[\|w_1-w_2\|_{L^\infty(\Omega\times (0,1))}\leq \delta^{-1}\|\bar{u}_1-\bar{u}_2\|_{L^\infty(\Omega \times (0,1))},\]
which claims the Lipschitz dependence of $w$ on $\bar{u}$.
Since $\beta>0$ and $0\leq w(x,t)\leq c_1$ for all $(x,t)\in(\overline{\Omega}\times [0,1])$, we have $w^\beta\in L^\infty(\Omega\times(0,1))$. Since $\Omega\times(0,1)$ is bounded, the continuous embedding $L^\infty(\Omega\times(0,1))\hookrightarrow L^p(\Omega\times(0,1))$ holds for every $1\leq p<\infty$. It follows that the source term $w^\beta$ in the system \eqref{v} lies in 
$L^p(\Omega\times(0,1))$ for each $1\leq p<\infty$. An application of the parabolic $L^p$-regularity theory from \cite[Theorem 9.1, Chapter IV]{Ladyzhynska} therefore yields
$v\in W^{2,1}_p(\Omega\times(0,1))$ for every $1<p<\infty$. For $p\in \left(d+2,\infty\right)$, the parabolic Sobolev embedding theorem \cite[Lemma 3.3, Chapter II, p.~82]{Ladyzhynska} implies that
$W^{2,1}_p(\Omega\times(0,1))\hookrightarrow C^{l_1+1,\frac{l_1+1}{2}}(\overline{\Omega}\times[0,1])$, where $l_1=1-\frac{d+2}{p}\in(0,1)$. Consequently, an application of the parabolic $L^p$-regularity theory \cite[Theorem 9.2, Chapter IV]{Ladyzhynska} yields
\begin{equation*}
\|v\|_{C^{l_1+1,\frac{l_1+1}{2}}(\overline{\Omega}\times[0,1])}\leq c_2,
\end{equation*}
where $l_1\in(0,1)$. Subsequently, we get
\begin{equation*}
	\|\nabla v(\cdot,t)\|_{L^\infty(\Omega\times(0,1))}\leq c_3.
\end{equation*}
Combining the above estimates with hypothesis \eqref{Hypothesis H1} and the definition of $\widehat{\gamma}(v)$ in \eqref{gamma hat}, we conclude that $\widehat{\gamma}(v)$ admits a positive lower bound that depends only on $M$. Moreover, since
$\widehat{\gamma}'(v)\nabla v \in L^\infty(\Omega\times(0,1))$ and
$\mu(1-\bar{u}^{\theta})(\bar{u}-A)
\in L^\infty(\Omega\times(0,1))$, the equation \eqref{u} is uniformly
parabolic.
Let us now establish the following results:
\begin{itemize}
	\item [\emph{(i)}] There exists a constant 
	$c_4>0$ such that
	$
	\|u\|_{L^\infty(\Omega\times(0,1))}\leq c_4.
	$
	\item [\emph{(ii)}] There exists a constant $c_5>0$ such that
	$\|u\|_{C^{l_2,\frac{l_2}{2}}(\overline{\Omega}\times[0,1])}
	\leq c_5,\ l_2\in(0,1)$.
\end{itemize}
\emph{Proof of (i)}. 
By virtue of the definition of the truncated motility function $\widehat{\gamma}$ in \eqref{gamma hat}, one can find positive constants $k_1$, $k_2$, and $k_3$ such that
\[
0<k_1\leq \widehat{\gamma}(v)\leq k_2,
\
|\widehat{\gamma}'(v)|\leq k_3\ \text{in} \ \Omega\times(0,1).
\]
To derive the constant $c_4>0$, we verify the hypotheses of
\cite[Theorem 2.1, Chapter V]{Ladyzhynska}. Since equation~\eqref{u} is written
in divergence form, it suffices to verify the following conditions.
 \begin{equation}\label{cond 1}
 	\begin{aligned}
 		|\widehat{\gamma}(v)\nabla u+\widehat{\gamma}'(v)u\nabla v|&\leq |\widehat{\gamma}(v)||\nabla u|+|\widehat{\gamma}'(v)||u||\nabla v|\\
 		&\leq k_2|\nabla u|+k_3c_3|u|,
 	\end{aligned}
 \end{equation}
 and
 \begin{equation}\label{cond 2}
 	\begin{aligned}
 		\mu u(1-\bar{u}^\theta)(\bar{u}-A)&=\mu(u\bar{u}-Au-u\bar{u}^{\theta+1}+Au\bar{u}^\theta)\\
 		&\leq \mu |u|\bar{u}+A|u|\bar{u}^{\theta},\\
 	|	\mu u(1-\bar{u}^\theta)(\bar{u}-A)|	&\leq (\mu M+AM^\theta)|u|,
 	\end{aligned}
 \end{equation}
 where $\theta\geq1$. Also,  using Young's inequality, we find 
 \begin{equation}\label{cond 3 u in L^infty}
 	\begin{aligned}
 		(\widehat{\gamma}(v)\nabla u+\widehat{\gamma}'(v)u\nabla v)\cdot\nabla u&=\widehat{\gamma}(v)|\nabla u|^2+\widehat{\gamma}'(v)u\nabla v\cdot\nabla u \\
 		&\geq \widehat{\gamma}(v)|\nabla u|^2-\widehat{\gamma}'(v)|u||\nabla v||\nabla u|\\
 		&\geq \frac{|\widehat{\gamma}(v)|}{2}|\nabla u|^2-|u|^2 \frac{|\widehat{\gamma}'(v)|^2}{|\widehat{\gamma}(v)|}|\nabla v|^2\\
 		&\geq \frac{k_1}{2}|\nabla u|^2-\frac{c^2_3k^2_3}{k_1}|u|^2.		
    \end{aligned}
\end{equation}
Therefore, it remains only to verify that
 \begin{equation}\label{C_4}
 	\begin{aligned}
 		-\mu u\left(u(1-\bar{u}^\theta)(\bar{u}-A)\right)&= \mu(-u^2\bar{u}+Au^2+u^2\bar{u}^{\theta+1}-Au^2\bar{u}^\theta) \\
 		&\leq \mu A|u|^2+\mu |u|^2 \bar{u}^{\theta+1}\\
 		&\leq \mu(A+M^{\theta+1})|u|^2.
 	\end{aligned}
 \end{equation}
 In view of \eqref{cond 1}-\eqref{C_4}, all the hypotheses of
 \cite[Theorem 2.1, Chapter V]{Ladyzhynska} are fulfilled. Consequently, there
 exists a constant $c_4>0$ such that
 $
 \|u\|_{L^\infty(\Omega\times(0,1))}\leq c_4.
 $ 
 
\noindent\emph{Proof of (ii).} 
 We now proceed to establish the parabolic H\"older regularity of $u$ by deriving an estimate of the form
 $
 \|u\|_{C^{l_2,\frac{l_2}{2}}(\overline{\Omega}\times[0,T])}\leq c_5,
 \ l_2\in(0,1).
 $
Since $
\|u\|_{L^\infty(\Omega\times(0,1))}\leq c_4
$,  the estimates \eqref{cond 1}-\eqref{C_4} allow us to invoke the parabolic regularity results \cite[Theorem 1.1, Chapter V]{Ladyzhynska} and \cite[Theorem 1.3, Remark 1.4]{Porzio}, from which the required parabolic Hölder continuity follows immediately such that
\begin{equation}\label{Holder-cont-of-u}
	\begin{aligned}
		\|u\|_{C^{l_2,\frac{l_2}{2}}{(\overline{\Omega}\times [0,1])}}\leq c_5,
	\end{aligned}
\end{equation}
where $l_2 \in (0,1)$. 
\vskip 0.1cm 
\noindent 
\textbf{Step 3:} \emph{Application of the Schauder fixed point theorem.} By using the definition \eqref{eqn-parabolic-norm}, from \eqref{Holder-cont-of-u},  we infer 
\begin{align*}
	|u(x,t)|
	\leq
	\|u_0\|_{L^\infty(\Omega)}
	+
	c_5 t^{\frac{l_2}{2}}, \ \text{ for all }\ (x,t)\in\overline{\Omega}\times[0,1]. 
\end{align*}
We therefore can choose $T\in(0,1)$ sufficiently small such that $T<\left(\frac{1}{c_5}\right)^{\frac{2}{l_2}}$ and  for all $x\in\overline{\Omega}$, and $t\in(0,T)$,
\[
|u(x,t)|
\leq
\|u_0\|_{L^\infty(\Omega)}
+
c_5 T^{\frac{l_2}{2}}
\leq M,
\]
where $M$ is defined in \eqref{eqn-M}. 
Therefore, $u\in S$, and hence the mapping $\phi$ maps $S$ into itself whenever $T>0$ is chosen sufficiently small. 

Additionally, we need to show that $\phi$ is continuous and $\phi(S)$ is relatively compact in $X$. Let $\widetilde u_i \in X$ and $u_i = \phi(\widetilde u_i)$ for $i=1,2$ and $v_i$ and $w_i$ satisfy
\begin{align}\label{v_i equation}
	\begin{cases}
		v_{it} = \Delta v_i - v_i + w_i^\beta, & x \in \Omega,\ 0<t<1,\\
		\partial_\nu v_i = 0, & x \in \partial\Omega,\ 0<t<1,\\
		v_i(x,0) = v_0(x), & x \in \Omega,
	\end{cases}
\end{align}
and
\begin{align}\label{w_i equation}
	\begin{cases}
		w_{it} = -\delta w_i + \bar{u}_i, & x \in \Omega,\ 0<t<1,\\
		w_i(x,0) = w_0(x), & x \in \Omega.
	\end{cases}
\end{align}
The self-mapping property of $\phi$, together with
\eqref{Holder-cont-of-u} and \eqref{w_i equation}, ensures that $w$
has the required parabolic H\"older regularity. Indeed,
fix $t\in[0,T]$ and $x,y\in\overline{\Omega}$. From the Duhamel
representation,
\[
w(x,t)-w(y,t)
=
e^{-\delta t}\bigl(w_0(x)-w_0(y)\bigr)
+
\int_0^t e^{-\delta(t-\tau)}
\bigl(u(x,\tau)-u(y,\tau)\bigr)\,d\tau.
\]
Since $w_0\in C^1(\overline{\Omega})$, it is Lipschitz on $\overline{\Omega}$,
with
\[
|w_0(x)-w_0(y)|
\leq L_0|x-y|,
\
L_0=\|\nabla w_0\|_{L^\infty(\Omega)}.
\]
Further, $u(\cdot,\tau)\in C^{l_2}(\overline{\Omega})$
uniformly in $\tau\in[0,T]$ with constant $c_5$, so that
\[
\left|
\int_0^t e^{-\delta(t-\tau)}
\bigl(u(x,\tau)-u(y,\tau)\bigr)\,d\tau
\right|
\leq
c_5|x-y|^{l_2}
\int_0^t e^{-\delta(t-\tau)}\,d\tau
\leq
\frac{c_5}{\delta}|x-y|^{l_2}.
\]
Combining the above estimates and using $e^{-\delta t}\leq 1$, we obtain
\[
|w(x,t)-w(y,t)|
\leq C|x-y|^{l_2},\ \text{ where }\ C=C(L_0,\Omega,l_2,c_5,\delta).
\]
For the temporal estimate, since
\[
w_t=-\delta w+u,
\]
with $0\leq w\leq c_1$ and $0\leq u\leq M$ in $\overline{\Omega}\times[0,1]$, we have
\[
|w_t|
\leq \delta c_1+M=L_1,
\]
on $\overline{\Omega}\times[0,T]$. Hence $w(x,\cdot)$ is Lipschitz
continuous in $t$, uniformly with respect to $x$:
\[
|w(x,t)-w(x,s)|
\leq L_1|t-s|.
\]
Since $t,s\in[0,T]\subset[0,1]$, we have $|t-s|\leq1$, 
combining both spatial and temporal estimates, we conclude that
\[
w\in C^{l_2,\frac{l_2}{2}}
\bigl(\overline{\Omega}\times[0,T]\bigr). 
\]
We claim that
\begin{align}\label{eqn-para-reg}
w^\beta\in C^{l_3,\frac{l_3}{2}}
\bigl(\overline{\Omega}\times[0,T]\bigr), \ \text{ where }\  l_3=l_2\min\{\beta,1\}\in(0,1).
\end{align}

\textit{Case 1: $\beta\geq 1$.}
The map $s\mapsto s^\beta$ is $C^1$, hence Lipschitz, on $[0,M]$,
with constant $\beta M^{\beta-1}$. Composing a Lipschitz map with a
H\"older-$l$ map preserves the exponent $l$ (only the
constant changes), so
\[
|w^\beta(x,t)-w^\beta(y,s)|
\leq
\beta M^{\beta-1}|w(x,t)-w(y,s)|
\leq
C\left(|x-y|^{l_2}+|t-s|^{l_2/2}\right),
\]
that is, 
\[
w^\beta
\in
C^{l_2,\frac{l_2}{2}}
\bigl(\overline{\Omega}\times[0,T]\bigr).
\]

\textit{Case 2: $0<\beta<1$.}
For all $a,b\geq0$, the elementary inequality
\[
|a^\beta-b^\beta|\leq |a-b|^\beta,
\]
holds. Applying this with $a=w(x,t)$ and $b=w(y,s)$, and using
\[
(A+B)^\beta\leq A^\beta+B^\beta,
\ A,B\geq0,\ \beta\in(0,1),
\]
we obtain
\[
\begin{aligned}
	|w^\beta(x,t)-w^\beta(y,s)|
	&\leq |w(x,t)-w(y,s)|^\beta \\
	&\leq
	\left(
	C\bigl(|x-y|^{l_2}+|t-s|^{l_2/2}\bigr)
	\right)^\beta \\
	&\leq
	C^\beta
	\left(
	|x-y|^{l_2\beta}
	+
	|t-s|^{l_2\beta/2}
	\right).
\end{aligned}
\]
In either case, we deduce \eqref{eqn-para-reg}. 
An application of the classical parabolic regularity theory in \cite[Theorems 5.1-5.3, Chapter IV, p.~322]{Ladyzhynska} then yields
\begin{align}\label{schauder regu. v}
	v \in C^{2+l_3,\,1+\frac{l_3}{2}}(\overline{\Omega} \times [0,T]),
\end{align}
for some $l_3 \in (0,1)$. Now, using \eqref{schauder regu. v} and  applying  \cite[Theorem 6.1, Chapter V, p.~452]{Ladyzhynska}, we also obtain
\begin{align}\label{schauder regu. u}
	u \in C^{2+l_4,\,1+\frac{l_4}{2}}(\overline{\Omega} \times [0,T]),
\end{align}
for some $l_4 \in (0,1)$. 

Letting $h=\phi(\bar{u}_1)-\phi(\bar{u}_2)=u_1-u_2$, then from \eqref{u}, we find
 \begin{equation}\label{h eq}
 \left\{
 \begin{aligned}
 	&h_t = \nabla\cdot(\widehat{\gamma}(v_1)\nabla h)
 	+ \widehat{\gamma}'(v_1)\nabla v_1\cdot\nabla h
 	+ g_1(x,t)h + g_2(x,t),\quad
 	&& x\in\Omega,\ t\in(0,T), \\ 
 	&\partial_{\nu} h = 0,\quad
 	&& x\in\partial\Omega,\ t\in(0,T), \\ 
 	&h(x,0) = 0,\quad
 	&& x\in\Omega,
 \end{aligned}
 \right.
 \end{equation}
where
\begin{equation}\label{g1 define}
	\begin{aligned}
		g_1(x,t)=\widehat{\gamma}''(v_1)|\nabla v_1|^2+\widehat{\gamma}'(v_1)\Delta v_1+\mu \left(\bar{u}_1-A-\bar{u}_1^{\theta+1}+A\bar{u}_1^\theta\right)
	\end{aligned},
\end{equation}
with
\begin{equation}
	\begin{aligned}
		g_2(x,t)=&\nabla\cdot([\widehat{\gamma}(v_1)-\widehat{\gamma}(v_2)]\nabla u_2)+\nabla\cdot[(\widehat{\gamma}'(v_1)-\widehat{\gamma}'(v_2))u_2\nabla v_1]+\nabla\cdot[\widehat{\gamma}'(v_2)u_2(\nabla v_1-\nabla v_2)]\\
		&+\mu u_2\left(\bar{u}_1-\bar{u}_2\right)\left[1-(\theta+1)(\bar{u}_1+\bar{u}_2)^{\theta}+A\theta(\bar{u}_1+\bar{u}_2)^{\theta-1}\right].
	\end{aligned}
\end{equation}
Owing to the facts that $g_1\in L^\infty(\Omega\times(0,T))$ and $\widehat{\gamma}'(v_1)\nabla v_1 \in L^\infty(\Omega \times (0,T))$, an application of the parabolic $L^p$-regularity theory to \eqref{h eq} (cf. \cite[Theorem~9.2, Chapter~IV]{Ladyzhynska}) implies that
\begin{equation}
	\begin{aligned}\label{h-g_2 est}
		\|h\|_{W_p^{2,1}{(\Omega\times(0,T))}}&\leq c_6\|g_2\|_{L^p{(\Omega\times(0,T))}},
	\end{aligned}
\end{equation}
for $1<p<\infty$. By the mean value theorem, together with the assumption $\widehat{\gamma}\in C^3(\mathbb{R})$ and \eqref{gamma hat}, there exists a positive constant $C$ such that
$|\widehat{\gamma}(v_1)-\widehat{\gamma}(v_2)|\le C|v_1-v_2|$ and
$|\widehat{\gamma}'(v_1)-\widehat{\gamma}'(v_2)|\le C|v_1-v_2|$. Consequently, each nonlinear term in $g_2$ can be estimated in terms of $\Delta(v_1-v_2)$, $\nabla(v_1-v_2)$, and $v_1-v_2$. Moreover, since $u_i\in C^{2+l_4,\,1+\frac{l_4}{2}}(\overline{\Omega}\times[0,T])$ and $v_i\in C^{2+l_3,\,1+\frac{l_3}{2}}(\overline{\Omega}\times[0,T])$ for $i=1,2$, it follows that $u_2$, $\nabla u_2$, and $\nabla v_1$ are uniformly bounded on $\overline{\Omega}\times[0,T]$. Therefore, from \eqref{h-g_2 est}, we deduce
\begin{equation}\label{h in v and u}
	\begin{aligned}
		\|h\|_{W_p^{2,1}{(\Omega\times(0,T))}}&\leq c_7\left(\|v_1-v_2\|_{W^{2,p}(\Omega\times (0,T))}+\|\bar{u}_1-\bar{u}_2\|_{L ^{p}(\Omega\times (0,T))}\right).
	\end{aligned}
\end{equation}
Further, from \eqref{v_i equation} and \eqref{w_i equation}, it follows that
\begin{equation}\label{v in w beta}
	\begin{aligned}
		\|v_1-v_2\|_{W^{2,p}(\Omega \times (0,T))}\leq c_{8}\|w^\beta_1 -w^\beta_2\|_{L ^{p}({\Omega}\times (0,T))}.
	\end{aligned}
\end{equation}
Next, we estimate the $L^p$-norm of $w_1^\beta-w_2^\beta$ by considering two separate cases, depending on the value of $\beta$. \\
\textit{Case 1: $0<\beta<1$.} Invoking inequality \eqref{ineq power p 1} for $A,B> 0$ with $0<\beta<1$ yields
\begin{align}\label{combined ineq}
	(B^{1-\beta}+A^{1-\beta})|A^\beta-B^\beta|\leq \frac{1}{2^{\beta-1}}|A-B|.
\end{align}
Subsequently, we employ the elementary inequalities
$(a+b)^q \leq a^q+b^q$ for $0<q<1$ and
$(a+b)^q \leq 2^{q-1}(a^q+b^q)$ for $q\geq1$, valid for all
$a,b\geq0$. Setting $q=\frac{1-\beta}{\beta}$ and recalling that
$\beta\in(0,1)$, we obtain 
\begin{equation*}
	\begin{aligned}
		\int_{\Omega}|w^\beta_1-w^\beta_2|^p &=\int_{\Omega}|w^\beta_1-w^\beta_2|^\frac{1-\beta}{\beta}|w^\beta_1-w^\beta_2||w^\beta_1-w^\beta_2|^{p-\frac{1}{\beta}}\\
		&\leq {C_1}\int_{\Omega}|w^{1-\beta}_1+w^{1-\beta}_2||w^\beta_1-w^\beta_2||w^\beta_1-w^\beta_2|^{p-\frac{1}{\beta}},
	\end{aligned}
\end{equation*}
where 
\[
C_1=
\begin{cases}
	1, & \text{if } \beta>\dfrac{1}{2},\\[1ex]
	2^{\frac{1-2\beta}{\beta}}, & \text{if } 0<\beta\leq\dfrac{1}{2}.
\end{cases}
\] 
In addition, for $p\geq\frac{1}{\beta}$, we further get
\begin{equation*}
	\int_{\Omega}|w^\beta_1-w^\beta_2|^p\leq {C_1} 2^{p-\frac{1}{\beta}-1}\sup_{x\in \Omega}\left[|w_1(x)|^{\beta(p-\frac{1}{\beta})}+|w_2(x)|^{\beta(p-\frac{1}{\beta})}\right]\int_{\Omega}|w^{1-\beta}_1+w^{1-\beta}_2||w^{\beta}_1-w^{\beta}_2|,
\end{equation*}		
for all $t\in(0,T)$. Consequently, \eqref{combined ineq} yields
\begin{equation*}
	\begin{aligned}
		\int_{\Omega}|w^\beta_1-w^\beta_2|^p&\leq {C_1} \frac{2^{p-\frac{1}{\beta}-1}}{2^{\beta-1}}\sup_{x\in \Omega}[|w_1(x)|^{\beta(p-\frac{1}{\beta})}+|w_2(x)|^{\beta(p-\frac{1}{\beta})}]|\Omega|\|w_1-w_2\|_{L^2(\Omega)}\\
		&\leq {C_2}\|w_1-w_2\|_{L^2(\Omega)}\\
		&\leq {C_3}\|w_1-w_2\|_{L^\infty (\Omega)},
	\end{aligned}
\end{equation*}
for all $t\in(0,T)$. Taking supremum over time $t\in(0,T)$, it leads to
\begin{equation*}
		\sup_{0<t<T}\int_{\Omega}|w_1^\beta(\cdot,t)-w_2^\beta(\cdot,t)|^p
		\le
		C_3\sup_{0<t<T}
		\|w_1(\cdot,t)-w_2(\cdot,t)\|_{L^\infty(\Omega)},
	\end{equation*}
	and
	\begin{equation*}
		\|w_1^\beta-w_2^\beta\|_{L^\infty(0,T;L^p(\Omega))}^p\leq {C_4}\|w_1-w_2\|_{L^{\infty}(\Omega\times(0,T))}.
\end{equation*}
Finally, combining the above estimate with the explicit representation formula \eqref{duhamels for w} for $w$, we obtain
\begin{equation}\label{w_i in u_i}
	\begin{aligned}
	\|w^\beta_1-w^\beta_2\|^{p}_{L^p(\Omega\times(0,T))}\leq {C_5}\|w_1-w_2\|_{L^\infty (\Omega\times(0,T))}\leq C \|\bar{u}_1-\bar{u}_2\|_{L^{\infty}(\Omega\times(0,T))},
	\end{aligned}
\end{equation}
for all $t\in(0,T)$, where $C_5$ and $C$ are positive constants depending on $T$. Hence, from \eqref{v in w beta}, we infer 
\[
\|v_1-v_2\|_{W^{2,p}(\Omega\times(0,T))}
\leq c_{9}\left(\|\bar{u}_1-\bar{u}_2\|_{L^p(\Omega\times(0,T))}\right)^{1/p},
\]
with $c_{9}=c_{8}C^{1/p}$. Having established this, from \eqref{h in v and u}, we further have 
\begin{equation}\label{h in u in C0}
\begin{aligned}
\|h\|_{W_p^{2,1}(\Omega)\times(0,T)}&\leq c_{10} \left(\|\bar{u}_1-\bar{u}_2\|_{L^p(\Omega\times(0,T))}\right)^{1/p}	+c_7\|\bar{u}_1-\bar{u}_2\|_{L^p(\Omega\times(0,T))}\\
&\leq c_{11}\left(\left (\|\bar{u}_1-\bar{u}_2\|_{C^0(\overline{\Omega}\times[0,T])}\right)^{1/p}	+\|\bar{u}_1-\bar{u}_2\|_{C^0(\overline{\Omega}\times[0,T])}\right).
\end{aligned}
\end{equation}
Furthermore, by the parabolic Sobolev embedding theorem, which serves as the parabolic counterpart of Morrey's inequality, we obtain
\[
W_p^{2,1}(\Omega \times (0, T)) \hookrightarrow C^{l_5,\frac{l_5}{2}}(\overline{\Omega} \times [0, T]) \hookrightarrow C^0(\overline{\Omega} \times [0, T]).
\]
for every  $\frac{d+2}{2}<p<d+2$ (see \cite[Chapter~II, p.~82]{Ladyzhynska}; cf. \cite[Chapter~5, p.~265]{evans}) and some $l_5\in(0,1)$. It follows that
\begin{align}\label{embedding}
	\|h\|_{C^0(\overline{\Omega}\times[0,T])} \leq c_{12}\|h\|_{C^{l_5,\frac{l_5}{2}}(\overline{\Omega}\times[0,T])} \leq c_{13}\|h\|_{W_p^{2,1}(\Omega\times(0,T))},\ \text{ where }\ p\in\left(\frac{d+2}{2},d+2\right).
\end{align}
Finally, combining \eqref{h in u in C0} with the embedding result \eqref{embedding}, and taking $p$ sufficiently large, we obtain, for the case $0<\beta<1$,
\begin{equation*}
	\begin{aligned}
		\|\phi(\bar{u}_1)-\phi(\bar{u}_2))\|_{C^0(\overline{\Omega}\times[0,T])}&=\|h\|_{C^0(\overline{\Omega}\times[0,T])}\\
		&\leq c_{14}\left(\left (\|\bar{u}_1-\bar{u}_2\|_{C^0(\overline{\Omega}\times[0,T])}\right)^{1/p}	+\|\bar{u}_1-\bar{u}_2\|_{C^0(\overline{\Omega}\times[0,T])}\right),
	\end{aligned}
\end{equation*}
where $c_{14}=c_{13}c_{11}.$ Thus, $\phi$ is continous for the case $0<\beta<1$.\\
\textit{Case 2: $\beta\geq1$.} Applying the inequality \eqref{ineq power p 2} from Lemma \ref{power inequality} to \eqref{v in w beta}, we directly establish
\[
\|v_1-v_2\|_{W^{2,p}(\Omega\times(0,T))}
\leq c_{15}\|w_1-w_2\|_{L^\infty(\Omega\times(0,T))}
\leq c_{16}\|\bar{u}_1-\bar{u}_2\|_{C^0(\overline{\Omega}\times[0,T])},
\]
where $c_{15}=\beta\sup_{x\in\Omega}|w_1(x)+w_2(x)|^{\beta-1}$. Again, by similar embedding argument in \eqref{embedding}, we arrive at 
\[
\|\phi(\bar{u}_1) - \phi(\bar{u}_2)\|_{C^0(\overline{\Omega}\times[0,T])}
= \|h\|_{C^0(\overline{\Omega}\times[0,T])}
\leq c_{17}\|\bar{u}_1 - \bar{u}_2\|_{C^0(\overline{\Omega}\times[0,T])},
\]
accordingly, $\phi$ is also continuous for the case $\beta\geq1$. 

Afterwards, for every $\bar{u}\in S$, there exist constants $c_{18}>0$ and
$l_6\in(0,1)$ such that
$
\|\phi(\bar{u})\|_{C^{l_6,\frac{l_6}{2}}(\overline{\Omega}\times[0,T])}
\le c_{18}.
$
Hence, by the compact embedding
$
C^{l_6,\frac{l_6}{2}}(\overline{\Omega}\times[0,T])
\hookrightarrow\hookrightarrow
C^0(\overline{\Omega}\times[0,T]),
$
it follows that $\phi(S)$ is relatively compact in
$X=C^0(\overline{\Omega}\times[0,T])$. Therefore, Schauder's fixed point
theorem guarantees the existence of a fixed point $u\in X$ satisfying
$\phi(u)=u$. Substituting $\bar{u}=u$ into \eqref{u}, we obtain the
corresponding solution of \eqref{Main} with $\gamma(\cdot)$ replaced by
$\widehat{\gamma}(\cdot)$. The associated function $w$ is then determined
uniquely from \eqref{w}, while $v$ is recovered from \eqref{voc of v}.
Furthermore, \eqref{banach_space} yields
$
u(x,t)\ge0
\ \text{ for all }(x,t)\in\overline{\Omega}\times[0,T].
$
Since \eqref{w} contains no spatial derivatives, its explicit representation,
together with the nonnegativity of $w_0$ and $u$, implies that
$
w(x,t)\ge0
\ \text{ in }\overline{\Omega}\times[0,T].
$
Consequently, the source term $w^\beta$ in \eqref{v} is nonnegative, and the
parabolic comparison principle ensures that
$
v(x,t)\ge0
\ \text{ in }\overline{\Omega}\times[0,T].
$
Moreover, since $0\le v\le M$, it follows from the definition
\eqref{gamma hat} that
$
\widehat{\gamma}(v)=\gamma(v).
$
Hence, the truncation is inactive along the obtained solution. Finally,
combining this observation with the parabolic regularity theory, we conclude
that $(u,v,w)$ is a nonnegative local classical solution of the original
system \eqref{Main} with the smooth initial data given in \eqref{eqn-smooth}.
\vskip 0.1cm 
\noindent 
\textbf{Step 4:} \emph{Existence of local maximal solutions to \eqref{Main}.} 
We denote the approximating solutions by
$(u^n,v^n,w^n)$,  where
$(u^n,v^n,w^n)$ is the local classical solution of \eqref{Main}
corresponding to the approximating initial data
$(u_0^n,v_0^n,w_0^n)$. Furthermore,
\[
0\leq u^n(x,t)\leq \|u_0^n\|_{L^\infty(\Omega)}+1
\leq \|u_0\|_{L^\infty(\Omega)}+1
=: \widetilde{M},
\]
where the constant $\widetilde{M}$ is independent of $n$. Moreover, we have the
uniform regularity estimates (\cite{Ladyzhynska})
\begin{align*}
u^n&\in C^{l_2,\,\frac{l_2}{2}}
(\overline{\Omega}\times[0,T]) \cap C^{2+l_4,\,1+\frac{l_4}{2}}
(\overline{\Omega}\times[\varepsilon,T]),\\
v^n&\in C^{1+l_1,\,\frac{1+l_1}{2}}
(\overline{\Omega}\times[0,T]) \cap C^{2+l_3,\,1+\frac{l_3}{2}}
(\overline{\Omega}\times[\varepsilon,T]),\\
w^n&\in C^{l_2,\,\frac{l_2}{2}}
(\overline{\Omega}\times[0,T]),
\end{align*}
for some $\varepsilon>0$,  where $l_i,l_j\in(0,1)$, $j=2$ and $i=3,4$, with bounds independent of $n$.
Therefore, by the compact embedding of parabolic H\"older spaces, there exists a subsequence, still denoted by $(u^n,v^n,w^n)$, and functions $(u,v,w)$ such that, for every $0<\ell_i<l_i$, $i=1,\ldots,4$,
\begin{align*}
	u^n &\rightarrow u
	\quad \text{strongly in} \quad
	C^{\ell_2,\frac{\ell_2}{2}}
	(\overline{\Omega}\times[0,T])
	\cap
	C^{2+\ell_4,\,1+\frac{\ell_4}{2}}
	(\overline{\Omega}\times[\varepsilon,T]),\\[1ex]
	v^n &\rightarrow v
	\quad \text{strongly in} \quad
	C^{1+\ell_1,\frac{1+\ell_1}{2}}
	(\overline{\Omega}\times[0,T])
	\cap
	C^{2+\ell_3,\,1+\frac{\ell_3}{2}}
	(\overline{\Omega}\times[\varepsilon,T]),\\[1ex]
	w^n &\rightarrow w
	\quad \text{strongly in} \quad
	C^{\ell_2,\frac{\ell_2}{2}}
	(\overline{\Omega}\times[0,T]).
\end{align*}
Consequently, all nonlinear terms converge strongly, allowing us to pass to the limit in the equations and the homogeneous Neumann boundary conditions. Moreover, since the convergence holds up to $t=0$ in the lower parabolic H\"older spaces,
\[
u(\cdot,0)=u_0,\
v(\cdot,0)=v_0,\
w(\cdot,0)=w_0.
\]
Hence, $(u,v,w)$ is a local solution of \eqref{Main} corresponding to the initial data
\[
(u_0,v_0,w_0)\in
C(\overline{\Omega})
\times
C^1(\overline{\Omega})
\times
C^1(\overline{\Omega}),
\]
with the regularity given in \eqref{eqn-regularity}. 

Since the choice of $T$ depends only on
$M=\|u_0\|_{L^\infty(\Omega)}+1$ and the prescribed initial data in
\eqref{Initial Cond}, rather than on the initial time, the local
construction may be repeated from any $t_0$ for which the solution
exists. More precisely,
$
\bigl(u(\cdot,t_0),v(\cdot,t_0),w(\cdot,t_0)\bigr)
$
can be taken as new initial data, since they preserve the regularity
required in \eqref{Initial Cond} as long as $\|u(\cdot,t_0)\|_{L^\infty(\Omega)}$ remains finite. Repeating this continuation procedure yields a maximal
existence interval $(0,T_{\max})$, with
$T_{\max}\in(0,\infty]$. The standard continuation argument then gives
the extensibility criterion stated in \eqref{blow up crit}. 
\end{proof}
Having established the local existence of classical solutions satisfying \eqref{eqn-regularity}, we now turn to their uniqueness. The previous theorem guarantees the existence of at least one local classical solution for lower regularity initial data satisfying 
\[
u_0\in C^0(\overline{\Omega}),\
v_0,w_0\in C^1(\overline{\Omega}),
\]
consequently, at least one local classical solution exists satisfying \eqref{eqn-regularity}.
Moreover, the classical parabolic regularity ensures that, for every
\(\varepsilon>0\),
\begin{align}\label{classical parab. regularity}
u\in
C^{2+l_4,\,1+\frac{l_4}{2}}
(\overline{\Omega}\times[\varepsilon,T_{\max})),
\
v\in
C^{2+l_3,\,1+\frac{l_3}{2}}
(\overline{\Omega}\times[\varepsilon,T_{\max})),
\end{align}
for some \(l_3,l_4\in(0,1)\). 

\begin{remark}
	The initial data are regularized according to \eqref{eqn-smooth} in order to invoke the abstract parabolic theory \cite[Theorem~6.1, Chapter~V, p.~452]{Ladyzhynska}. The resulting regularity of $u$ and $v$ plays a key role in showing the continuity of the associated fixed-point map, thereby allowing the Schauder fixed-point theorem to be applied to construct local classical solutions with the regularity stated in \eqref{eqn-regularity}. In the absence of this regularization, the classical parabolic regularity in \eqref{classical parab. regularity} alone does not yield the $L^\infty(\Omega\times(0,T))$ bound required for $g_1$. Consequently, the parabolic $L^p$-regularity estimate in \eqref{h-g_2 est} cannot be applied directly.
\end{remark}
We shall now show that, under a stronger regularity assumption on the initial data given in \eqref{eqn-smooth}, the corresponding local classical solution satisfying \eqref{eqn-regularity} is unique.
	
\begin{theorem}\label{Uniqueness of local classical solutions}
Let $\Omega\subset\mathbb{R}^d$ $(d\ge2)$ be a bounded domain with a smooth boundary. Let $\mu,\beta,\delta>0$, $r\in\mathbb{R}$, and $\theta\ge1$. Suppose that the initial data satisfy  \eqref{initial-data-reg}. 
Moreover, assume that $w_0>0$ when $\beta\in(0,1)$, while $w_0\ge0$ when $\beta\in[1,\infty)$. Finally, suppose that the motility function satisfies \eqref{Hypothesis H1}.
Then any local classical solution $(u,v,w)$ of \eqref{Main} satisfying
\begin{align}\label{regularity at 0}
\left\{
\begin{aligned}
	u &\in
	C^{l_2,\frac{l_2}{2}}
	(\overline{\Omega}\times[0,T_{\max}))
	\cap
	C^{2+l_4,\,1+\frac{l_4}{2}}
	(\overline{\Omega}\times[0,T_{\max})),\\
	v &\in
	C^{1+l_1,\frac{1+l_1}{2}}
	(\overline{\Omega}\times[0,T_{\max}))
	\cap
	C^{2+l_3,\,1+\frac{l_3}{2}}
	(\overline{\Omega}\times[0,T_{\max})),\\
	w &\in
	C^{l_2,\frac{l_2}{2}}
	(\overline{\Omega}\times[0,T_{\max})),
\end{aligned}
\right.
\end{align}
is unique.
\end{theorem}
\begin{proof}
Following the approach of \cite{SIAM,unique2}, we establish the uniqueness of  local solutions. Fix $t_0\in(0,T_{\max})$, and let $(u_1,v_1,w_1)$ and $(u_2,v_2,w_2)$ be two solutions of \eqref{Main} in $\Omega\times(0,T_{\max})$. Define
\begin{align*}
	U = u_1 - u_2,\ V = v_1 - v_2,\ \text{and}\ W = w_1 - w_2,
\end{align*}
and
\begin{align}\label{IC of U,V,W}
	U(0,x)=0,\ V(0,x)=0,\ \text{and}\ W(0,x)=0, \ x\in\Omega. 
\end{align}
In view of the regularity of the solutions $(u_i,v_i,w_i)$, $i=1,2$ in \eqref{regularity at 0}, and the assumptions imposed on $\gamma$ in \eqref{Hypothesis H1}, it follows that
\begin{align}\label{regularity assumptions}
	\|u_1\|_{W^{1,\infty}} + \|u_2\|_{W^{1,\infty}}+ \|v_1\|_{W^{1,\infty}} + \|v_2\|_{W^{1,\infty}} + \|w_1\|_{L^\infty} + \|w_2\|_{L^\infty} \le C_{1},
\end{align}
with
\begin{align}\label{more assump on gamma}
	\gamma(v_1) \ge C_{2} > 0,\ |\gamma'(v_1)| + |\gamma'(v_2)| \le C_{3},\ |\gamma(v_1) - \gamma(v_2)| + |\gamma'(v_1) - \gamma'(v_2)| \le C_{4}|V|. 
\end{align}
  Furthermore, from the system \eqref{Main}, we see that
\begin{align}
		U_t &= \nabla \cdot (\gamma(v_1)\nabla U) + \nabla \cdot ((\gamma(v_1) - \gamma(v_2))\nabla u_2) + \nabla \cdot (\gamma'(v_1)u_1\nabla v_1 - \gamma'(v_2)u_2\nabla v_2)\nonumber\\
		&\quad +\mu\left((u_1 + u_2)U - AU - (u_1^{\theta+2}-u_2^{\theta+2}) + A(u_1^{\theta+1}-u_2^{\theta+1})\right), \label{U eq}
 \\
	V_t &= \Delta V - V + w_1^\beta - w_2^\beta, \label{V equation}
	 \\
	W_t &= -\delta W + U. \label{W equation}
\end{align}
The following result will be used in the sequel.
\begin{align*}
	(u_1^q-u_2^q)=\int_0^1\frac{d}{d\eta}\left(\eta u_1+(1-\eta)u_2\right)^qd\eta =q\int_0^1 \left(\eta u_1+(1-\eta)u_2\right)^{q-1}d\eta(u_1-u_2),
\end{align*}
for all $q\in[1,\infty)$. 

\vskip 0.1cm
\noindent 
\textbf{Step 1:} \emph{A priori estimates for  $U$.}
Multiplying \eqref{U eq} by $U$ and integrating it over spatial domain $\Omega$, we get
\begin{equation*}
	\begin{aligned}
		&\frac{1}{2}\frac{d}{dt}\int_\Omega U^2 + \int_\Omega \gamma(v_1)|\nabla U|^2+(\theta+2) \int_{\Omega}\left(\int_{0}^{1}\left(\eta u_1+(1-\eta)u_2\right)^{\theta+1}\right)U^2+AU^2\nonumber\\&= - \int_\Omega (\gamma(v_1) - \gamma(v_2))\nabla u_2 \cdot \nabla U- \int_{\Omega} (\gamma'(v_1)u_1\nabla v_1 - \gamma'(v_2)u_2\nabla v_2)\cdot \nabla U\\
		&\quad+ \mu \int_\Omega (u_1 + u_2)U^2+ A(\theta+1)\int_{\Omega}\left(\int_{0}^{1}\left(\eta u_1+(1-\eta)u_2\right)^{\theta}\right)U^2\\
		&= I_1 + I_2 + I_3.
	\end{aligned}
\end{equation*}
Then, by using the assumption  $\gamma(v_1) \ge C_{2} > 0$ from \eqref{more assump on gamma}, it entails for all $t \in (0,t_0)$ that 
\begin{align}\label{derivative of U}
	\frac{1}{2}\frac{d}{dt}\int_{\Omega} U^2 + C_{2}\int_{\Omega} |\nabla U|^2 \le I_1 + I_2 + I_3. 
\end{align}
Under the assumptions \eqref{regularity assumptions} and \eqref{more assump on gamma}, an application of Hölder's and Young's inequalities, with $\epsilon=\frac{C_{2}}{8}$, yields, for every $t\in(0,t_0)$,
\begin{align}\label{I1}
		I_1 &= - \int_{\Omega} (\gamma(v_1) - \gamma(v_2))\nabla u_2 \cdot \nabla U \nonumber\\
		&\le C_{1} \int_{\Omega} |\gamma(v_1) - \gamma(v_2)| |\nabla U| \nonumber\\
		&\le C_{1}C_4 \int_{\Omega} |V||\nabla U| \le \frac{C_{2}}{8} \|\nabla U\|_{L^2}^2 + \frac{2C_1^2 C_4^2}{C_{2}} \|V\|_{L^2}^2.
	\end{align}
 Similarly, the terms $I_2$ and $I_3$ can be estimated as 
\begin{align}
		I_2 &= - \int_{\Omega} (\gamma'(v_1)u_1\nabla v_1 - \gamma'(v_2)u_2\nabla v_2)\cdot \nabla U\nonumber\\
		&=- \int_{\Omega} \gamma'(v_1)u_1 \nabla V \cdot \nabla U 
		- \int_{\Omega} (\gamma'(v_1) - \gamma'(v_2))u_1 \nabla v_2 \cdot \nabla U- \int_{\Omega} \gamma'(v_2)U \nabla v_2 \cdot \nabla U\nonumber\\
		&\le C_1C_{3} \int_{\Omega} |\nabla V||\nabla U|
		+ C_1^2 C_4 \int_{\Omega} |V||\nabla U|
		+ C_1C_{3} \int_{\Omega} |U||\nabla U|\nonumber\\
		&\le \frac{C_{2}}{8} \|\nabla U\|_{L^2}^2
		+ \frac{2C_{3}^2 C_1^2}{C_{2}} \|\nabla V\|_{L^2}^2
		+ \frac{C_{2}}{8} \|\nabla U\|_{L^2}^2
		+ \frac{2C_4^2 C_1^4}{C_{2}} \|V\|_{L^2}^2
		+ \frac{C_{2}}{8} \|\nabla U\|_{L^2}^2
		+ \frac{2C_{3}^2 C_1^2}{C_{2}} \|U\|_{L^2}^2\nonumber\\
		&\le \frac{3C_{2}}{8} \|\nabla U\|_{L^2}^2
		+ \frac{2C_{3}^2 C_1^2}{C_{2}} \left( \|U\|_{L^2}^2 + \|\nabla V\|_{L^2}^2 \right)
		+ \frac{2C_4^2 C_1^4}{C_{2}} \|V\|_{L^2}^2,
	\end{align}
and
\begin{align}\label{I2}
		I_3 &= \mu \int_\Omega (u_1 + u_2)U^2+ A(\theta+1)\int_{\Omega}\left(\int_{0}^{1}\left(\eta u_1+(1-\eta)u_2\right)^{\theta}\right)U^2\nonumber\\
		&\le \mu\left(\left(\|u_1\|_{L^\infty} + \|u_2\|_{L^\infty}\right) + A(\theta+1)\left(\|u_1\|_{L^\infty} + \|u_2\|_{L^\infty}\right)^{\theta}\right) \|U\|_{L^2}^2\nonumber\\
		&\le C_{5} \|U\|_{L^2}^2,
\end{align}
where $C_{5}=\mu(C_1 + A(\theta+1)C^{\theta}_1), \theta \geq 1$, and $A\in(0,1).$
On substituting estimates of $I_1, I_2$ and $I_3$ in equation \eqref{derivative of U}, it follows that, for every $t\in(0,t_0)$,
\begin{equation}\label{I1I2I3}
	\begin{aligned}
\frac{d}{dt}\|U\|_{L^2}^2 + C_{2}\|\nabla U\|_{L^2}^2
&\le C_{6} (\|U\|_{L^2}^2 + \|\nabla V\|_{L^2}^2)+ C_{7}\|V\|_{L^2}^2, 
	\end{aligned}
\end{equation}
where
\[
C_{6}=2\max\left\{C_{5},\,\frac{2C_{3}^2C_1^2}{C_{2}}\right\},\
C_{7}=\frac{4C_1^2C_4^2(C_1^2+1)}{C_{2}},
\]
are positive constants.

\vskip 0.1cm
\noindent 
\textbf{Step 2:} \emph{A priori estimates for  $W$ and $V$.}
Testing equation \eqref{W equation} with $W$, we obtain
\[
\frac{1}{2}\frac{d}{dt}\int_{\Omega} |W|^2 + \delta \int_{\Omega} W^2 \le \int_{\Omega} WU\le \frac{1}{2}\int_{\Omega} W^2 + \frac{1}{2}\int_{\Omega} U^2,
\]
\begin{align}\label{derivative of W}
	\frac{d}{dt}\|W\|_{L^2}^2 \le \|U\|_{L^2}^2 + \|W\|_{L^2}^2.
\end{align}
Next, we multiply equation \eqref{V equation} by $V_t$ and integrate over $\Omega$ to find 
\begin{equation*}
	\frac{1}{2}\frac{d}{dt}\left( \int_{\Omega} |V|^2 + \int_{\Omega} |\nabla V|^2 \right) + \int_{\Omega} V_t^2
\le \frac{1}{2}\int_{\Omega} V_t^2 + \frac{1}{2}\int_{\Omega} (w_1^\beta - w_2^\beta)^2,
\end{equation*}
which implies 
\begin{align}\label{deri of V without W}
	\frac{d}{dt}\left( \int_{\Omega} |V|^2 + \int_{\Omega} |\nabla V|^2 \right)+\int_{\Omega}{V_t}^2\leq \int_{\Omega} (w_1^\beta - w_2^\beta)^2,		
\end{align}
for all $t\in (0,t_0).$ We now distinguish two cases, as in the proof of local existence of solutions. 

\emph{Case (i)}: $0 < \beta < 1$ with assumption $w_0>0$ on $\overline{\Omega}$.
Since $w_0\in C^1(\overline{\Omega})$, we have
\[C_0 := \min_{\overline{\Omega}} w_0 > 0.\]
Using the explicit representation of $w$, we have 
\begin{equation*}
	w_i(x,t) = e^{-\delta t} w_0(x) + \int_0^t e^{-\delta(t-s)} u_i(x,s)\,ds
	\;\ge\; e^{-\delta t} w_0(x)
	\;\ge\; e^{-\delta t_0} C_0
	\;=\; C_8 \;>\; 0,
	\label{eq:w_lower_bound}
\end{equation*}
for all $(x,t)\in\overline{\Omega}\times[0,t_0]$ and $i=1,2$. Moreover, by \eqref{regularity assumptions}, we already have $\|w_i\|_{L^\infty}\leq C_1$. By the mean value theorem, there exists $\zeta=\zeta(x,t)$ between $w_1(x,t)$ and $w_2(x,t)$ such that
\[
|w_1^\beta(x,t)-w_2^\beta(x,t)|
=\beta \zeta^{\beta-1}|W(x,t)|.
\]
Since $\zeta\geq C_8$ and $0<\beta<1$, the function $s\mapsto s^{\beta-1}$ is decreasing on $[C_8,C_1]$. Hence,
\begin{align}\label{case 1 w}
	|w^{\beta}_1-w^{\beta}_2|&=\beta    \zeta^{\beta-1}|W|\leq \beta C_8^{\beta-1}|W|\leq C|W|.
\end{align}
Hence, combining \eqref{I1I2I3}, \eqref{derivative of W}, \eqref{deri of V without W} and \eqref{case 1 w}, it follows that
\begin{align}\label{uniqueness ineq}
\frac{d}{dt}\left(\|U\|_{L^2}^2+\|V\|_{L^2}^2+\|\nabla V\|_{L^2}^2+\|W\|_{L^2}^2\right)
\leq C_{9}\left(\|U\|_{L^2}^2+\|V\|_{L^2}^2+\|\nabla V\|_{L^2}^2+\|W\|_{L^2}^2\right),
\end{align}
for all $t\in (0,t_0)$, where
$
C_{9}=\max\{C_6+1, C_7, C^2+1\}.
$

\emph{Case (ii)}: $\beta\geq 1$ with $w_0\geq 0$ on $\overline{\Omega}$. From the equation \eqref{deri of V without W}, we have
\[
\frac{d}{dt}\left(\int_\Omega |V|^2+\int_\Omega |\nabla V|^2\right)
+\int_\Omega V_t^2
\leq  \int_\Omega (w_1^\beta-w_2^\beta)^2,
\ \text{ for all } \ t\in (0,t_0).
\]
Invoking inequality \eqref{ineq power p 2} from Lemma \ref{power inequality}, we express the difference $w_1^\beta-w_2^\beta$ in the form
\[
|w_1^\beta-w_2^\beta|
\leq \beta(w_1+w_2)^{\beta-1}|w_1-w_2|,
\]
with $\beta\geq1$, moreover the regularity assumptions on $w_i$ in \eqref{regularity assumptions} further gives
\begin{align}\label{wi in W}
	|w_1^\beta-w_2^\beta|\leq C|w_1-w_2|,
\end{align}
where $C=\beta(C_1)^{\beta-1}$. Now, using \eqref{wi in W} in \eqref{deri of V without W}, we obtain
\begin{align}\label{deri of V without nonlinearity}
	\frac{d}{dt}\left(\|V\|_{L^2}^2+\|\nabla V\|_{L^2}^2\right)
	\leq C\|W\|_{L^2}^2,
	\ \text{ for all } \ t\in (0,t_0).
\end{align}
Using \eqref{I1I2I3}, \eqref{derivative of W}, and \eqref{deri of V without nonlinearity} once again, we reach at \eqref{uniqueness ineq}.
Finally, applying Gr\"onwall's inequality together with the initial conditions \eqref{IC of U,V,W}, we conclude that $U\equiv 0$, $V\equiv 0$, and $W\equiv 0$ for all $t_0\in(0,T)$. Since $t_0$ was chosen arbitrarily, the uniqueness of the local solution follows. This completes the proof.
\end{proof}

\section{Boundedness under hypothesis \eqref{Hypothesis H1} on the motility function}\label{sec:global1} 
We now establish a collection of a priori estimates that are essential for the
proof of Theorem~\ref{global bound_H1}. More precisely, we derive
$L^q$-estimates for $\nabla v$, $L^\infty$-estimates for $v$, and
$L^p$-estimates for $u$, which ultimately yield the boundedness of
$\|u\|_{L^\infty(\Omega)}$ for all $t\in(0,T_{\max})$ via the
Alikakos-Moser iteration (\cite{Moser}). We begin by deriving uniform $L^1$ and
spatio-temporal estimates for $u$.
	\begin{lemma}
		Under the assumption \eqref{Initial Cond}, there exist some constants $m^*>0$ and $C>0$ such that
		\begin{equation}\label{L1 estimate of u}
			\int_{\Omega} u \le m^*,
			\  \text{ for all }\ t \in (0,T_{\max})
		\end{equation}
		and
		\begin{equation}\label{dobuble integral bound u}
			\int_t^{t+\tau} \int_{\Omega} u^{\theta+2} \le C,
			\  \text{ for all } \ t \in (0,T_{\max}-\tau),
		\end{equation}
		where $\tau<T_{\max}$.
	\end{lemma}
	\begin{proof}
		Integrating the first equation of \eqref{Main}, we obtain 
		\begin{align*}		
			\frac{d}{dt} \int_{\Omega} u
			&= \int_{\Omega} \Delta\big(\gamma(v)u\big)
			+ \mu \int_{\Omega} u(1-u^\theta)(u-A)
			\\[6pt]
			&= \int_{\partial\Omega} \nabla\big(\gamma(v)u\big)\cdot \nu \, ds
			+ \mu \int_{\Omega} \left(-u^{\theta+2} + Au^{\theta+1} + u^2 - Au\right)
			\\[6pt]
			&= \int_{\partial\Omega}
			\left(\gamma(v)\nabla u \cdot \nu
			+ \gamma'(v)\nabla v \cdot \nu \, u \right) ds
			+ \mu \int_{\Omega} \left(-u^{\theta+2} + Au^{\theta+1} + u^2 - Au\right)
			\\[6pt]
			&= \int_{\partial\Omega}
			\left(
			\gamma(v)\frac{\partial u}{\partial \nu}
			+ \gamma'(v)u \frac{\partial v}{\partial \nu}
			\right) ds
			+ \mu \int_{\Omega} \left(-u^{\theta+2} + Au^{\theta+1} + u^2 - Au\right).
		\end{align*}
		Using the boundary conditions from \eqref{Main}, we deduce
		\begin{align}\label{L1 without Youngs}
			\frac{d}{dt} \int_{\Omega} u + \mu \int_{\Omega}u^{\theta+2} + \mu A \int_{\Omega}u =  \mu A \int_{\Omega}u^{\theta+1}+\mu \int_{\Omega} u^2.
		\end{align}
		By invoking the Young's inequality to estimate the terms on the right-hand side with the choice $\epsilon=\frac{\mu}{4}$, we infer
		\begin{align*}
			A\mu\int_{\Omega}u^{\theta+1}&\leq\left(\int_\Omega u^{\theta+2}\right)^\frac{\theta+1}{\theta+2}\left(\int_{\Omega}(A\mu)^{\theta+2}\right)^\frac{1}{\theta+2}\\
			&\leq \frac{\mu}{4} \int_{\Omega}u^{\theta+2}+\frac{|\Omega|(A\mu)^{\theta+2}}{\theta+2}\left(\frac{4 (\theta+1)}{\mu(\theta+2)}\right)^{\theta+1},
		\end{align*}
		and
		\begin{align*}
			\mu\int_{\Omega}u^2&\leq \left(\int_{\Omega}u^{\theta+2}\right)^{\frac{2}{\theta+2}}\left(\int_{\Omega}\mu^{\frac{\theta+2}{\theta}}\right)^{\frac{\theta}{\theta+2}}\\
			&\leq \frac{\mu}{4}\int_{\Omega}u^{\theta+2}+\frac{|\Omega|\theta\mu^{\frac{\theta+2}{\theta}}}{\theta+2}\left(\frac{8}{\mu(\theta+2)}\right)^{\frac{2}{\theta}}.
		\end{align*}
		By using the above estimates, we can write \eqref{L1 without Youngs} as
		\begin{equation}\label{Holder-2}
			\begin{aligned}
				\frac{d}{dt} \int_{\Omega} u + \frac{\mu}{2}\int_{\Omega}u^{\theta+2} + \mu A \int_{\Omega}u &\leq C^*(\mu,\theta, \abs{\Omega},A).
			\end{aligned}
		\end{equation}
	Now, applying Lemma \ref{ode comparison} to \eqref{Holder-2}, we derive the estimate \eqref{L1 estimate of u}. Subsequently, integrating \eqref{Holder-2} with respect to $s$ over the interval $(t,t+\tau)$ yields
	 
		\begin{equation}\label{double integral u inequality}
			\begin{aligned}
				\int_{t}^{t+\tau}\frac{d}{dt} \int_{\Omega} u + \frac{\mu}{2}	\int_{t}^{t+\tau}\int_{\Omega}u^{\theta+2} &\leq C^*\tau,\  \text{ for all } \ t \in (0,T_{\max}-\tau),
				\\[6 pt]
				\int_\Omega u(\cdot,t+\tau)+ \frac{\mu}{2}	\int_{t}^{t+\tau}\int_{\Omega}u^{\theta+2}&\leq C^*\tau +\int_{\Omega} u(\cdot,t), \  \text{ for all } \ t \in (0,T_{\max}-\tau).
			\end{aligned}
		\end{equation}
		Combining \eqref{double integral u inequality} with \eqref{L1 estimate of u}, we obtain the desired estimate \eqref{dobuble integral bound u}.
		
\end{proof}
Further, we establish the $L^{\theta+2}$-estimate for $w$. Moreover, for $\beta \leq \theta+2$, we derive the $L^1$-estimate for $v$ in the following lemma:

\begin{lemma}\label{L1 estimate of w theta}
	   Under the assumption \eqref{Initial Cond}, there exists a positive constant $C$ such that
	   \begin{equation}\label{w estimate}
	   	\begin{aligned}
	   		\int_{\Omega} w^{\theta+2}(\cdot,t)\leq C, \ \text{for all}\; t\in(0,T_{\max}).
	   	\end{aligned}
	   \end{equation}
	   Furthermore, provided that $\beta\leq\theta+2$, the following estimate holds:
	   \begin{equation}\label{L1 estimate of v}
	   	\begin{aligned}
	   		\int_{\Omega} v(\cdot,t)\leq C, \ \text{for all}\; t\in(0,T_{\max}).
	    		\end{aligned}
	    	\end{equation}
		\end{lemma}
	\begin{proof}
		Upon multiplying the third equation in \eqref{Main} by $(\theta+2)w^{\theta+1}$ and employing Young's inequality, we obtain 
		\begin{equation*}
				\begin{aligned}
				\frac{d}{dt} \int_{\Omega} w^{\theta+2} + (\theta+2)\delta \int_{\Omega}w^{\theta+2} &= (\theta+2) \int_{\Omega} u w^{\theta+1}\\
				&\leq (\theta+2)\left[\left(\int_{\Omega}w^{\theta+2}\right)^{\frac{\theta+1}{\theta+2}} \left(\int_{\Omega}u^{\theta+2}\right)^{\frac{1}{\theta+2}}\right]\\
				&\leq \frac{(\theta+2)\delta}{2}\int_{\Omega}w^{\theta+2}+\left[\frac{2(\theta+1)}{\delta(\theta+2)}\right]^{\theta+1}\int_{\Omega}u^{\theta+2},
			\end{aligned}
		\end{equation*} 
		which gives
		 \begin{equation*}
		 	\begin{aligned}
		 		\frac{d}{dt} \int_{\Omega} w^{\theta+2} + \frac{(\theta+2)\delta}{2} \int_{\Omega}w^{\theta+2}&\leq C(\theta,\delta) \int_{\Omega} u^{\theta+2}.
		 	\end{aligned}
		 \end{equation*}
	In view of \eqref{dobuble integral bound u}, which ensures that $u\in L^{\theta+2}((t,t+\tau)\times\Omega)$, and by applying Lemma~\ref{ode comparison}, we obtain \eqref{w estimate}.
	
	 In the case $\beta=\theta+2$, integrating the second equation in \eqref{Main} with respect to the spatial variable $x\in\Omega$, we obtain
		\begin{equation}\label{L1 estimate of v 1}
			\begin{aligned}
				\frac{d}{dt} \int_{\Omega} v + \int_{\Omega}v = \int_{\Omega} w^{\theta+2}, 
			\end{aligned}
		\end{equation}
		 for all $ t\in(0,T_{\max})$. If $\beta<\theta+2$, a similar calculation by using Young's inequality provides us to select a positive constant $c_1$ satisfying 
			\begin{equation}\label{L1 estimate of v 2}
			\begin{aligned}
				\frac{d}{dt} \int_{\Omega} v + \int_{\Omega}v = \int_{\Omega} w^{\beta}\leq \int_{\Omega} w^{\theta+2}+c_1, 
			\end{aligned}
		\end{equation}
	 for all $t\in(0,T_{\max})$. Employing \eqref{w estimate}, together within \eqref{L1 estimate of v 1} and \eqref{L1 estimate of v 2}, and invoking Lemma~\ref{ode comparison}, we deduce the estimate \eqref{L1 estimate of v}.
	 \end{proof}
The $L^1$-bound obtained in Lemma~\ref{L1 estimate of v} provides the necessary groundwork for establishing the boundedness of $v$. We next derive the required uniform bounds for $v$.
\begin{lemma}\label{Lq norm of v and grad v}
		Assume that \eqref{Initial Cond} and \eqref{Hypothesis H1} hold. If
		$\beta < \frac{2(\theta+2)}{d}$, then there exist $q>d$ and $C>0$ such that
		\begin{equation}\label{Lq estimate grad v}
			\|\nabla v(\cdot,t)\|_{L^q(\Omega)} \le C,
		\end{equation}
		and
		\begin{equation}\label{Linfty est of v}
			\|v(\cdot,t)\|_{L^\infty(\Omega)} \le C,
		\end{equation}
		for every $t \in (0,T_{\max})$.
		\end{lemma}
		
		\begin{proof}
			From \eqref{beta cond for grad v} and \eqref{grad v 1st cond}, choosing $k=\theta+2,$ and $\beta < \frac{\theta+2}{d}$, we have
			\begin{equation}\label{Linfty of grad v}
				\|\nabla v(\cdot,t)\|_{L^\infty(\Omega)} \le c_1 \ \text{ for all}\ t\in (0,T_{\max}),
			\end{equation}
		 with $c_1>0$. Moreover, for $\beta \ge \frac{\theta+2}{d},$ we observe 
		\[	\|\nabla v(\cdot,t)\|_{L^{q}(\Omega)}
		\le C \left( 1 + \sup_{0<t<T_{\max}} \|w(\cdot,t)\|_{L^{\theta+2}(\Omega)}^\beta \right).\]	
	Now, to obtain such $q>d$, from \eqref{beta cond for grad v}, it must satisfy the condition $d<\frac{d(\theta+2)}{\beta d-(\theta+2)},$ which is ensured by $\beta<\frac{2(\theta+2)}{d},$ therefore we can pick such $q>d$ and $c_2>0$ such that
			\begin{align}\label{grad v with Lq norm}
			\|\nabla v(\cdot,t)\|_{L^{q}(\Omega)} \le c_2,
			\end{align}
			for all $t \in (0,T_{\max})$, which yields \eqref{Lq estimate grad v}. Combining the Gagliardo-Nirenberg inequality \eqref{GNI} with estimates \eqref{L1 estimate of v} and \eqref{Lq estimate grad v} yields
			\[
			\|v(\cdot,t)\|_{L^q(\Omega)}
			\le c_3 
			\|\nabla v(\cdot,t)\|_{L^q(\Omega)}^{
				\frac{1-\frac{1}{q}}{1+\frac{1}{d}-\frac{1}{q}}
			}
			\|v(\cdot,t)\|_{L^1(\Omega)}^{
				\frac{\frac{1}{d}}{1+\frac{1}{d}-\frac{1}{q}}
			}
			+ c_3 \|v(\cdot,t)\|_{L^1(\Omega)}
			\le c_4,
			\]
	for all $t \in (0,T_{\max})$, where $c_3, c_4$ denote positive constants. Together with \eqref{grad v with Lq norm}, this allows us to choose a constant $c_5>0$ such that
			\[
			\|v(\cdot,t)\|_{W^{1,q}(\Omega)} \le c_5,
			\]
			for all $t \in (0,T_{\max})$.
			According to the fact that $q>d$, we have \eqref{Linfty est of v} by using the Sobolev embedding \cite[§5.6.2, Theorem 6 (Morrey's inequality), and §5.6.3 (General Sobolev inequalities)]{evans}.
		\end{proof}
We now apply Lemma~\ref{Lq norm of v and grad v} to derive the desired $L^p$-bound for $u$.
\begin{lemma}\label{Lp estimate of u 1}
		Assume that \eqref{Initial Cond} and \eqref{Hypothesis H1} are satisfied. If
		$\beta<\frac{2(\theta+2)}{d}$, then, for each $p>1$, there exists a positive constant $C$ such that
		\[
		\|u(\cdot,t)\|_{L^p(\Omega)} \le C,\quad \text{for all}\; t\in(0,T_{\max}).
		\]
\end{lemma}
		\begin{proof}
		Owing to \eqref{Hypothesis H1} and \eqref{Linfty est of v}, we know that $v\in L^\infty(\Omega\times(0,T_{\max}))$. Hence, there exists a constant $K>0$ such that $0\leq v(x,t)\leq K$ for all $(x,t)\in\Omega\times(0,T_{\max})$. Since $\gamma\in C^3([0,\infty))$, both $\gamma$ and $\gamma'$ are continuous on the compact interval $[0,K]$, and consequently, they are bounded on $[0,K]$. Hence there exist positive constants $k_1,k_2>0$ such that 
			
			\begin{equation}\label{k1 lower bound on gamma}
				\gamma(v(\cdot,t))\geq k_1, \ \text{for all}\; t\in(0,T_{\max}),
			\end{equation}
			and
			\begin{equation}\label{bound on gamma'}
				\abs{\gamma'(v(\cdot,t))}\leq k_2, \ \text{for all}\; t\in(0,T_{\max}).
			\end{equation}
			Testing the first equation of \eqref{Main} by $u^{p-1}$ and integrating over the spatial variable $x\in \Omega$, it gives 
			
			\begin{align}\label{testing with u powerp-1}
					\frac{1}{p}\frac{d}{dt}\int_{\Omega} u^p 
					&= -(p-1)\int_{\Omega} u^{p-2}\gamma(v)|\nabla u|^2-(p-1)\int_{\Omega} u^{p-1}\gamma'(v)\nabla u \cdot \nabla v\nonumber\\
					&\quad+\mu \int_{\Omega}(-u^{\theta+p+1}+Au^{\theta+p}-Au^p+u^{p+1}).
			\end{align}
			Using \eqref{k1 lower bound on gamma} and \eqref{bound on gamma'}, we obtain
			\begin{align*}
						\frac{1}{p}\frac{d}{dt}\int_{\Omega} u^p&+\mu\int_{\Omega}u^{\theta+p+1}+ \mu A\int_{\Omega}u^p
						+(p-1)k_1\int_{\Omega}u^{p-2}|\nabla u|^2\nonumber\\&\leq k_2(p-1)\int_{\Omega}u^{p-1}\abs{\nabla u}\abs{\nabla v}+ \mu\int_{\Omega}(Au^{\theta+p}+u^{p+1}).
				\end{align*}
			By using Young's inequality, this leads to
			\[\mu \int_{\Omega} u^{p+1} \leq \left(\int_{\Omega} u^{\theta+p+1}\right)^{\frac{p+1}{\theta+p+1}} \left(\int_{\Omega} \mu^{\frac{\theta+p+1}{\theta}}\right)^{\frac{\theta}{\theta+p+1}}
			\leq \frac{\mu}{4} \int_{\Omega} u^{\theta+p+1} + {C_1}|\Omega|\mu^{\frac{\theta+p+1}{\theta}},\] 
			and
			\[A\mu \int_{\Omega} u^{\theta+p} \leq \left(\int_{\Omega} u^{\theta+p+1}\right)^{\frac{\theta+p}{\theta+p+1}} \left(\int_{\Omega} (A\mu)^{\theta+p+1}\right)^{\frac{1}{\theta+p+1}}
			\leq \frac{\mu}{4} \int_{\Omega} u^{\theta+p+1} + {C_2}|\Omega|(A\mu)^{\theta+p+1},\]
			where ${C_1}=\left(\frac{\theta}{\theta+p+1}\right)\left(\frac{4(p+1)}{\mu(\theta+p+1)}\right)^\frac{p+1}{\theta},$	and ${C_2}=\left(\frac{1}{\theta+p+1}\right)\left(\frac{4(\theta+p)}{\mu(\theta+p+1)}\right)^{\theta+p}.$ Furthermore, the above estimates yield
			\begin{equation}\label{3.40}
				\begin{aligned}
					\frac{1}{p}\frac{d}{dt}\int_{\Omega} u^p+\frac{\mu}{2}\int_{\Omega}u^{\theta+p+1}&+\mu A\int_{\Omega}u^p
					+(p-1)k_1\int_{\Omega}u^{p-2}|\nabla u|^2\\
					&\leq k_2(p-1)\int_{\Omega}u^{p-1}\abs{\nabla u}\abs{\nabla v}+{C_3},
				\end{aligned}
			\end{equation}
		where ${C_3}={C_1}|\Omega|\mu^{\frac{\theta+p+1}{\theta}}+ {C_2}|\Omega|(A\mu)^{\theta+p+1}.$ Now, consider the right hand side term and estimate it as follows:
		\begin{equation}\label{3.41}
		\begin{aligned}
				k_2(p-1)\int_{\Omega} u^{p-1}|\nabla u||\nabla v|
				&= k_2(p-1)\int_{\Omega} u^{\frac{p-2}{2}}|\nabla u|\,|\nabla v|\,u^{\frac{p}{2}} \\
				&\le (p-1)
				\left(
				\int_{\Omega} u^{p-2}|\nabla u|^2
				\right)^{\frac12}
				\left(
				\int_{\Omega} k^{2}_2|\nabla v|^2 u^p
				\right)^{\frac12}\\
				&\leq\frac{(p-1)k_1}{2}\int_{\Omega} u^{p-2}|\nabla u|^2+\frac{k^{2}_2 (p-1)}{2k_1}\int_{\Omega} u^{p}|\nabla v|^2.
		\end{aligned}
		\end{equation}
		From \eqref{3.40} and \eqref{3.41}, it follows that
			\begin{equation*}
				\begin{aligned}
						\frac{1}{p}\frac{d}{dt}\int_{\Omega} u^p+\mu A\int_{\Omega}u^p+\frac{\mu}{2}\int_{\Omega}u^{\theta+p+1}
					&+\frac{(p-1)k_1}{2}\int_{\Omega}u^{p-2}|\nabla u|^2\\
					&\leq \frac{k^{2}_2 (p-1)}{2k_1}\int_{\Omega} u^{p}|\nabla v|^2	+{C_3}.
				\end{aligned}
			\end{equation*}
	Moreover, observing that
	$\frac{(p-1)k_1}{2}\int_{\Omega}u^{p-2}|\nabla u|^2
	=\frac{2(p-1)k_1}{p^2}\int_{\Omega}|\nabla u^{p/2}|^2$,
	we obtain
		\begin{equation}\label{3.42}
		\begin{aligned}
			\frac{1}{p}\frac{d}{dt}\int_{\Omega} u^p+\mu A\int_{\Omega}u^p+\frac{\mu}{2}\int_{\Omega}u^{\theta+p+1}+\frac{2(p-1)k_1}{p^2}\int_{\Omega}|\nabla u^{p/2}|^2
			\leq \frac{k^{2}_2 (p-1)}{2k_1}\int_{\Omega} u^{p}|\nabla v|^2+{C_3}.
		\end{aligned}
	\end{equation}
	To estimate the right-hand side of the above equation, we invoke Lemma \ref{Lq norm of v and grad v} and choose \(q_1>d\) such that
			\begin{equation*}
			\|\nabla v(\cdot,t)\|_{L^{q_1}(\Omega)} \le {C_4},
		\end{equation*}
	for all $t \in (0,T_{\max})$. By employing Hölder's inequality and the Gagliardo-Nirenberg inequality, followed by Young's inequality with the choice $\epsilon=\frac{k_1(p-1)}{p^2}$, we proceed as follows:
		\begin{equation}\label{3.43}
			\begin{aligned}
				\frac{k_2^2 (p-1)}{2k_1}\int_{\Omega} u^p |\nabla v|^2
				&\le \frac{k_2^2 (p-1)}{2k_1}
				\left(\int_{\Omega} |\nabla v|^{q_1}\right)^{\frac{2}{q_1}}
				\left(\int_{\Omega} u^{\frac{pq_1}{q_{1}-2}}\right)^{\frac{q_{1}-2}{q_1}} \\
				&\le {C_5} \left\|u^{\frac{p}{2}}\right\|_{L^{\frac{2q_1}{q_{1}-2}}(\Omega)}^2 \\
				&\le {C_6} \left\|\nabla u^{\frac{p}{2}}\right\|_{L^2(\Omega)}^{\frac{2d}{q_1}}
				\left\|u^{\frac{p}{2}}\right\|_{L^2(\Omega)}^{\frac{2(q_{1}-d)}{q_1}}
				+ {C_6} \left\|u^{\frac{p}{2}}\right\|_{L^2(\Omega)}^2 \\
				&\le\left(\|\nabla u^{\frac{p}{2}}\|_{L^2(\Omega)}^2\right)^{\frac{d}{q_1}}\left({C_6}^{\frac{q_1}{q_{1}-d}}\|u^{\frac{p}{2}}\|_{L^2(\Omega)}^2\right)^{\frac{q_{1}-d}{q_1}}+{C_6} \left\|u^{\frac{p}{2}}\right\|_{L^2(\Omega)}^2\\
				&\le \frac{k_1(p-1)}{p^2}
				\left\|\nabla u^{\frac{p}{2}}\right\|_{L^2(\Omega)}^2
				+ {C_7}\int_{\Omega} u^p.
			\end{aligned}
		\end{equation}
It follows from \eqref{3.42} and \eqref{3.43} that
\begin{equation}\label{3.45}
	\begin{aligned}
		\frac{1}{p}\frac{d}{dt}\int_{\Omega} u^p+\mu A\int_{\Omega}u^p+\frac{\mu}{2}\int_{\Omega}u^{\theta+p+1} \leq {C_7}\int_{\Omega}u^p+{C_3},
	\end{aligned}
\end{equation}
applying Young's inequality to the term ${C_7}\int_{\Omega}u^p$, we find
\begin{equation}\label{youngs in u^p}
	\begin{aligned}
		{C_7}\int_{\Omega}u^p&\leq\left(\int_{\Omega}u^{\theta+p+1}\right)^{\frac{p}{\theta+p+1}}\left(\int_{\Omega}{C_7}^{\frac{\theta+p+1}{\theta+1}}\right)^{\frac{\theta+1}{\theta+p+1}}\\
		&\leq\frac{\mu}{4}\int_{\Omega}u^{\theta+p+1}+|\Omega|{C_8}{C_7}^{\frac{\theta+p+1}{\theta+1}},
	\end{aligned}
\end{equation} 
where ${C_8}=\frac{\theta+1}{\theta+p+1}\left(\frac{4p}{\mu (\theta+p+1)}\right)^{\frac{p}{\theta+1}}$. Thus, from \eqref{3.45} and \eqref{youngs in u^p}, we obtain
\begin{equation*}
	\begin{aligned}
		\frac{d}{dt}\int_{\Omega} u^p+\mu Ap\int_{\Omega}u^p+\frac{\mu p}{4}\int_{\Omega}u^{\theta+p+1} \leq C(p),
	\end{aligned}
\end{equation*}
where $C(p)=\left({C_3}+|\Omega|{C_8}{C_7}^{\frac{\theta+p+1}{\theta+1}}\right)p.$ Thus, by applying  Lemma \ref{ode comparison}, we obtain the desired $L^p$ estimate for $u$.
\end{proof}
Lemma \ref{Lp estimate of u 1} provides uniform $L^p(\Omega)$-bounds for $u$ for every $p>1$. Building on these estimates, we next derive an $L^\infty(\Omega)$-bound for $u$ by applying the Alikakos--Moser iteration technique \cite{Moser,M1,M3,M2}.
\begin{lemma}\label{Linfty bound of u}
	Assume that \eqref{Initial Cond} and \eqref{Hypothesis H1} are satisfied. Then there exists a constant $C>0$ such that
	\begin{equation}\label{uniform bound C on u}
		\|u(\cdot,t)\|_{L^\infty(\Omega)} \leq C,\ \text{for all } t \in (0,T_{\max}).
	\end{equation}
\end{lemma}						
\begin{proof}
	Combining \eqref{testing with u powerp-1} with the bounds for $\gamma(v)$ and $\gamma'(v)$ given in \eqref{k1 lower bound on gamma} and \eqref{bound on gamma'}, respectively, yields
\begin{align}\label{ener-1}
				\frac{1}{p}\frac{d}{dt} \int_\Omega u^p
			&\leq -(p-1)k_1 \int_\Omega u^{p-2}|\nabla u|^2
			+ (p-1)k_2 \int_\Omega u^{p-1}|\nabla u| |\nabla v|\nonumber\\
			&\quad	+ \mu \int_\Omega \left(-u^{\theta+p+1} + Au^{\theta+p} - Au^p + u^{p+1}\right),
\end{align}
for all $t\in(0,T_{\max})$. By using H\"older's and Young's inequality, we find 
\begin{align*}
(p-1)k_2 \int_\Omega u^{p-1}|\nabla u| |\nabla v|&\leq (p-1) k_2\left[\left(\int_{\Omega}u^{p-2}|\nabla u|^2\right)^\frac{1}{2} \left(\int_{\Omega}u^p|\nabla v|^2\right)^\frac{1}{2}\right]
\nonumber\\&\leq \frac{(p-1)k_1}{2}\int_{\Omega}u^{p-2}|\nabla u|^2+\frac{k^2_{2}(p-1)}{2k_1}\int_{\Omega}u^p|\nabla v|^2.
\end{align*}
In order to estimate $\int_\Omega u^{p-2}|\nabla u|^2$, we use  $|\nabla(u^{p/2})|^2=\frac{p^2}{4}u^{p-2}|\nabla u|^2$, for $p \geq 2$.  Therefore, we infer from \eqref{ener-1} that
\begin{equation*}
\begin{aligned}
	\frac{1}{p}\frac{d}{dt} \int_\Omega u^p&\leq -\frac{4(p-1)k_1}{p^2} \int_\Omega \left|\nabla u^{p/2}\right|^2
+\frac{2k_1(p-1)}{p^2}\int_{\Omega}|\nabla u^{p/2}|^2+\frac{k^2_{2}(p-1)}{2k_1}\int_{\Omega}u^p|\nabla v|^2 \\
&\quad-\mu \int_\Omega u^{\theta+p+1}
+ \mu A \int_\Omega u^{\theta+p}-\mu A\int_{\Omega} u^p + \mu \int_\Omega u^{p+1},
\end{aligned}
\end{equation*}
for all $t\in(0,T_{\max})$. Moreover, as $p\geq2$, we observe that $\frac{p-1}{p}\geq\frac{1}{2}$, $p^2\geq(p^2-p)$, this entails
\begin{align}\label{4.31}
&	\frac{d}{dt} \int_\Omega u^p+\mu Ap \int_{\Omega} u^p+\mu p\int_\Omega u^{\theta+p+1}+C_1\int_\Omega |\nabla u^{p/2}|^2
		\nonumber\\
		&   \leq C_2p^2\int_{\Omega}u^p|\nabla v|^2+ \mu Ap \int_\Omega u^{\theta+p} + \mu p \int_\Omega u^{p+1},
\end{align}
for all $t\in(0,T_{\max})$, where $C_1=k_1, C_2=\frac{k_2^{2}}{2k_1}$.
We recall Lemma \ref{Lq norm of v and grad v}, which invokes $q_2\in(d,\frac{d(\theta+2)}{\beta d-(\theta+2)})$ and $C_3>0$ independent of $p$ satisfying
\[	\|\nabla v(\cdot,t)\|_{L^{q_2}(\Omega)} \le C_3,\ \text{for all } t \in (0,T_{\max}).\]
This with the Gagliardo-Nirenberg inequality and Young's inequality enable us to choose $C_4>0$ and $C_5>0$ fulfilling
\begin{align}\label{4.32}
	C_2 p^2 \int_{\Omega} u^p |\nabla v|^2
	&\leq C_2 p^2
	\left( \int_{\Omega} |\nabla v|^{q_2} \right)^{\frac{2}{q_2}}
	\left( \int_{\Omega} u^{\frac{p q_2}{q_{2}-2}} \right)^{\frac{q_{2}-2}{q_{2}}}
	\nonumber\\
	&\leq C_2 C_3^2 p^2
	\left\| u^{\frac{p}{2}} \right\|_
	{L^{\frac{2q_{2}}{q_{2}-2}}(\Omega)}^2
	\nonumber\\
	&\leq C_4 p^2
	\left\| \nabla u^{\frac{p}{2}} \right\|_{L^2(\Omega)}^{
		2\frac{\frac{1}{q_{2}}+\frac{1}{2}}
		{\frac{1}{d}+\frac{1}{2}}}
	\left\| u^{\frac{p}{2}} \right\|_{L^1(\Omega)}^{
		2\frac{\frac{1}{d}-\frac{1}{q_{2}}}
		{\frac{1}{d}+\frac{1}{2}}}
	+ C_4 p^2
	\left\| u^{\frac{p}{2}} \right\|_{L^1(\Omega)}^2
	\nonumber\\
	&\leq \frac{C_1}{4}
	\left\| \nabla u^{\frac{p}{2}} \right\|_{L^2(\Omega)}^2
	+ C_5 p^{
		\frac{\frac{2}{d}+1}
		{\frac{1}{d}-\frac{1}{q_{2}}}}
	\left\| u^{\frac{p}{2}} \right\|_{L^1(\Omega)}^2,
\end{align}
for all $t\in(0,T_{\max})$, where $C_4$ and $C_5$ are positive constants. Additionally, $q_2>d$ validates $\frac{\frac{1}{q_{2}}+\frac{1}{2}}
{\frac{1}{d}+\frac{1}{2}}<1$. 
One can write 
\begin{align*}
	u^{\theta+p}=u^{\frac{p(\theta+1)+\theta(\theta+1)}{\theta+1}}=u^{\frac{(p+\theta+1)\theta}{\theta+1}}u^{\frac{p}{\theta+1}}. 
\end{align*}
In addition, applying H\"older's inequality with exponents $\frac{\theta+1}{\theta}$ and $ {\theta+1},$ and Young's inequality  to the penultimate integral on the right-hand side of \eqref{4.31}, we obtain
	\begin{equation}\label{Youngs 1 in u theta p}
		\begin{aligned}
				\mu Ap \int_{\Omega}u^{\theta+p} &\leq \mu p
				\left(\left(\int_{\Omega}u^{p+\theta+1}\right)^{\frac{\theta}{\theta+1}}\left(A^{\theta+1}\int_{\Omega}u^p\right)^{\frac{1}{\theta+1}}\right)\\
				&\leq \frac{\mu p}{4} \int_{\Omega}u^{p+\theta+1}+ p c_1(\theta,\mu,A)\int_{\Omega}u^p,
		\end{aligned}
	\end{equation}
where $c_1(\theta,\mu,A) =\mu A^{\theta+1} \frac{1}{(\theta+1)}\left(\frac{4\theta}{\theta+1}\right)^\theta$. Similarly, we obtain 
\begin{equation}\label{youngs in u power p+1}
	\begin{aligned}
		\mu p \int_{\Omega}u^{p+1} &\leq \mu p
		\left(\left(\int_{\Omega}u^{p+\theta+1}\right)^{\frac{1}{\theta+1}}\left(\int_{\Omega}u^p\right)^{\frac{\theta}{\theta+1}}\right)\\
		&\leq \frac{\mu p}{4} \int_{\Omega}u^{p+\theta+1}+ p c_2(\theta,\mu)\int_{\Omega}u^p,
	\end{aligned}
\end{equation}
where $c_2(\theta,\mu) = \frac{\mu 4^{\frac{1}{\theta}}}{(\theta+1)^{(1+\frac{1}{\theta})}}$.
Using \eqref{4.32}, \eqref{Youngs 1 in u theta p} and \eqref{youngs in u power p+1} in \eqref{4.31}, it follows that
\begin{equation}\label{4.35}
	\begin{aligned}
	\frac{d}{dt} \int_\Omega u^p+\mu Ap \int_{\Omega} u^p+\frac{\mu p}{2}\int_\Omega u^{\theta+p+1}+\frac{3C_1}{4}\int_\Omega |\nabla u^{p/2}|^2
&\leq  C_5 p^{
	\frac{\frac{2}{d}+1}
	{\frac{1}{d}-\frac{1}{q_{2}}}}
	\left(\int_{\Omega}u^\frac{p}{2}\right)^2+ C_6p\int_\Omega u^{p},		
	\end{aligned}
\end{equation}
for all $t\in(0,T_{\max})$, where $C_6=c_1(\theta,\mu,A)+c_2(\theta,\mu)$ is a positive constant. To estimate the second integral on the right-hand side of the above inequality, we further invoke the Gagliardo-Nirenberg inequality as
\begin{align*}
	C_6p\int_{\Omega}u^p=C_6 p\|u^{p/2}\|_{L^2}^2
	&\leq
	C_7 p\|\nabla u^{p/2} \|_{L^2}^{\frac{2d}{d+2}}
	\|u^{p/2}\|_{L^1}^{\frac{4}{d+2}}
	+
	C_7p\|u^{p/2}\|_{L^1}^2\\ 
	&\leq\frac{C_1}{4}\|\nabla u^{p/2} \|_{L^2}^2+C_8p^{\frac{d+2}{2}}\|u^{p/2}\|_{L^1}^2+C_8p\|u^{p/2}\|_{L^1}^2\\
    &\leq\frac{C_1}{4}\left\| \nabla u^{\frac{p}{2}} \right\|_{L^2(\Omega)}^2+2C_8p^{\frac{d+2}{2}}\left\| u^{\frac{p}{2}} \right\|_{L^1(\Omega)}^2,
\end{align*}
for all $t\in(0,T_{\max})$, with $C_8>0$. We observe that
$
\frac{\frac{2}{d}+1}{\frac{1}{d}-\frac{1}{q_2}}
>
\frac{d+2}{2}.
$
Therefore, substituting the above estimate in \eqref{4.35}, we obtain
\begin{equation*}
	\begin{aligned}
		\frac{d}{dt} \int_\Omega u^p+\mu Ap \int_{\Omega} u^p+\frac{\mu p}{2}\int_\Omega u^{\theta+p+1}+\frac{C_1}{2}\int_\Omega |\nabla u^{p/2}|^2
		&\leq  C_9 p^{
			\frac{\frac{2}{d}+1}
			{\frac{1}{d}-\frac{1}{q_{2}}}}
		\left(\int_{\Omega}u^\frac{p}{2}\right)^2,		
	\end{aligned}
\end{equation*}
for all $t\in(0,T_{\max})$, with $C_9=2C_8+C_5>0$. Moreover, we can write
\begin{equation*}
	\begin{aligned}
		\frac{d}{dt} \int_\Omega u^p+\mu Ap \int_{\Omega} u^p\leq  C_9 p^{
			\frac{\frac{2}{d}+1}
			{\frac{1}{d}-\frac{1}{q_{2}}}}
		\left(\int_{\Omega}u^\frac{p}{2}\right)^2,\ \text{for all}\ t\in(0,T_{\max}).		
	\end{aligned}
\end{equation*}
It follows that
\begin{align}
		\sup_{0<t<T_{\max}}\left(\int_{\Omega}u^p(\cdot,t)\right) &\leq e^{-\mu Apt}\int_{\Omega}u_0^{p}+C_{10} p^{
			\frac{\frac{1}{d}+\frac{1}{q_2}+1}
			{\frac{1}{d}-\frac{1}{q_{2}}}}(1-e^{-\mu Apt})\sup_{0<t<T_{\max}} \left(\int_\Omega u^{p/2}(\cdot,t)\right)^{2},\nonumber\\ 
			&\leq\int_{\Omega}u_0^{p}+C_{10} p^{
				\frac{\frac{1}{d}+\frac{1}{q_2}+1}
				{\frac{1}{d}-\frac{1}{q_{2}}}}\sup_{0<t<T_{\max}} \left(\int_\Omega u^{p/2}(\cdot,t)\right)^{2}, \nonumber
\end{align}
for all $t\in(0,T_{\max})$, where $C_{10}=\frac{C_{9}}{\mu A}>0$. Taking the $p$-th root on both sides, we obtain
\begin{align}\label{sup u ineq}
	\sup_{0<t<T_{\max}}\left(\int_{\Omega}u^p(\cdot,t)\right)^{1/p} &\leq \left(\int_\Omega u_0^p(x)\right)^{1/p}
	+{C_{10}}^{\frac{1}{p}} \left(p^{\frac{\kappa}{p}}\right)
	\sup_{0<t<T_{\max}} \left(\int_\Omega u^{p/2}(\cdot,t)\right)^{2/p}\nonumber\\
 	&\leq |\Omega|^{1/p}\|u_0\|_{L^\infty(\Omega)}
	+ {C_{10}}^{\frac{1}{p}}\left(p^{\frac{\kappa}{p}}\right)
	\sup_{0<t <T_{\max}} \left(\int_\Omega u^{p/2}(\cdot,t)\right)^{2/p},
\end{align}
where 
\[
\kappa= \frac{\frac{1}{d}+\frac{1}{q_2}+1}
{\frac{1}{d}-\frac{1}{q_{2}}}.
\]
Further, we define
\begin{align*}
	G(q) = \max\left\{\|u_0\|_{L^\infty(\Omega)}, \sup_{0< t< T_{\max}}\left(\int_\Omega u^q\right)^{1/q}\right\}.
\end{align*}
By using \eqref{sup u ineq} in above, we deduce
\begin{equation*}
	\begin{aligned}
		G(p) &\leq \max\left\{\|u_0\|_{L^\infty(\Omega)}, |\Omega|^{1/p}\|u_0\|_{L^\infty(\Omega)} + {C_9}^{\frac{1}{p}}\left(p^{\frac{\kappa}{p}}\right)
		\sup_{0<t <T_{\max}} \left(\int_\Omega u^{p/2}(\cdot,s)\right)^{2/p}\right\}\\
		&\leq |\Omega|^{1/p}\|u_0\|_{L^\infty(\Omega)} + {C_9}^{\frac{1}{p}}\left(p^{\frac{\kappa}{p}}\right)
		\sup_{0<t <T_{\max}} \left(\int_\Omega u^{p/2}(\cdot,s)\right)^{2/p}.
	\end{aligned}
\end{equation*} 
It follows directly from the definition of $G(p)$ that $\|u_0\|_{L^\infty(\Omega)} \leq G(p/2),$
and 
\[\sup_{0<t<T_{\max}} \left(\int_\Omega u^{p/2}(x,t)\right)^{2/p} \leq G(p/2).\]
Consequently, we obtain
\[
G(p) \leq |\Omega|^{1/p}G(p/2) + {C_9}^{\frac{1}{p}}\left(p^{\frac{\kappa}{p}}\right) G(p/2)\leq (C_{10})^{1/p}\left(p^{\frac{\kappa}{p}}\right) G(p/2),
\]
where $C_{10}=\max\{|\Omega|^{1/p},C_9\}$. Now, we let $p=2^j\ (j=1,2,3,\cdots)$ and $C_{11}=2^{\kappa}$, it follows that
\begin{equation*}
	\begin{aligned}
G(2^j) \leq C^{\frac{1}{2^j}}_{10} C^{\frac{j}{2^j}}_{11}G(2^{j-1})\leq\cdots\leq C_{10}^{\sum_{j=1}^{\infty}\frac{1}{2^j}}
C_{11}^{\sum_{j=1}^{\infty}\frac{j}{2^j}}G(2)\leq C_{12}.
	\end{aligned}
\end{equation*}
Due to Lemma \ref{Lp estimate of u 1},  $C_{12}>0$ is a constant, independent of $p$.  Using the above bound, we deduce  from \eqref{sup u ineq} that
\begin{align*}
	\sup_{0<t<T_{\max}}\|u(\cdot,t)\|_{L^{2^j}}\leq C_{12}. 
\end{align*}
We therefore achieve \eqref{uniform bound C on u} by sending $j\to\infty$.
\end{proof}
We are now ready to conclude the proof of Theorem \ref{global bound_H1}.
\begin{proof}[Proof of Theorem \ref{global bound_H1}]
By Lemma \ref{Linfty bound of u},  there exists a constant $C>0$ satisfying
	\[
	\|u(\cdot,t)\|_{L^\infty(\Omega)}
	\leq C,
	\ \text{ for all } t\in(0,T_{\max}).
	\]
 Consequently, the blow-up criterion \eqref{blow up crit} would contradict for all $t\in(0,T_{\max})$. We therefore
	conclude that $T_{\max}=\infty$, and the conclusion of Theorem
	\ref{global bound_H1} follows.
\end{proof}
\section{Boundedness under hypotheses \eqref{Hypothesis H1}, \eqref{Hypothesis H2} on the motility function}\label{sec:global2}		
In this section, we prove the existence of globally bounded classical solutions to \eqref{Main} by imposing a stronger condition on the motility function $\gamma$. Under the assumptions on $\gamma$ specified in \eqref{Hypothesis H1} and \eqref{Hypothesis H2}, the higher regularity of $v$ obtained in Lemma~\ref{Lq norm of v and grad v}, namely \eqref{Lq estimate grad v}, is not needed to derive the $L^p$ estimate for $u$. We first establish several integral estimates for solutions of \eqref{Main}, which will then serve as key ingredients in proving the desired global boundedness.

\begin{lemma}\label{estimate 4.1}
	Let the assumptions \eqref{Initial Cond}, \eqref{Hypothesis H1} and \eqref{Hypothesis H2} be valid. Then for all $1<p<\infty$, we have 
	\begin{equation}\label{estimate 4.47}
		\begin{aligned}
				\frac{1}{p}\frac{d}{dt}\int_{\Omega} u^p+ \mu A \int_{\Omega} u^p
				&+ \frac{p-1}{2} \int_{\Omega} \gamma(v) u^{p-2}|\nabla u|^2+\frac{3\mu}{4}\int_{\Omega}u^{\theta+p+1}\\
				& \leq \frac{k_3 k_4^{2} (p-1)}{2}\int_{\Omega}u^p |\nabla v|^2+C.
		\end{aligned}
	\end{equation}
\end{lemma}
\begin{proof}
	By having hypotheses \eqref{Hypothesis H1} and \eqref{Hypothesis H2}, we have 
	\begin{equation}\label{additional bounds}
		\begin{aligned}
			\gamma(v)\leq \gamma(0)=k_3,\  \frac{|\gamma'(v(\cdot,t)|)}{\gamma(v(\cdot,t))}\leq k_4,
		\end{aligned}
	\end{equation}
	for all $t\in(0,T_{\max})$ with $k_3,k_4>0$. We multiply equation \eqref{Main} by  $u^{p-1}$   and integrate it over spatial domain to obtain 
	\begin{equation}\label{testing u^p-1}
		\begin{aligned}
			\frac{1}{p}\frac{d}{dt}\int_{\Omega} u^p +(p-1)\int_{\Omega} u^{p-2}\gamma(v)|\nabla u|^2 
			&\leq (p-1)\int_{\Omega}u^{p-1}|\gamma'(v)||\nabla u||\nabla v|\\
			 &  \quad+\mu\int_{\Omega} (u^{p+1}-Au^p-u^{\theta+p+1}+Au^{\theta+p}).
		\end{aligned}
	\end{equation}
	For the first term on the right-hand side of the preceding equation, an application of Hölder's inequality followed by Young's inequality with $\epsilon=\frac{\gamma(v)}{2}$ yields
	\begin{align}\label{grad u grad v youngs}
		&	(p-1)\int_{\Omega}u^{p-1}|\gamma'(v)||\nabla u||\nabla v|  \nonumber\\&=(p-1)\int_{\Omega}u^{\frac{p-2}{2}}|\nabla u||\nabla v||\gamma'(v)|u^{\frac{p}{2}}
		\nonumber\\
			&\leq (p-1)\left(\int_{\Omega}(u^{p-2}|\nabla u|^2)\right)^\frac{1}{2} \left(\int_{\Omega}u^p|\nabla v|^2|\gamma'(v)|^2\right)^\frac{1}{2}
			\nonumber\\
			&\leq \frac{\gamma(v)(p-1)}{2}\int_{\Omega}u^{p-2}|\nabla u|^2 +\frac{(p-1)}{2\gamma(v)}\int_{\Omega} u^p |\nabla v|^2 |\gamma'(v)|^2
			\nonumber\\
			&\leq \frac{\gamma(v)(p-1)}{2}\int_{\Omega}u^{p-2}|\nabla u|^2 +\frac{(p-1)}{2}\int_{\Omega} u^p |\nabla v|^2 \gamma(v)\frac{|\gamma'(v)|^2}{|\gamma(v)|^2}.
		\end{align}
 Combining \eqref{grad u grad v youngs} with \eqref{testing u^p-1} and invoking assumption \eqref{additional bounds}, we achieve
		\begin{align}\label{estimate 4.51}
			\frac{1}{p}\frac{d}{dt}\int_{\Omega} u^p &+\frac{(p-1)}{2}\int_{\Omega} u^{p-2}\gamma(v)|\nabla u|^2
			+\mu A\int_{\Omega}u^p+\mu\int_{\Omega}u^{\theta+p+1}\nonumber\\
		    & \leq \frac{(p-1)}{2}\int_{\Omega} u^p |\nabla v|^2 \gamma(v)\frac{|\gamma'(v)|^2}{|\gamma(v)|^2} + \mu\int_{\Omega} (u^{p+1}+Au^{\theta+p})\nonumber\\
			&\leq \frac{k_3 k_4^{2}(p-1)}{2} \int_{\Omega} u^p |\nabla v|^2+\mu\int_{\Omega}u^{p+1}+A\mu\int_{\Omega}u^{\theta+p}. 
		\end{align}
	Now, we estimate second and third integral appearing in the right  hand side of \eqref{estimate 4.51}, by using H\"older's and Young's inequalities as
	\begin{align}\label{youngs 1}
			\mu\int_{\Omega}u^{p+1}&\leq\left(\int_{\Omega}u^{\theta+p+1}\right)^{\frac{p+1}{\theta+p+1}}\left(\int_{\Omega}\mu^{\frac{\theta+p+1}{\theta}}\right)^\frac{\theta}{\theta+p+1}\leq \frac{\mu}{8}\int_{\Omega}u^{\theta+p+1}+{C_1},
		\end{align}
	where ${C_1}=\frac{|\Omega|\theta\mu^{\frac{\theta+p+1}{\theta}}}{\theta+p+1}\left(\frac{8(p+1)}{\mu(\theta+p+1)}\right)^{p+1/\theta}$, and
		\begin{align}\label{youngs 2}
			A\mu\int_{\Omega}u^{\theta+p}&\leq\left(\int_{\Omega}u^{\theta+p+1}\right)^\frac{\theta+p}{\theta+p+1}\left(\int_{\Omega}(A\mu)^{\theta+p+1}\right)^\frac{1}{\theta+p+1}
			\leq \frac{\mu}{8}\int_{\Omega}u^{\theta+p+1}+{C_2},
	    \end{align}	    
	where ${C_2}=\frac{|\Omega|(A\mu)^{\theta+p+1}}{(\theta+p+1)}\left(\frac{8(\theta+p)}{\mu (\theta+p+1)}\right)^{p+\theta}$. By using \eqref{youngs 1} and \eqref{youngs 2} in \eqref{estimate 4.51}, we finally deduce
	\begin{align*}
			&\frac{1}{p}\frac{d}{dt}\int_{\Omega} u^p+\mu A\int_{\Omega}u^p +\frac{(p-1)}{2}\int_{\Omega} u^{p-2}\gamma(v)|\nabla u|^2+\frac{3\mu}{4}\int_{\Omega}u^{\theta+p+1}
			\nonumber\\&\leq \frac{k_3 k_4^{2}(p-1)}{2} \int_{\Omega} u^p |\nabla v|^2+C,
	\end{align*}
	which completes the proof. 
\end{proof}
We adopt the approach presented in \cite[Section 4]{JAMAA25} to establish the following lemma, which is needed to complete the proof of Lemma \ref{lemma 4.8}.
\begin{lemma}\label{derivative of grad v to power 2q}
	Let \eqref{Initial Cond} be true. Then for any $1<q<\infty$, there exists a constant $C$ such that 
	\begin{equation}\label{estimate 4.52}
			\frac{1}{q}\frac{d}{dt}\int_{\Omega}|\nabla v|^{2q}+2\int_{\Omega}|\nabla v|^{2q}+\frac{q-1}{q^2}\int_{\Omega}|\nabla|\nabla v|^q|^2
			\leq \left(4q-4+\frac{d}{2}\right)\int_{\Omega}w^{2\beta}|\nabla v|^{2q-2}+ C,
	\end{equation} 
	for all $t\in (0,T_{\max})$                .
\end{lemma}
\begin{proof}
	From the second equation of \eqref{Main}, we obtain 
	\begin{equation}\label{decompose into three int}
		\begin{aligned}
			\frac{1}{q}\frac{d}{dt}\int_{\Omega}|\nabla v|^{2q} &=2\int_{\Omega}|\nabla v|^{2q-2}\nabla v \cdot \nabla v_t\\
			&=2\int_{\Omega}|\nabla v|^{2q-2}\nabla v \cdot \nabla(\Delta v-v+w^\beta)\\
			&=2\int_{\Omega} |\nabla v|^{2q-2}\nabla v\cdot\nabla \Delta v-2\int_{\Omega}|\nabla v|^{2q}+2\int_{\Omega}|\nabla v|^{2q-2}\nabla v \cdot \nabla w^\beta\\
			&=I_1+I_2+I_3.
		\end{aligned}
	\end{equation}
	\vskip 0.1 cm
	\noindent 
	\textbf{Estimation of $I_1$}:
	To deal with $(2\nabla v\cdot\nabla \Delta v)$ in $I_1$, 
	 we use  the identity
$$
	2 \nabla v \cdot \nabla \Delta v
	= \Delta |\nabla v|^2 - 2 |D^2 v|^2.
$$
Therefore, we estimate $I_1$ by using the above identity as
	\begin{align}\label{bound. int of v}
			I_1&=2\int_{\Omega}|\nabla v|^{2q-2}\nabla v\cdot \nabla\Delta v \nonumber\\&=\int_{\Omega}|\nabla v|^{2q-2}\Delta|\nabla v|^2-2\int_{\Omega}|\nabla v|^{2q-2}|D^2v|^2\nonumber\\
			&=-\int_{\Omega}\nabla |\nabla v|^{2q-2}\cdot \nabla|\nabla v|^2+\int_{\partial\Omega}|\nabla v|^{2q-2}\frac{\partial|\nabla v|^2}{\partial \nu}-2\int_{\Omega}|\nabla v|^{2q-2}|D^2 v|^2
		\nonumber	\\ &=-\frac{4(q-1)}{q^2}\int_{\Omega}|\nabla |\nabla v|^q|^2+\int_{\partial\Omega}|\nabla v|^{2q-2}\frac{\partial|\nabla v|^2}{\partial \nu}-2\int_{\Omega}|\nabla v|^{2q-2}|D^2 v|^2.
		\end{align}
We next estimate the boundary integral in \eqref{bound. int of v}. By Lemma~\ref{control boun integral}, which provides an estimate for the normal derivative of $|\nabla f|^2$ on $\partial\Omega$ for $f\in C^2(\overline{\Omega})$ satisfying the homogeneous Neumann boundary condition, we obtain from \eqref{normal derivative bound}, upon taking $f=\nabla v$, that
	\begin{equation*}
		\begin{aligned}
			\frac{\partial |\nabla v|^2}{\partial \nu} \leq c_{\Omega} |\nabla v|^2.
		\end{aligned}
	\end{equation*}
On multiplying the above equation with $|\nabla v|^{2q-2}$ and integrating over the boundary, we get 
	\begin{equation}\label{boundary integral}
		\begin{aligned}
		\int_{\partial\Omega} \frac{\partial |\nabla v|^2}{\partial \nu}\, |\nabla v|^{2q-2}
		\le c_\Omega \int_{\partial\Omega} |\nabla v|^2 |\nabla v|^{2q-2} = c_{\Omega}\bigl\||\nabla v|^q\bigr\|_{L^2(\partial\Omega)}^2.
		\end{aligned}
	\end{equation}
On the other hand, by Lemma~\ref{compact embedd} and Lemma \ref{linear bounded map}, choosing \(r\in\left(0,\frac12\right)\), the trace operator \(W^{r+\frac12,2}(\Omega)\to W^{r,2}(\partial\Omega)\) is bounded, while the embedding \(W^{r,2}(\partial\Omega)\hookrightarrow L^2(\partial\Omega)\) is compact. Consequently, the composition \(W^{r+\frac12,2}(\Omega)\hookrightarrow L^2(\partial\Omega)\) is compact. Therefore,
	\begin{equation}\label{fractional reg}
	\bigl\||\nabla v|^q\bigr\|_{L^2(\partial\Omega)}^2
	\le C \bigl\||\nabla v|^q\bigr\|_{W^{r+\frac12,2}(\Omega)}^2.
\end{equation}
Furthermore, in order to apply the fractional Gagliardo-Nirenberg inequality from Lemma~\ref{fractional gagliardo} to the right hand side of \eqref{fractional reg}, let us pick $a\in(0,1)$ satisfying
\[
\frac12-\frac{r+\frac12}{d}
=(1-a)\frac{q}{\varrho}+a\left(\frac12-\frac1d\right),
\]
that is, 
\[
a = \left(\frac{1}{2} - \frac{1}{2d} - \frac{q}{\varrho} - \frac{r}{d}\right)
\left(\frac{1}{2} - \frac{1}{d} - \frac{q}{\varrho}\right)^{-1}.
\]
Noting that \(r\in\left(0,\frac12\right)\) and $\varrho\leq2q$ imply   $r+\frac{1}{2}\leq a<1$.
We see from the fractional Gagliardo-Nirenberg inequality \eqref{gagliardo ineq} with \(f=|\nabla v|^q\)  together with Lemma~\ref{beta k condition} (by choosing $k=\beta d$) which yields a uniform bound for the quantity
$
\big\||\nabla v|^q\big\|_{L^{\frac{\varrho}{q}}(\Omega)}
$ by fixing $\varrho$ in such a way that $1<q\leq \varrho \leq 2q$.
Consequently, there exist positive constants \(c_3\) and \(c_4\) such that
\begin{equation}\label{gagliardo apply}
\begin{aligned}
	\bigl\||\nabla v|^q\bigr\|_{W^{r+\frac12,2}(\Omega)}^2
	&\leq c_1 \bigl\|\nabla |\nabla v|^q\bigr\|_{L^2(\Omega)}^{2a}
	\bigl\||\nabla v|^q\bigr\|_{L^{\varrho/q}(\Omega)}^{2(1-a)} 
	+ c_2\bigl\||\nabla v|^q\bigr\|_{L^{\varrho/q}(\Omega)}^2\\
	&\leq c_3 \|\nabla |\nabla v|^q\|_{L^2(\Omega)}^{2a} + c_4.
\end{aligned}
\end{equation}
Combining \eqref{boundary integral}, \eqref{fractional reg}, and \eqref{gagliardo apply}, we proceed as
	\begin{equation}\label{eqn-000}
		\int_{\partial\Omega}\frac{\partial |\nabla v|^2}{\partial \nu} |\nabla v|^{2q-2} \,
		\leq c_4 \|\nabla|\nabla v|^q\|_{L^2(\Omega)}^{2a} + c_5
		= c_4 \left( \int_{\Omega} |\nabla|\nabla v|^q|^2 \,  \right)^a + c_5.
	\end{equation}
Using the fact that \(a\in(0,1)\) and \(1<q<\infty\), we apply Young's inequality with conjugate exponents \(\frac{1}{a}\) and \(\frac{1}{1-a}\) with \(\epsilon=\frac{2(q-1)}{q^2}\), we obtain
	\begin{align*}
		c_4 \left( \int_{\Omega} |\nabla |\nabla v|^q|^2 \, \right)^a\leq \frac{2(q-1)}{q^2} \left( \int_{\Omega} |\nabla |\nabla v|^q|^2 \,  \right)+(1-a)\left[\frac{aq^2}{2(q-1)}\right]^{\frac{a}{1-a}}c_{4}^{\frac{1}{1-a}}.
	\end{align*}
	Therefore, from \eqref{eqn-000}, we infer 
	\begin{equation}\label{final estimate on boundary integral}
	\int_{\partial\Omega}\frac{\partial |\nabla v|^2}{\partial \nu}|\nabla v|^{2q-2} \, 
		\leq \frac{2(q-1)}{q^2} \left( \int_{\Omega} |\nabla |\nabla v|^q|^2 \, \right)+ c_6,
	\end{equation}
where \(c_6=(1-a)\left[\frac{aq^2}{2(q-1)}\right]^{\frac{a}{1-a}}c_4^{\frac{1}{1-a}}+c_5\) is a positive constant. 
Hence, substituting \eqref{final estimate on boundary integral} in \eqref{bound. int of v}, we get
\begin{align}\label{I_2 secon ineq}
		I_1
	&\leq-\frac{2(q-1)}{q^2}\int_{\Omega}|\nabla |\nabla v|^q|^2-2\int_{\Omega}|\nabla v|^{2q-2}|D^2 v|^2+c_6.
\end{align}
\textbf{Estimation of $I_3$}: We have 
	\begin{align*}
	I_3&=2\int_\Omega |\nabla v|^{2q-2}\nabla v\cdot \nabla w^{\beta}\nonumber\\
	&= -2\int_\Omega w^{\beta}\,\nabla\cdot\big(|\nabla v|^{2q-2}\nabla v\big)
	\nonumber\\&= -2(q-1)\int_\Omega w^{\beta}|\nabla v|^{2q-4}\,\nabla v\cdot\nabla|\nabla v|^2
		- 2\int_\Omega w^{\beta}|\nabla v|^{2q-2}\Delta v.
	\end{align*}
Now, using the inequality \( |\Delta v|\leq \sqrt{d}\,|D^2v| \) in the penultimate term of above inequality, we infer
	\[
	-2\int_\Omega w^{\beta}|\nabla v|^{2q-2}\Delta v
	\le 2\int_\Omega w^{\beta}|\nabla v|^{2q-2}|\Delta v|
	\le 2\sqrt{d}\int_\Omega w^{\beta}|\nabla v|^{2q-2}|D^2v|.
	\]
Therefore,
\begin{align}\label{I_3 ineq}
	I_3
	&\leq-2(q-1)\int_\Omega w^{\beta}|\nabla v|^{2q-4}\,\nabla v\cdot\nabla|\nabla v|^2+2\sqrt{d}\int_\Omega w^{\beta}|\nabla v|^{2q-2}|D^2v|
	\nonumber\\&\leq   \frac{4(q-1)}{q}\int_\Omega \big(w^{\beta}|\nabla v|^{q-1}\big)\,\big|\nabla|\nabla v|^q\big|+2\sqrt{d}\int_\Omega w^{\beta}|\nabla v|^{2q-2}|D^2v|
	\nonumber\\&\leq  \frac{(q-1)}{q^2}\int_\Omega \big|\nabla|\nabla v|^q\big|^2
	+ {4(q-1)}\int_\Omega w^{2\beta}|\nabla v|^{2q-2}+\int_\Omega
	|\nabla v|^{2q-2}
	\left(
	\frac{d}{2}w^{2\beta}+2|D^2v|^2
	\right)\,.
\end{align}
By combining $I_1$, $I_2$ and $I_3$ from \eqref{I_2 secon ineq}, \eqref{I_3 ineq} and \eqref{decompose into three int}, finally we achieve
	\begin{align*}
	\frac{1}{q}\frac{d}{dt}\int_{\Omega}|\nabla v|^{2q}+2\int_{\Omega}|\nabla v|^{2q}+\frac{q-1}{q^2}\int_{\Omega}|\nabla|\nabla v|^q|^2
	\leq (4q-4+\frac{d}{2})\int_{\Omega}w^{2\beta}|\nabla v|^{2q-2}+ c_6,
	\end{align*}
	which completes the proof.
\end{proof}
\begin{lemma}\label{second w lemma}
	Assume that \eqref{Initial Cond} holds and let $1<p<\infty$. Then the solution component $w$ of \eqref{Main} obeys the following differential inequality:
	\begin{equation}\label{estimate 4.76}
		\begin{aligned}
			\frac{d}{dt}\int_{\Omega}w^p+\frac{\delta p}{2}\int_{\Omega}w^p\leq\left(\frac{2}{\delta}\right)^{p-1}\int_{\Omega}u^p,
		\end{aligned}
	\end{equation}
	for every $t\in(0,T_{\max})$.
\end{lemma}
\begin{proof}
Multiplying the third equation in \eqref{Main} by $w^{p-1}$ and integrating over $\Omega$, and subsequently applying Hölder's  and Young's inequalities with $\epsilon=\frac{\delta}{2}$ and conjugate exponents $\frac{p}{p-1}$ and $p$, we obtain

\begin{align*}
			\frac{1}{p}\frac{d}{dt}\int_{\Omega}w^p&=-\delta\int_{\Omega}w^p+\int_{\Omega}uw^{p-1}\\
			&\leq -\delta\int_{\Omega}w^p+\frac{\delta}{2}\int_{\Omega}w^p+\left(\frac{2}{\delta}\right)^{p-1}\left(\frac{p-1}{p}\right)^{p-1} \frac{1}{p}\int_{\Omega}u^p.
\end{align*}
The above inequality can be rewritten in the form
\begin{align*}
			\frac{d}{dt}\int_{\Omega}w^p &\leq -\frac{\delta p}{2}\int_{\Omega}w^p+\left(\frac{2}{\delta }\right)^{p-1}\int_{\Omega}u^p,
\end{align*}
where \(\left(\frac{p-1}{p}\right)^{p-1}<1\).
\end{proof}
The following integral estimate is an immediate consequence of Lemmas~ \ref{estimate 4.1}-\ref{second w lemma}.
\begin{lemma}\label{lemma 4.8}
	Assume that \eqref{Initial Cond}, \eqref{Hypothesis H1}, and \eqref{Hypothesis H2} are satisfied. Then, for any sufficiently large $1<p<\infty$ and $1<q<\infty$, there exist positive constants ${C_1}$, ${C_2}$, and ${C_3}$ such that
	\begin{equation}\label{lemma for second theorem proof}
		\begin{aligned}
			&\frac{d}{dt}\left(\frac{1}{p}\int_{\Omega}u^p+\frac{1}{q}\int_{\Omega}|\nabla v|^{2q}+\int_{\Omega}w^p\right)+\mu A\int_{\Omega}u^p+2\int_{\Omega}|\nabla v|^{2q}+\frac{\delta p}{4}\int_{\Omega}w^p\\
			&\quad+\frac{\mu}{2}\int_{\Omega}u^{\theta+p+1}+\frac{q-1}{q^2}\int_{\Omega}|\nabla|\nabla v|^q|^2\leq C_1\int_{\Omega}|\nabla v|^{\chi_1}+C_2\int_{\Omega}|\nabla v|^{\chi_2}+C_3,
		\end{aligned}
	\end{equation}
	for every $t\in (0,T_{\max})$, where the exponents $\chi_1$ and $\chi_2$ are given by
	\begin{equation*}\label{chi values}
		\chi_1=\chi_1(p,q)=\frac{2(\theta+p+1)}{\theta+1},
		\ 
		\chi_2=\chi_2(p,q)=\frac{p(2q-2)}{p-2\beta}.
	\end{equation*}
\end{lemma}
\begin{proof}
By combining the equations \eqref{estimate 4.47}, \eqref{estimate 4.52} and \eqref{estimate 4.76}, we have
	\begin{equation}\label{estimate 4.79}
		\begin{aligned}
			&\frac{d}{dt}\left(\frac{1}{p}\int_{\Omega}u^p+\frac{1}{q}\int_{\Omega}|\nabla v|^{2q}+\int_{\Omega}w^p\right)+\mu A\int_{\Omega}u^p+2\int_{\Omega}|\nabla v|^{2q}+\frac{\delta p}{2}\int_{\Omega}w^p\\
			&\leq \frac{k_3 k_4^{2} (p-1)}{2}\int_{\Omega}u^p |\nabla v|^2 -\frac{3\mu}{4}\int_{\Omega}u^{\theta+p+1}+\left(4q-4+\frac{d}{2}\right)\int_{\Omega}w^{2\beta}|\nabla v|^{2q-2}\\
			&\quad+\left(\frac{2}{\delta}\right)^{p-1}\int_{\Omega}u^p-\frac{q-1}{q^2}\int_{\Omega}|\nabla|\nabla v|^q|^2+{C_1},
		\end{aligned}
	\end{equation}
for all $t\in(0,T_{\max})$. Since $\gamma>0$ and $p>1$, the term $\frac{p-1}{2}\int_{\Omega}\gamma(v)u^{p-2}|\nabla u|^2$ is nonnegative and may therefore be dropped from the left-hand side of the above inequality. Furthermore, by choosing $p$ sufficiently large, Young's inequality can be applied to obtain
\begin{equation}\label{ineq 4.80}
		\begin{aligned}
			\frac{k_3 k_4^{2} (p-1)}{2}\int_{\Omega}u^p |\nabla v|^2&\leq\left(\int_{\Omega}u^{\theta+p+1}\right)^{\frac{p}{\theta+p+1}}\left(\int_{\Omega}\left(\frac{k_3 k_4^{2} (p-1)}{2}|\nabla v|^2\right)^\frac{\theta+p+1}{\theta+1}\right)^\frac{\theta+1}{\theta+p+1}\\
			&\leq\frac{\mu}{8}\int_{\Omega}u^{\theta+p+1}+{C_2}\int_{\Omega}|\nabla v|^\frac{2(\theta+p+1)}{\theta+1},
		\end{aligned}
\end{equation}
where $\theta\geq1,p>1$ and ${C_2}=\left(\frac{k_3k^2_4(p-1)}{2}\right)^\frac{(\theta+p+1)}{\theta+1}\left(\frac{\theta+1}{\theta+p+1}\right)\left(\frac{8p}{\mu(\theta+p+1)}\right)^\frac{p}{\theta+1}$.
We now estimate the remaining terms in \eqref{estimate 4.79}. Since $p$ has been chosen sufficiently large, Young's inequality can be applied, yielding
\begin{equation}\label{ineq 4.81}
\begin{aligned}
\left(4q-4+\frac{n}{2}\right)\int_{\Omega}w^{2\beta}|\nabla v|^{2q-2}&\leq \left(\int_{\Omega}w^p\right)^\frac{2\beta}{p}\left(\left(4q-4+\frac{d}{2}\right)^\frac{p}{p-2\beta}|\nabla v|^\frac{(2q-2)p}{(p-2\beta)}\right)^\frac{p-2\beta}{p}\\
&\leq \frac{\delta p}{4}\int_{\Omega}w^p+{C_3}\int_{\Omega}|\nabla v|^\frac{p(2q-2)}{p-2\beta},
		\end{aligned}
\end{equation}
where ${C_3}=\left(4q-4+\frac{d}{2}\right)^\frac{p}{p-2\beta}\left(\frac{p-2\beta}{p}\right)\left(\frac{8\beta}{\delta p^2}\right)^\frac{2\beta}{p-2\beta}$.
Moreover, we have 
\begin{equation}\label{ineq 4.82}
\begin{aligned}
\left(\frac{2}{\delta}\right)^{p-1}\int_{\Omega}u^p\leq&\left(\int_{\Omega}u^{\theta+p+1}\right)^\frac{p}{\theta+p+1}\left(\int_{\Omega}\left(\left(\frac{2}{\delta}\right)^{p-1}\right)^\frac{\theta+p+1}{\theta+1}\right)^\frac{\theta+1}{\theta+p+1}\\
&\leq \frac{\mu}{8}\int_{\Omega}u^{\theta+p+1}+{C_4},
\end{aligned}
\end{equation}
where ${C_4}=|\Omega|\left(\frac{2}{\delta}\right)^{\frac{(p-1)(\theta+p+1)}{\theta+1}}\left(\frac{\theta+1}{\theta+p+1}\right)\left(\frac{8p}{\mu (\theta+p+1)}\right)^\frac{p}{\theta+1}$.
Now, collecting the relevant exponents, we define
\begin{equation*}
		\chi_1=\chi_1(p,q)=\frac{2(\theta+p+1)}{\theta+1},
		\ 
		\chi_2=\chi_2(p,q)=\frac{p(2q-2)}{p-2\beta}.
\end{equation*}
Combining \eqref{ineq 4.80}, \eqref{ineq 4.81}, and \eqref{ineq 4.82} with \eqref{estimate 4.79}, we obtain the desired estimate \eqref{lemma for second theorem proof}.
\end{proof}
Having established Lemma \ref{lemma 4.8}, we observe that controlling the integrals $\int_{\Omega}|\nabla v|^{\chi_i}$, $i=1,2$, is crucial for proving Theorem \ref{global bound H1-H2}. To this end, we establish the following result, which further provides the required control of $\int_{\Omega}|\nabla v|^{\chi_i}$, $i=1,2$, via the Gagliardo-Nirenberg inequality.
\begin{lemma}\label{cond on p and q}
	
If $\frac{\theta+2}{d}\leq\beta<\frac{d(\theta+1)+2(\theta+2)}{2d}$ and $s\in\left[1,\frac{d(\theta+2)}{\beta d-(\theta+2)}\right)$, then, for every sufficiently large $1<p<\infty$, there exists $1<q<\infty$ such that $\lambda_i\in(0,1)$ and $f_i<2$ for $i=1,2$, where $\chi_i$ $(i=1,2)$ is defined in \eqref{chi values} and
\begin{equation}\label{equation 4.84}
	\begin{aligned}
		\lambda_i=\lambda_i(p,q,s)&=\frac{\frac{q}{s}-\frac{q}{\chi_i}}{\frac{q}{s}+\frac{1}{d}-\frac{1}{2}}\in(0,1),\\
		f_i=f_i(p,q,s)&=\frac{\chi_i}{q}\lambda_i
		=\frac{\frac{\chi_i}{s}-1}{\frac{q}{s}+\frac{1}{d}-\frac{1}{2}}<2.
     \end{aligned}
\end{equation}
\end{lemma}
\begin{proof}
By \eqref{equation 4.84} with  $\lambda_i\in(0,1)$  and $f_i<2$ for $i=1,2$,  it suffices to show that
\begin{equation}\label{equiv to lambda and fi}
				2q+\frac{2s}{d}>\chi_i>s,\ \text{for all } i=1,2.
\end{equation}
\textbf{Step (i):}
Since $\lambda_i\in(0,1)$, it follows that
\begin{align*}
				\chi_i>s,\ \text{and} \ q>\frac{\chi_i}{2}-\frac{\chi_i}{d}.
\end{align*}
\textbf{Step (ii):}
Moreover, the condition $f_i<2$ yields
\begin{align*}
	\frac{\chi_i}{s}-1&<2\left(\frac{q}{s}+\frac{1}{d}-\frac{1}{2}\right)
	\implies 2q+\frac{2s}{d}>\chi_i.
\end{align*}
Combining the two steps above yields \eqref{equiv to lambda and fi}.  Consequently, any choice of $\chi_i$ satisfying \eqref{equation 4.84} also satisfies \eqref{equiv to lambda and fi}. It therefore remains to show that there exists some $1<q<\infty$ such that, for all sufficiently large $p$, the pair $(p,q)$ satisfies the conditions obtained from \eqref{equiv to lambda and fi}, namely,
		\begin{equation}\label{derive conditions}
			\begin{aligned}
				\chi_i>s \ \text{ and } \ \chi_i<2q+\frac{2s}{d}.
			\end{aligned}
		\end{equation}
Since $\chi_1=\frac{2(p+\theta+1)}{\theta+1}$ and $\chi_2=\frac{p(2q-2)}{p-2\beta}$, we first invoke the condition $\chi_i>s$, $i=1,2$, from \eqref{derive conditions}. It then follows that $p$ and $q$ must satisfy the following lower bounds:
\begin{equation}\label{lower bounds of p,q}
		 \begin{aligned}
		 	\chi_1>s
		 	&\Longrightarrow \frac{2(p+\theta+1)}{\theta+1}>s
		 	\Longrightarrow p>\frac{s(\theta+1)}{2}-(\theta+1),\\[2mm]
		 	\chi_2>s
		 	&\Longrightarrow \frac{p(2q-2)}{p-2\beta}>s
		 	\Longrightarrow q>\frac{s(p-2\beta)}{2p}+1.
		 \end{aligned}
\end{equation}
Now, similarly, by applying the second condition $\chi_i<2q+\frac{2s}{d}$, $i=1,2$ from \eqref{derive conditions}, we obtain
\begin{equation*}
	\begin{aligned}
		\chi_1<2q+\frac{2s}{d}
		&\Longleftrightarrow
		\frac{p+\theta+1}{\theta+1}-\frac{s}{d}<q,\\[2mm]
		\chi_2<2q+\frac{2s}{d}
		&\Longleftrightarrow
		\frac{p(2q-2)}{2(p-2\beta)}<q+\frac{s}{d}\Longleftrightarrow
		q<\frac{p}{2\beta}+\frac{s(p-2\beta)}{2\beta d}.
	\end{aligned}
\end{equation*}
This implies that
\begin{equation}\label{q upper lower bounds}
			\begin{aligned}
				\frac{p+\theta+1}{\theta+1}-\frac{s}{d}<q<\frac{p}{2\beta}+\frac{s(p-2\beta)}{2\beta d}.
			\end{aligned}
\end{equation}
It follows from \eqref{lower bounds of p,q} and \eqref{q upper lower bounds} that the inequality
$\frac{s(p-2\beta)}{2p}+1<\frac{p}{2\beta}+\frac{s(p-2\beta)}{2\beta d}$
holds for all sufficiently large $p>\max\{1,\beta d\}$. Furthermore, comparing the lower and upper bounds for $q$ in \eqref{q upper lower bounds}, we obtain the following additional condition on $p$, which ensures that $1<q<\infty$ as:
\begin{align*}
				&\frac{p+\theta+1}{\theta+1}-\frac{s}{d}<\frac{(p-2\beta)s}{2\beta d}+\frac{p}{2\beta},\\
			&\Longrightarrow 	\frac{p}{\theta+1}+1-\frac{s}{d}<\frac{ps}{2\beta d}-\frac{s}{d}+\frac{p}{2\beta},\\
				&\Longrightarrow 	p\left[\frac{1}{\theta+1}-\frac{s}{2\beta d}-\frac{1}{2\beta}\right]<-1.
		\end{align*}
Since
$
s\in\left[1,\frac{d(\theta+2)}{\beta d-(\theta+2)}\right),
$
by continuity, we can choose $s$ sufficiently close to
$\frac{d(\theta+2)}{\beta d-(\theta+2)}$. Substituting this choice of $s$
into the above inequality, we obtain
\begin{align*}
	&	p\left[\frac{1}{\theta+1}
		-\frac{d(\theta+2)}{2\beta d(\beta d-(\theta+2))}
		-\frac{1}{2\beta}\right]<-1,\\ 
	&\Longrightarrow 	p\left[\frac{d(\theta+1)-2\beta d+2(\theta+2)}
		{2(\theta+1)(\beta d-(\theta+2))}\right]>1,\\ 
	&\Longrightarrow 	p>\frac{2(\theta+1)(\beta d-(\theta+2))}
		{d(\theta+1)-2\beta d+2(\theta+2)}.
\end{align*}
As $1<p<\infty$, the positivity of the resulting lower bound requires $\beta\geq\frac{\theta+2}{d}$. Furthermore, for the lower bound to be finite, we require
$d(\theta+1)-2\beta d+2(\theta+2)>0$, which is equivalent to
$\beta<\frac{d(\theta+1)+2(\theta+2)}{2d}$. We are therefore in a position to summarize the resulting lower bounds for $p$ as follows:
\begin{equation*}
			\begin{aligned}
				p>1,\ p>\beta d, \ p>\frac{(\theta+1)s}{2}-(\theta+1),\ p>\frac{2(\theta+1)(\beta d-(\theta+2))}{d(\theta+1)-2\beta d+2(\theta+2)},
			\end{aligned}
\end{equation*}
accordingly, we fix a constant $p_0$ such that
\begin{equation}\label{p0}
			\begin{aligned}
				p_0=\max\left\{1, \beta d,\ \frac{d(\theta+1)(\theta+2)}{2(\beta d-(\theta+2))}-(\theta+1),\ \frac{2(\theta+1)(\beta d-(\theta+2))}{d(\theta+1)-2\beta d+2(\theta+2)}\right\}.
			\end{aligned}
\end{equation}
It is worth noting that the third term in the definition of $p_0$ depends on the choice of $s$. Since $s<\frac{d(\theta+2)}{\beta d-(\theta+2)}$, we may choose $s$ sufficiently close to its upper bound so that $p_0$ is well defined and $p$ can be taken sufficiently large.

Consequently, for every sufficiently large $p>p_0$, there exists $q>1$ satisfying both \eqref{lower bounds of p,q} and \eqref{q upper lower bounds}. This completes the proof. 
\end{proof}
Having established Lemma~\ref{cond on p and q}, we are now in a position to derive the \(L^p\)-estimate for \(u\), thereby completing this section.

\begin{lemma}\label{final lemma for theorem 2}
	Suppose that \eqref{Initial Cond}, \eqref{Hypothesis H1}, and \eqref{Hypothesis H2} hold. In addition, let
	\[
	\beta < \frac{d(\theta+1)+2(\theta+2)}{2d}.
	\]
	Then, for some sufficiently large $1<p<\infty$, there exists a constant $C>0$ such that every classical solution of \eqref{Main} fulfills
	\begin{align}\label{second Lp estimate of u}
		\|u(\cdot,t)\|_{L^p(\Omega)} \leq C,
	\end{align}
	\textit{for all $t \in (0,T_{\max})$.}
\end{lemma}
\begin{proof}
In view of Lemma~\ref{beta k condition} and the estimates for $\|\nabla v\|_{L^m(\Omega)}$ established in \eqref{grad v 1st cond}, we distinguish two cases. We first consider the case $\beta<\frac{\theta+2}{d}$, for which we obtain
\begin{align*}
			\|\nabla v(t)\|_{L^\infty(\Omega)} < \infty, \ \text{for all } t \in (0,T_{\max}). 
\end{align*}
For the remaining case, where $\frac{\theta+2}{d}\leq\beta<\frac{d(\theta+1)+2(\theta+2)}{2d}$, the Gagliardo-Nirenberg  inequality yields 
\begin{align}\label{grad v controll}
C \int_\Omega |\nabla v|^{\chi_i}
= C\||\nabla v|^q\|_{L^{\frac{\chi_i}{q}}(\Omega)}^\frac{\chi_i}{q}
&\leq {C_1}\left(\|\nabla(|\nabla v|^q)\|_{L^2(\Omega)}^{\lambda_i}
\||\nabla v|^q\|_{L^{\frac{s}{q}}(\Omega)}^{1-\lambda_i}
+ \||\nabla v|^q\|_{L^{\frac{s}{q}}(\Omega)}\right)^{\frac{\chi_i}{q}},\nonumber\\
&\leq {C_1}\left(\|\nabla(|\nabla v|^q)\|_{L^2(\Omega)}^{\frac{\lambda_i\chi_i}{q}}
\||\nabla v|^q\|_{L^{\frac{s}{q}}(\Omega)}^{\frac{(1-\lambda_i)\chi_i}{q}}
+ \||\nabla v|^q\|_{L^{\frac{s}{q}}(\Omega)}^{\frac{\chi_i}{q}}\right).	
\end{align}
for \(i=1,2\). Subsequently, applying Young's inequality to the first term on the right-hand side of the above inequality, we set
$X:=\|\nabla(|\nabla v|^q)\|_{L^2(\Omega)}^{f_i}$
and
$Y:=C_1\||\nabla v|^q\|_{L^{\frac{s}{q}}(\Omega)}^{\frac{(1-\lambda_i)\chi_i}{q}}$.
Then, Young's inequality yields
$XY\leq\varepsilon X^{p_1}+C(\varepsilon)Y^{p_2}$,
where
$p_1=\frac{2}{f_i}>1$,
$p_2=\frac{2}{2-f_i}>1$,
and
$\varepsilon=\frac{q-1}{4q^2}$.
Since
$f_i=\frac{\lambda_i\chi_i}{q}<2$
by Lemma~\ref{cond on p and q}, both conjugate exponents $p_1$ and $p_2$ are well defined and strictly greater than one. It follows that
\begin{align}\label{grad v control_1}
	\|\nabla(|\nabla v|^q)\|_{L^2(\Omega)}^{\frac{\lambda_i\chi_i}{q}}
	\||\nabla v|^q\|_{L^{\frac{s}{q}}(\Omega)}^{\frac{(1-\lambda_i)\chi_i}{q}} \leq \frac{q-1}{4q^2}\|\nabla(|\nabla v|^q)\|_{L^2(\Omega)}^2
	+ C(\epsilon)\left({C_1}\||\nabla v|^q\|_{L^{\frac{s}{q}}(\Omega)}^{\frac{(1-\lambda_i)\chi_i}{q}}\right)^{\frac{2}{2 - f_i}},
\end{align}
where
\begin{align*}
	C(\epsilon) = \frac{2-f_i}{2}\left[\frac{4q^2 f_i}{2(q-1)}\right]^{\frac{f_i}{2-f_i}}.
\end{align*}
In addition, using Lemma \ref{beta k condition} together with the choice of \(s\) satisfying \(s\in\left[1,\frac{d(\theta+2)}{\beta d-(\theta+2)}\right)\) as specified in Lemma~\ref{cond on p and q}, we have
\begin{align}\label{grad v control_2}
	\||\nabla v|^q\|_{L^{\frac{s}{q}}(\Omega)}
	= \|\nabla v\|_{L^s(\Omega)}^q\leq C.
\end{align}
Consequently, from \eqref{grad v controll}, \eqref{grad v control_1} and \eqref{grad v control_2}, we deduce
	\begin{align}\label{chi power of grad v}
			C \int_\Omega |\nabla v|^{\chi_i} \leq \frac{q-1}{4q^2}\|\nabla(|\nabla v|^q)\|_{L^2(\Omega)}^2 + {C_2}.
	\end{align}
Under the assumption $\beta<\frac{d(\theta+1)+2(\theta+2)}{2d}$, substituting  \eqref{chi power of grad v} into \eqref{lemma for second theorem proof}, we may choose a positive constant ${C_3}$ such that
\begin{align*}
	\frac{d}{dt}\left(\frac{1}{p}\int_\Omega u^p + \frac{1}{q}\int_\Omega |\nabla v|^{2q} + \int_\Omega w^p \right)
	&+ \mu A\int_\Omega u^p + 2\int_\Omega |\nabla v|^{2q} + \frac{\delta p}{4}\int_\Omega w^p+ \frac{\mu}{2}\int_{\Omega}u^{\theta+p+1}\\
	&+\frac{q-1}{2q^2}\int_{\Omega}|\nabla|\nabla v|^q|^2\leq {C_3},
\end{align*}
for all $t \in (0,T_{\max})$. Further, we have
\begin{align}\label{derivate of u,grad v, w}
			\frac{d}{dt}\left(\frac{1}{p}\int_\Omega u^p + \frac{1}{q}\int_\Omega |\nabla v|^{2q} + \int_\Omega w^p \right)
		+ \mu A\int_\Omega u^p + 2\int_\Omega |\nabla v|^{2q} + \frac{\delta p}{4}\int_\Omega w^p
		\leq {C_3}, 
\end{align}
for all $t \in (0,T_{\max}).$ Let
	\[
	Z(t) := \frac{1}{p}\int_\Omega u^p + \frac{1}{q}\int_\Omega |\nabla v|^{2q} + \int_\Omega w^p.
	\]
Together with \eqref{derivate of u,grad v, w}, this yields
	\[
	Z'(t) + \min\left\{\mu Ap, 2q, \frac{\delta p}{4}\right\} Z(t) \leq {C_3},
	\]
for all \(t\in(0,T_{\max})\). Consequently, the estimate \eqref{second Lp estimate of u} follows from Lemma~\ref{ode comparison}. 
\end{proof}

Before proving Theorem~\ref{global bound H1-H2}, we compare the restrictions on $\beta$ arising from the preceding results. Theorem~\ref{global bound_H1} guarantees the global existence of solutions to \eqref{Main} under the condition $\beta<\frac{2(\theta+2)}{d}$, whereas Lemma~\ref{final lemma for theorem 2} requires $\beta<\frac{d(\theta+1)+2(\theta+2)}{2d}$. Comparing these two upper bounds, we observe that whenever $\theta>\frac{4-d}{d-2}$, one has $\frac{2(\theta+2)}{d}<\frac{d(\theta+1)+2(\theta+2)}{2d}$. Hence, in this parameter regime, the restriction imposed by Theorem~\ref{global bound_H1} is more restrictive than that of Lemma~\ref{final lemma for theorem 2}.

\begin{proof}[Proof of Theorem \ref{global bound H1-H2}] 
To obtain a positive lower bound for $\gamma$, we first establish an $L^\infty$-estimate for $v$, as no a priori positive lower bound for $\gamma$ is available. By Lemma \ref{final lemma for theorem 2}, we have
$
\|\nabla v\|_{L^{2q}(\Omega)}\leq C(p)
$
provided that
$
\beta<\frac{d(\theta+1)+2(\theta+2)}{2d}.
$
It therefore suffices  choose to a $q$ such that $q>\frac{d}{2}$. Indeed, once this condition is satisfied, the Sobolev embedding, together with the above estimate, yields
$
\|v\|_{L^\infty(\Omega)}\leq C(p).
$
This uniform bound for $v$ then allows us to derive the desired positive lower bound for $\gamma$. By the admissibility conditions for \(q\) in Lemma \ref{cond on p and q}, we can choose
\[
p_0=
\max\left\{
1,\,
\beta d,\,
\frac{d(\theta+1)(\theta+2)}
{2(\beta d-(\theta+2))}
-(\theta+1),\,
\frac{2(\theta+1)(\beta d-(\theta+2))}
{d(\theta+1)-2\beta d+2(\theta+2)}
\right\}.
\]
Then, for every \(p>p_0\), there exists an admissible exponent \(q\in(1,\infty)\) satisfying, in particular,
\begin{align}\label{lower bounds of q}
	q>\max\left\{
	\frac{p+\theta+1}{\theta+1}-\frac{s}{d},
	\frac{s(p-2\beta)}{2p}+1
	\right\}.
\end{align}
We shall show that, under the additional restriction on \(\beta\), these lower bounds allow us to choose \(q>\frac{d}{2}\), which is the key condition needed for the subsequent \(L^\infty\)-estimate of \(v\).
Now, firstly we want to ensure that $q>\frac d2$ and it is sufficient to require that the lower bound in \eqref{lower bounds of q} itself is larger than
 $q>\frac d2$. Thus we impose
\begin{equation*}
			\frac{p+\theta+1}{\theta+1}-\frac{s}{d}
			>\frac d2.
\end{equation*}
Since \(\theta+1>0\), this is equivalent to
\begin{align*}
			p+\theta+1
			&>
			(\theta+1)\left(\frac d2+\frac{s}{d}\right),
\end{align*}
and hence
\begin{align*}
			p>
			(\theta+1)
			\left(
			\frac d2+\frac{s}{d}-1
			\right).
\end{align*}
Therefore, define
\begin{align}\label{p* define}
		p_*=\max\left\{
				p_0,\,
				(\theta+1)
				\left(
				\frac d2+\frac{s}{d}-1
				\right)
				\right\}.
\end{align}
Thus, for every \(p>p_*\), the lower bound in \eqref{lower bounds of q} ensures that \(q>\frac{d}{2}\). It is worth noting that no additional restriction on the upper bound of \(q\) is required in deriving the second term in the definition of \(p_*\). We retain the original value of \(p_0\) in \eqref{p* define} to ensure that the admissible interval for \(q\) is nonempty and the corresponding conditions are compatible.

Next, we fix an exponent \(P>p_*\). This fixed exponent will be distinguished from the variable exponent \(p\) used later in the Moser iteration argument to obtain an \(L^\infty\)-bound for \(u\) on \((0,T_{\max})\). For this choice of \(P\), we select and fix an admissible exponent \(q=q(P)>\frac{d}{2}\). Henceforth, both \(P\) and \(q\) are fixed and independent of the variable exponent \(p\) appearing in the subsequent Moser iteration. Applying estimate \eqref{derivate of u,grad v, w} with \(p=P\), we obtain
\begin{align*}
			\frac{d}{dt}
			\left(
			\frac1P\int_\Omega u^P
			+\frac1q\int_\Omega |\nabla v|^{2q}
			+\int_\Omega w^P
			\right)
			&+\mu A\int_\Omega u^P
			\nonumber
			+2\int_\Omega |\nabla v|^{2q}
			+\frac{\delta P}{4}\int_\Omega w^P
			\leq C_3(P).
\end{align*}
Let us define 
		\begin{equation*}
			Z_P(t)
			:=
			\frac1P\int_\Omega u^P
			+\frac1q\int_\Omega |\nabla v|^{2q}
			+\int_\Omega w^P. 
		\end{equation*}
Then, we have 
\begin{equation*}
			Z_P'(t)
			+
			\min\left\{
			\mu AP,\,
			2q,\,
			\frac{\delta P}{4}
			\right\}
			Z_P(t)
			\leq C_3(P).
\end{equation*}
By using Lemma \ref{ode comparison}, we obtain
		\[
		Z_P(t)\leq C(P),
		\ \text{ for all }\ t\in(0,T_{\max}).
		\]
In particular, we infer
		\begin{align}\label{grad 2q norm}
			\int_\Omega |\nabla v(\cdot,t)|^{2q}\leq C(P), \ \text{ for all }\ t\in(0,T_{\max}).
		\end{align}
Since $q>\frac d2$, we have $2q>d$. Hence, together with the Sobolev embedding
$
W^{1,2q}(\Omega)\hookrightarrow L^\infty(\Omega)
$
gives
\begin{equation}\label{Linfty of v 2}
	\|v(\cdot,t)\|_{L^\infty(\Omega)}
	\leq C(P).
\end{equation}
Therefore, since \(\gamma>0\) and \(\gamma'<0\), together with \eqref{Linfty of v 2}, we obtain
$
\gamma(v(\cdot,t))
\geq
\gamma(C(P))
=:\gamma_*>0.
$
Here, \(\gamma_*\) may depend on the fixed exponent \(P\), but it is independent of the subsequent variable exponent \(p\) used in the Moser iteration. Furthermore, arguing similarly to Lemma \ref{Linfty bound of u} and using \eqref{grad 2q norm}, we obtain the following estimate of the same form as \eqref{4.32}:
\begin{align*}
	\int_\Omega u^p|\nabla v|^2
	&\leq
	\left(\int_\Omega |\nabla v|^{2q}\right)^{1/q}
	\left(
	\int_\Omega u^{\frac{pq}{q-1}}
	\right)^{\frac{q-1}{q}}
	\nonumber\\
	&=
	\|\nabla v\|_{L^{2q}}^2
	\left\|u^{p/2}\right\|_{L^{\frac{2q}{q-1}}}^2
	\nonumber\\
	&\leq
	C(P)^2
	\left\|u^{p/2}\right\|_{L^{\frac{2q}{q-1}}}^2.
\end{align*}
Further by using, the  Young and Gagliardo-Nirenberg inequalities, we obtain
\begin{align}\label{GN estimate 2}
	Cp^2\int_\Omega u^p|\nabla v|^2
	\leq
	\varepsilon
	\|\nabla u^{p/2}\|_{L^2}^2
	+
	C_\varepsilon
	p^{\frac{2q(d+2)}{2q-d}}
	\|u^{p/2}\|_{L^1}^2.
\end{align}
Consequently, estimate \eqref{GN estimate 2} holds for every subsequent choice of \(p>1\). 
Thus, the exponent \(P\) is fixed solely to obtain the required regularity of \(v\), whereas \(p\) remains free in all subsequent estimates involved in the Moser iteration for \(u\). It follows that the $L^\infty$-boundedness of $u$ can be established by applying an analogous argument to the one presented in Lemma \ref{Linfty bound of u}. On the other hand, Lemma \ref{final lemma for theorem 2} ensures that, provided
$
\beta < \frac{d(\theta+1)+2(\theta+2)}{2d},
$
one can find a constant $C>0$ such that
\begin{align*}
	\|u(\cdot,t)\|_{L^2(\Omega)}\leq C,
	\ \text{ for all } \ t\in(0,T_{\max}).
\end{align*}
Consequently, by applying the Alikakos-Moser iteration procedure analogously to the proof of Lemma \ref{Linfty bound of u}, we obtain a uniform \(L^\infty\)-bound for \(u\), which completes the proof of Theorem \ref{global bound H1-H2}.
\end{proof}
\section{Conclusion and future directions}
	In this work, we investigated the global existence and uniqueness of globally bounded classical solutions to system \eqref{Main}. We first established the local existence and uniqueness of classical solutions in Section \ref{sec:local} by applying the Schauder fixed point theorem together with suitable \emph{a priori} estimates, after replacing the motility function with an appropriate cutoff function. Subsequently, in Section \ref{sec:global1}, assuming only condition \eqref{Hypothesis H1} on the motility function, we proved the global existence of classical solutions under the condition $\beta < \frac{2(\theta+2)}{d}$. Furthermore, by additionally imposing the condition \eqref{Hypothesis H2}, we significantly enlarged the admissible range of the parameter $\beta$ and established global existence under the condition $\beta<\frac{d(\theta+1)+2(\theta+2)}{2d}$ whenever $\theta>\frac{4-d}{d-2}$.
	
  Although the present work establishes the global existence and uniqueness of classical solutions, several challenging questions remain open. A natural continuation of this work is to investigate the large-time behavior of solutions through the construction of an appropriate Lyapunov functional for sufficiently large values of $\mu$, following the approach developed in \cite{SIAM,JAMAA25}, which may yield further insight into the asymptotic dynamics of the system. Independently, it would also be worthwhile to examine the existence and stability of nonconstant steady-state solutions, as investigated in \cite{IMAA24} for the case $\theta=1$. Owing to the strong coupling effects and nonlinearities inherent in the system, both directions are expected to require considerably more delicate analytical arguments than those developed in the present work.
	
    \noindent	{\bf  Declarations:} 
	
	\noindent 	{\bf  Ethical Approval:}   Not applicable. 
	
	\noindent  {\bf   Competing interests: } The author declare no competing interests.

	\noindent  {\bf   Author contributions: } Om Tripathi conceived the study, developed the conceptual framework, carried out the mathematical analysis, and prepared the initial and subsequent drafts of the manuscript. Sourav K. Sasmal contributed to the literature review, the formulation of the mathematical model, and the writing and revision of the manuscript. Manil T. Mohan supervised the research, provided guidance in the development of the main ideas and mathematical analysis, and critically reviewed and revised the manuscript. All authors contributed to the interpretation of the results, reviewed the manuscript, and approved the final version for publication.


	\noindent 	{\bf   Availability of data and materials: } Not applicable. 
	
	\noindent 	{\bf   Funding declaration: }  We thank Prof. Michel Winkler, Department of Mathematics, Paderborn University, for useful discussions concerning Lemma 2.7.
	The first author would like to thank Ministry of Education of India for financial assistance. Support for S. K. Sasmal's research recieved from the IIT Roorkee, Faculty Initiation Grant (IITR/SRIC/2165/FIG-400164). Support for M. T. Mohan's research received from the National Board of Higher Mathematics (NBHM), Department of Atomic Energy, Government of India (Project No. 02011/13/2025/NBHM(R.P)/R\&D II/1137).

	\bibliographystyle{plain}
	\bibliography{References} 
\end{document}